\documentclass[11pt]{article}

\usepackage{authblk}
\usepackage[utf8]{inputenc}
\usepackage[english]{babel}
\usepackage{amsmath}
\usepackage{amsfonts}
\usepackage{amssymb}
\usepackage{amsthm}
\usepackage{color}
\usepackage{algpseudocode}
\usepackage{algorithm}

\usepackage{hyperref,etoolbox}
\usepackage{cleveref}

\usepackage{biblatex}
\usepackage{graphicx}
\usepackage{lmodern}
\usepackage[left=2cm,right=2cm,top=2cm,bottom=2cm]{geometry}

\newtheorem{theorem}{Theorem}

\newtheorem{lemma}{Lemma}

\theoremstyle{definition}
\newtheorem{definition}{Definition}

\theoremstyle{definition}

\theoremstyle{definition}
\newtheorem{remark}{Remark}

\theoremstyle{definition}
\newtheorem{assumption}{Assumption}

\DeclareMathOperator*{\argmin}{arg\,min}

\newcommand*\dist{\mathop{}\!\mathrm{dist}}
\newcommand*\diff{\mathop{}\!\mathrm{d}}
\newcommand*\Diff[1]{\mathop{}\!\mathrm{D}}

\newcommand{\ack}{\section*{Acknowledgments}}

\author[1]{Guozhi Dong}
\author[2,3]{Michael Hintermüller}
\author[3]{Clemens Sirotenko}

\affil[1]{\small School of Mathematics and Statistics, HNP-LAMA, Central South University, Changsha 410083, China}
\affil[2]{\small Institute for Mathematics, Humboldt-Universität zu Berlin, Unter den Linden 6, 10099 Berlin, Germany}
\affil[3]{\small Weierstrass Institute for Applied Analysis and Stochastics (WIAS), Mohrenstraße 39, 10117 Berlin, Germany}

\affil[ ]{\texttt{guozhi.dong@csu.edu.cn}, \texttt{hint@math.hu-berlin.de}, \texttt{sirotenko@wias-berlin.de}}

\begin{document}
\title{Dictionary Learning Based Regularization in Quantitative MRI: A Nested Alternating Optimization Framework}

\maketitle

\begin{abstract}
In this article we propose a novel regularization method for a class of nonlinear inverse problems that is inspired by an application in {\it quantitative} magnetic resonance imaging (qMRI). The latter is a special instance of a general dynamical image reconstruction technique, wherein a radio-frequency pulse sequence gives rise to a time discrete physics-based mathematical model which acts as a side constraint in our inverse problem. To enhance reconstruction quality, we employ dictionary learning as a data-adaptive regularizer, capturing complex tissue structures beyond handcrafted priors. For computing a solution of the resulting non-convex and non-smooth optimization problem, we alternate between updating the physical parameters of interest via a Levenberg-Marquardt approach and performing several iterations of a dictionary learning algorithm. This process falls under the category of nested alternating optimization schemes. We develop a general overall algorithmic framework whose convergence theory is not directly available in the literature.  Global sub-linear and local strong linear convergence in infinite dimensions under certain regularity conditions for the sub-differentials are investigated based on the Kurdyka–Łojasiewicz inequality. Eventually, numerical experiments demonstrate the practical potential and unresolved challenges of the method.
\end{abstract}

\maketitle

\section{Introduction}
Magnetic Resonance Imaging (MRI) is a well-established, non-invasive medical imaging method. From a mathematical perspective, it involves solving an ill-posed inverse problem to reconstruct an image of a specific body region based on noisy Fourier measurements, which capture the spatial frequency content of tissue magnetization at a given time. In conventional MRI, the temporal evolution of magnetization is typically neglected, and the resulting image mainly reflects contrast information which depends on scanner settings and sampling patterns. In contrast, recent developments in quantitative MRI (qMRI) focus on estimating biophysical parameters by utilizing the time-dependent behavior of the magnetization \cite{dong2019quantitative,brown2014magnetic,ma2013magnetic,shafieizargar2023systematic}. A prominent technique in this domain is Magnetic Resonance Fingerprinting (MRF) \cite{ma2013magnetic}, which enables the simultaneous estimation of multiple tissue properties by using a specially designed acquisition sequence that generates unique signal evolutions for different tissue types. These measured signal evolutions are then compared against a precomputed database of signal trajectories, each generated by numerically solving the Bloch equations under a wide range of parameter combinations. By identifying the closest match in this database, MRF infers the underlying tissue parameters. Despite its success, MRF has two main drawbacks: solving a highly undersampled linear inverse problem first to obtain an estimate of the evolutions and second, relying on a precomputed discretized database, which can cause errors. To address this, recent works \cite{dong2019quantitative,scholand2023quantitative,zimmermann2024pinqi} propose model-based reconstructions using the Bloch equation \cite{bloch1946nuclear} directly, resulting in a nonlinear inverse problem formulation
\begin{equation}
    f^\delta = F(u_{true}) + \eta,  \qquad u_{true} \in U_{ad}, \label{eq:nonlinear_inverse_pro_into}
\end{equation}
for an operator $F:X \to Y$, function spaces $X,Y$ and a pre defined set $U_{ad} \subset X$. We follow this second model based approach here, where the aim is to recover the physical parameters \( u_{true} \in U_{ad}\) from noisy measurements \( f^\delta \). The noise \( \eta \) has typically zero mean and amplitude \( \delta \geq 0 \). We detail this nonlinear inverse problem, including notation and key properties, in \autoref{sec:nonlinear_inverse_pro}. Due to the strong undersampling necessary in MRF, regularization is essential to suppress artifacts and noise that arise during image reconstruction. Classical variational methods address this challenge by formulating and solving typically nonsmooth optimization problems, such as
\begin{equation}
    \min_{u \in U_{ad}} \frac{1}{2} \| F(u) - f^\delta \|_{Y}^2 + \lambda R(u).\label{eq:variational_approach_intro}
\end{equation}
These concepts remain fundamental and provide the basis for many state-of-the-art techniques. Typically, \( R:X \to \overline{\mathbb{R}}:= \mathbb{R} \cup\{+\infty \} \) is a convex, often nonsmooth regularization term encoding prior knowledge about the true solution \( u_{true} \). The parameter \( \lambda > 0 \) balances the regularizer \( R(u) \) and the \emph{data fidelity term} \( \|F(u) - f^\delta\|^2_{Y} \). Common examples include Total Variation (TV) and sparsity-based regularization. For TV, \( R(u) = \| Du \|_{\mathcal{M}} \), the Radon norm of the distributional gradient for \( u \in BV(\Omega) \); see \cite{rudin1992nonlinear, chambolle1997image, attouch2014variational}. In sparse regularization \cite{mallat1999wavelet, scherzer2009variational}, the solution is assumed sparse in a basis \( (\varphi_n)_{n \in \mathbb{N}} \subset H \), a Hilbert space, leading to a corresponding regularizer
\begin{equation}
    R(u) = \sum_{n \in \mathbb{N}} |\langle u ,\varphi_n \rangle_{H} |. \notag
\end{equation}
Methods of this kind have been extensively studied; see \cite{scherzer2009variational,kaltenbacher2008iterative,ito2014inverse,engl1996regularization} for classical references. However, they are often too general and show limitations in specific tasks. In quantitative imaging, for instance, the aim is to recover accurate parameter values, not just contrast. Classical regularizers can introduce systematic biases—e.g., total variation tends to compress values, underestimating magnitudes. Another drawback is their reliance on simple a-priori assumptions, such as sparsity in wavelet or Fourier domains, which may fail to capture the intricate structures present in real-world data. These handcrafted bases often overlook complex spatial or temporal correlations inherent in biomedical images. To address this limitation, data-driven regularization methods have emerged, that leverage training data to model more realistic image priors and capture richer patterns in the data \cite{arridge2019solving, ravishankar2019image}. These often outperform classical approaches like \eqref{eq:variational_approach_intro}, but sacrifice interpretability and robustness. For instance neural network-based methods, which represent the state of the art in many imaging applications, often function as black boxes and can exhibit undesired sensitivity to small perturbations in $f^\delta$ or changes in algorithmic parameters \cite{genzel2022solving,antun2020instabilities}. In this work, we regularize the solution of \eqref{eq:nonlinear_inverse_pro_into} using a blind dictionary learning approach, aiming to balance interpretability and the use of training data. Dictionary learning has been successfully applied in imaging, particularly in qualitative MRI \cite{ravishankar2010mr, ravishankar2012learning,ravishankar2015efficient}. In \autoref{sec:regularization_by_dict}, we detail our regularization strategy. The method operates on patches extracted from the image \( u \in X = H_0^1(\Omega)\), solving the inverse problem while decomposing each patch into a product of a shared dictionary \( D \in \mathcal{D} \) and sparse coefficients \( C \in \mathcal{C} \), with \( \mathcal{D}, \mathcal{C} \) being predefined matrix sets. This ultimately leads to a non-convex, non-smooth optimization problem of the type
\begin{align}
   \min_{(u,D,C) \in U_{ad} \times \mathcal{D} \times \mathcal{C}}J(u,D,C) = \frac{1}{2} \| F(u) - f^\delta \|_{Y}^2 &+ \frac{\alpha}{2} \| \nabla u \|_{L^2(\Omega)}^2  \notag \\
   &+  \lambda \left( \frac{1}{2}\| P\mathbb{D}^\mathfrak{h} u - DC \|_F^2 + \beta \|C\|_1 \right). \label{eq:P0_intro}
\end{align}
Here, \( P \) extracts patches from the discretized image \( \mathbb{D}^\mathfrak{h} u \) with mesh-size-parameter $\mathfrak{h}>0$, vectorizes them, and assembles them into a large matrix. Additionally, \( \lambda, \alpha, \beta > 0 \) represent regularization parameters (see \autoref{sec:regularization_by_dict} for details on the notation). Problems like \eqref{eq:P0_intro} are typically solved using alternating optimization techniques, following the general pattern outlined below in \autoref{fig:alt_optimization}:\\[-0.3cm]

\begin{figure}[h!]
\centering
\noindent\begin{minipage}{0.95\textwidth}
\hrule
\vspace{0.4em}
\begin{itemize}
    \item \textbf{Initialize:} \( u_0, D_0, C_0 \in U_{ad} \times \mathcal{D} \times \mathcal{C} \)
    \item \textbf{Repeat for } \( k = 0, 1, 2, \ldots \)
    \begin{itemize}
        \item[1.] Update \( u_k \) through a (possibly complex) reconstruction step.
        \item[2.] Update the dictionary \( D_k \) (computationally cheap).
        \item[3.] Update the sparse coefficients \( C_k \) (computationally cheap).
    \end{itemize}
\end{itemize}
\vspace{0.3em}
\hrule
\end{minipage}
\caption{Sketch of the alternating optimization scheme for solving \eqref{eq:P0_intro}.}
\end{figure}\label{fig:alt_optimization}
\vspace{-0.2cm}
\noindent
The approach outlined in \autoref{fig:alt_optimization} may be suboptimal. To see why, consider the problem
\begin{equation}
        \min_{D \in \mathcal{D}, C \in \mathcal{C}} \frac{1}{2} \|DC  - P \mathbb{D}^\mathfrak{h} u \|_F^2 + \beta \| C \|_1, \label{eq:dict_intro}
\end{equation}
which is typically approached through alternating updates of $D$ and $C$ until convergence to a stationary point. The iterative procedure can be interpreted as progressively denoising each patch in $P \mathbb{D}^\mathfrak{h} u$, with the final reconstruction represented by $D\cdot C$, where $D$ and $C$ solve \eqref{eq:dict_intro}. In this regard, dictionary learning alone acts as a powerful denoising tool, effectively capturing structure and suppressing noise.  However, if only a single update of $D$ and $C$ is performed per iteration, the denoising effect may be significantly weakened. A natural enhancement, therefore, is to carry out multiple updates of $D$ and $C$ within each iteration, allowing the representation to better adapt before updating the physical parameter $u$. This observation leads to the idea of nested optimization, where the dictionary learning step is solved more thoroughly within an inner loop, approaching stationarity. The resulting nested scheme is sketched in \autoref{fig:nested_optimization}.
\begin{figure}[h!]
\centering
\noindent\begin{minipage}{0.95\textwidth}
\hrule
\vspace{0.4em}
\begin{itemize}
    \item \textbf{Initialize:} \( u_0, D_0, C_0 \in U_{ad} \times \mathcal{D} \times \mathcal{C} \)
    \item \textbf{Repeat for } \( k = 0, 1, 2, \ldots \)
    \begin{itemize}
        \item[1.] Update \( u_k \) via a (possibly complex) reconstruction step.
        \item[2.] Repeat until some stationarity measure is sufficiently small:
        \begin{itemize}
            \item[2.1.] Update the dictionary \( D_k \) (computationally cheap).
            \item[2.2.] Update the sparse coefficients \( C_k \) (computationally cheap).
        \end{itemize}
    \end{itemize}
\end{itemize}
\vspace{0.3em}
\hrule
\end{minipage}
\caption{Outline of the nested alternating optimization scheme.}
\end{figure}\label{fig:nested_optimization}
\noindent
In \autoref{sec:nested_optimization}, we detail the nested optimization algorithm and analyze its convergence toward stationary points of the objective introduced in \autoref{sec:regularization_by_dict}. We provide a global convergence analysis with classical sublinear rates and study local strong convergence under the Kurdyka–Łojasiewicz (KL) framework; see \cite{bolte2010characterizations, attouch2010proximal, attouch2013convergence} for details on this fundamental concept. The one-step method fits into the convergence framework of \cite{frankel2015splitting}, extending \cite{attouch2013convergence}. In contrast, the nested scheme proposed in this work is, to the best of our knowledge, not covered by existing literature—primarily because it does not satisfy the sufficient decrease property commonly required in convergence analyses. 
Let us now summarize the main contributions of this work below.
\paragraph{Contributions}
\begin{enumerate}
    \item We propose a model-based regularization framework for qMRI that directly integrates the solution operator of the Bloch equation into the image reconstruction process. In contrast to the original two-step approach in \cite{ma2013magnetic}, this formulation leads to a single, nonlinear inverse problem that jointly accounts for the underlying physics and the reconstruction task. While model-based strategies are not new (e.g., \cite{dong2019quantitative, scholand2023quantitative}), their combination with dictionary learning in this context—along with the infinite-dimensional, resolution-independent formulation—is novel. Our method employs orthogonal dictionary learning to regularize the nonlinear inverse problem introduced in \autoref{sec:nonlinear_inverse_pro}, aiming to mitigate contrast bias while preserving data adaptivity and interpretability. To the best of our knowledge, the most closely related work is \cite{kofler2023quantitative}, which applies dictionary learning to a modified qMRI model focused on reconstructing \( R_1 := 1/T_1 \) and \( m_0 \). However, their approach does not address convergence theory and operates on pre-discretized images.
    \item Our framework is flexible and, in principle, allows the dictionary learning step to be replaced by more advanced denoising operators—such as those proposed for linear problems in \cite{ravishankar2019image}—opening avenues for broader applications beyond the specific orthogonal setting considered here.
    \item We tackle the resulting nonconvex and nonsmooth optimization problem using a nested optimization scheme inspired by \cite{gur2023convergent, gur2023nested}, but under different assumptions and with the u-variable defined in an infinite-dimensional space. Our setting falls outside the framework of \cite{gur2023nested}, as it lacks the partial strong convexity assumed therein, and we do not impose the Kurdyka–Łojasiewicz (KL) property on the full objective. While our nested update scheme is more general, the resulting convergence guarantees are correspondingly weaker—we do not establish global strong convergence. It is also worth noting the extensive literature on block alternating optimization strategies (e.g., \cite{attouch2013convergence, bolte2014proximal, phan2023inertial}), which typically rely on the KL inequality in finite-dimensional spaces and often assume that each block update solves a global nonconvex subproblem. These conditions do not apply in our framework.
\end{enumerate}

\section{Dictionary learning based regularization for quantitative MRI}\label{sec:nonlinear_inverse_pro}
In this section, we briefly outline the fundamental ideas behind MRI signal generation, following primarily the presentation in \cite{dong2019quantitative}. For detailed discussions, we refer to standard textbooks \cite{brown2014magnetic, liang2000principles, wright1997magnetic} and related research \cite{davies2014compressed, dong2019quantitative}. 
\subsection{Brief introduction to the physics of qMRI}
In MRI, the domain $\Omega \subset \mathbb{R}^d, d\leq 3$ represents the anatomical region of interest selected by the clinician for imaging. The observed image contrast within $\Omega$ is a result of the spatiotemporal dynamics of magnetic moments associated with hydrogen nuclei in the tissue. These dynamics are governed by the underlying biophysical properties of the tissue at each spatial location $x \in \Omega$, with each voxel capturing the aggregate behavior of the magnetic moments within its volume.  These magnetic moments evolve under the influence of an externally controlled magnetic field $B(t,x)$, where $t \in [0,T]$ denotes time and $T>0$ is the final time horizon. The time evolution of the magnetic moment located at \( x \in \Omega \), represented by the function \( m: [0,T] \times \Omega \to \mathbb{R}^3 \), is modeled as a solution of the Bloch equations
\begin{equation}
\partial_t m(x,t) = m(t,x) \times \gamma B(t,x) - \begin{pmatrix}
m_1(t,x)/T_2(x) \\
m_2(t,x)/T_2(x) \\
(m_3(t,x) - m_{eq} )/T_1(x) 
\end{pmatrix}, 
\quad m(0,x) = \begin{pmatrix}
0 \\ 
0 \\
m_0(x)\label{eq:continuous_bloch}
\end{pmatrix}. 
\end{equation} 
Here \( m_{\mathrm{eq}} \in \mathbb{R} \) is the equilibrium magnetization along the \( z \)-axis, \( m_0(x) \) denotes the initial longitudinal magnetization and $\gamma>0$ denotes the gyromagnetic ratio, a dimensionless physical constant. While fundamental to the physics of MRI, it does not play a central role in the mathematical analysis presented here. The quantities \(T_1(x)\) and \(T_2(x)\) are spatially varying relaxation times that characterize the recovery of the magnetic moment toward equilibrium \(m_{\mathrm{eq}} \in \mathbb{R}^3\) following excitation. These parameters are of fundamental importance, as they are tissue-dependent and enable the identification of underlying biomedical properties. The goal of quantitative MRI (qMRI), and of this article in particular, is to accurately reconstruct these relaxation times. The external magnetic field driving the evolution described in \eqref{eq:continuous_bloch} typically consists of three components:
\begin{equation}
B(t,x) = B_0(x) + B_1(t,x) + (0,0,G(t)\cdot x)^\top,
\end{equation}
where $B_0(x)$ is the static (main) magnetic field, constant in time and aligned with the $z$-axis; $B_1(t,x)$ is the time-dependent radio-frequency (RF) field used to perturb the magnetization from equilibrium, typically applied briefly as an \emph{RF pulse}; and $G(t)$ is the gradient field, which encodes spatial position into frequency for spatial localization in the signal.
The time between two consecutive RF pulses is called repetition time (TR). Under suitable assumptions one can show that the signal, that is measured at a receiver coil in a very short time after an RF-pulse is applied and turned off immediately, can be approximately described by
\begin{equation}
f(t)  \approx \int_{\Omega} \rho(x)  m_{12}(t,x) e^{-i \int_0^t G(\tau) \cdot x \diff \tau} \diff x. \notag
\end{equation}
Here we introduce the notation $m_{12}(t,x) := m_1(t,x) + i m_2(t,x) \in \mathbb{C}$ and a further physical quantity, the proton-spin density $\rho :\Omega \to \mathbb{R}$ that is commonly interpreted as the local density of (predominately) hydrogen protons or “spins” located at $x \in \Omega$.
Consequently, the measurement process can be modeled by a composition of the nonlinear solution operator of the Bloch-equation and the Fourier-transform:
\begin{equation}
f(t) \approx \mathcal{F}\left[\rho(\cdot)\, m_{12}(t, \cdot)\right]\left(k(t)\right) = \mathcal{F}[y_t](k(t)) \quad \text{where}\quad k(t) = \frac{\gamma}{2\pi} \int_0^t G(\tau)\, \diff \tau.
\notag
\end{equation}
Here we set $y_{t} = \rho(\cdot) m_{12}(t, \cdot)$ and employ the continuous Fourier transform $\mathcal{F}$. At the echo time $t = TE$, measurements are taken, and by manipulating the gradient field $t \mapsto G(t)$, different frequencies can be sampled. In basic Cartesian sampling, the frequencies in two dimensions are collected along parallel lines, with one full line being sampled between two consecutive pulses; see \cite{brown2014magnetic} for a comprehensive overview of MRI physics and its practical applications, also see  \autoref{fig:linear_MRI_example}.
\noindent
\paragraph{ Classical vs. Quantitative MRI}
The goal of classical MRI is to recover the \emph{ground truth image}  
\begin{equation}
y_{TE} = \rho(\cdot)\, m_{12}(TE, \cdot),
\end{equation}  
accessible only through its (noisy and undersampled) Fourier transform, where TE refers to the echo time. Due to the need for magnetization \( m(\cdot, x) \) to relax toward equilibrium after each RF pulse, only a subset of \( k \)-space can be sampled. This makes the reconstruction problem \emph{ill-posed}, as the available data is insufficient to uniquely recover the image without additional assumptions. Under the assumption that \( y_{TE} \) remains constant during acquisition, the groundtruth image can be estimated by inverting the subsampled and noisy Fourier transform of the measured noisy data $f$, modeled as  
\begin{equation}
    f = S \circ \mathcal{F}[y_{TE}] + \eta, \label{eq:linear_MRI_IP}
\end{equation}
where \( S \) is a linear \emph{undersampling operator} that selects the measured subset of Fourier coefficients, and \( \eta \) models measurement noise. A representative solution to this \emph{linear inverse problem} in \eqref{eq:linear_MRI_IP} is shown in \autoref{fig:linear_MRI_example} using the famous Shepp-Logan phantom \cite{shepp1974fourier}. A detailed treatment of linear inverse problems in the context of MRI can be found in \cite{lustig2007sparse, knoll2011second, starck2015sparse, adcock2021compressive}.
\begin{figure}[b!]
\centering
\begin{minipage}[b]{0.23\textwidth}
        \centering
        \includegraphics[width=\textwidth]{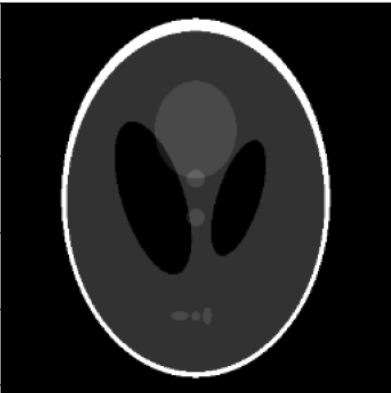}
        \footnotesize (a)
\end{minipage}
\begin{minipage}[b]{0.23\textwidth}
        \centering
        \includegraphics[width=\textwidth]{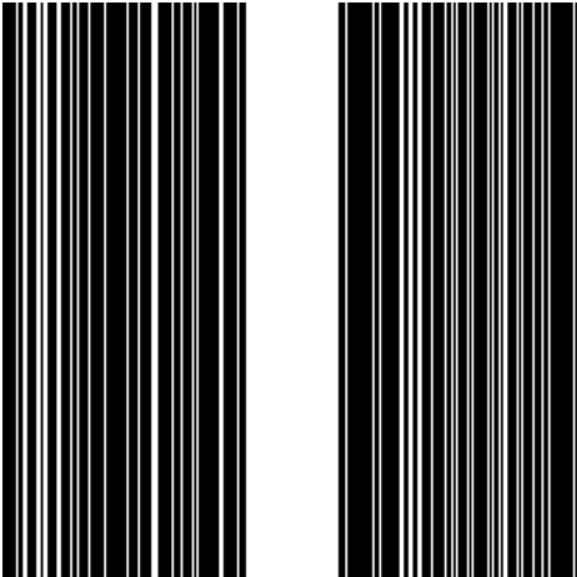}
        \footnotesize(b)
\end{minipage}
\begin{minipage}[b]{0.23\textwidth}
        \centering
        \includegraphics[width=\textwidth]{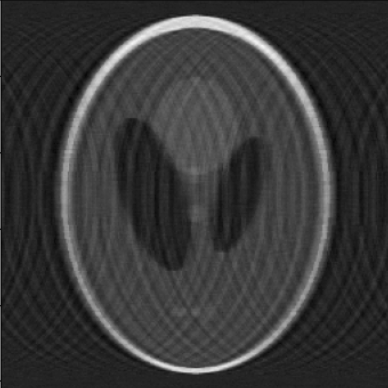}
        \footnotesize(c)
\end{minipage}   
\begin{minipage}[b]{0.23\textwidth}
        \centering
        \includegraphics[width=\textwidth]{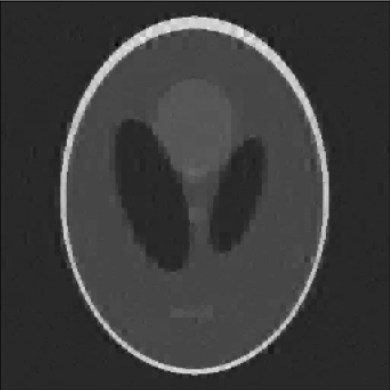}
        \footnotesize(d)
\end{minipage}   
\caption{Typical linear MRI setup: (a) Ground truth, (b) Sampling mask (white = sampled frequencies), (c) least squares reconstruction, (d) basic wavelet regularization, as in \cite{starck2015sparse}.}
\end{figure}\label{fig:linear_MRI_example}
Sampling a single image \( y_{TE} \) provides only contrast information and does not enable direct estimation of the underlying tissue parameters \( T_1 \), \( T_2 \), and \( \rho \). In contrast, quantitative MRI (qMRI) aims to recover these spatially varying parameters by acquiring Fourier measurements of multiple images sequentially over time. By sampling at multiple time points, the evolution of the magnetization trajectory \( m \) can be captured, enabling the reconstruction of \( m \) and consequently the estimation of the underlying parameters \( T_1 \), \( T_2 \), and \( \rho \).
\paragraph{Magnetic Resonance Fingerprinting (MRF)}
MRF \cite{ma2013magnetic} is a recent qMRI technique. It accelerates data acquisition by collecting a sequence of undersampled Fourier signals through rapid RF pulses using a specially designed excitation sequence. Denote the sequence of corresponding ground truth images by
\begin{equation}
y_{t_k} = \rho(\cdot)\, m_{12}(t_{k}, \cdot), \quad k = 1, \ldots, L,
\end{equation}
indexed by time points \( t_1 < \ldots < t_L \), with \( L \) typically ranging between 100 and 1000 in practice. In MRF, highly undersampled Fourier measurements in the spirit of \eqref{eq:linear_MRI_IP} are collected at each time point:
\begin{equation}
f_{t_k} = S_{t_k} \circ \mathcal{F}(y_{t_k}) + \eta_{t_k}, \label{eq:MRF_linear_IPs}
\end{equation}
where \( S_{t_k} \) denotes the sampling operator at time \( t_k \), and \( \eta_{t_k} \) models measurement noise. Then a two step procedure is applied:
\begin{enumerate}
    \item First, estimated ground truth images \( \widehat{y}_{t_k} \) are reconstructed from the undersampled data \( f_{t_k} \), typically using simple linear methods such as zero-filled inverse Fourier transform or compressed sensing techniques, such as those in \cite{starck2015sparse}.
    \item Then, in a second step, for each voxel, the temporal signal evolution \( (\widehat{y}_{t_k})_{k=1}^L \) is compared to a precomputed database of Bloch equation simulations, generated for various combinations of \( (T_1, T_2, \rho) \). The best match determines the estimated tissue parameters \( (\widehat{\rho}, \widehat{T}_1, \widehat{T}_2) \).
\end{enumerate}
In \cite{ma2013magnetic}, the authors employed an Inversion Recovery Steady-State Free Precession (IR-SSFP) sequence as an approximate model for the Bloch equation. We denote by
\begin{equation}
m_k(\mathbb{T}(x)) := m(t_k, x) \in \mathbb{R}^3,
\end{equation}
where \( m(t,x) \) denotes the solution of the continuous Bloch equations \eqref{eq:continuous_bloch} that implicitly depends on the tissue dependent relaxation times $\mathbb{T}=(T_1,T_2)(x)$ at position \( x \in \Omega \). Under the IR-SSFP sequence and suitable modeling assumptions, the magnetization dynamics (the solution of \eqref{eq:continuous_bloch}) can be approximated by a discrete-time dynamical system (see \cite{ma2013magnetic, brown2014magnetic}):
\begin{equation}
m_{k+1}(\mathbb{T}(x)) = E_k(\mathbb{T}(x)) R(\alpha_k) m_k(\mathbb{T}(x)) + b_k(\mathbb{T}(x)),  \quad m_0(x) \in \mathbb{R}^3.\label{eq:discrete_bloch_system}
\end{equation}
The matrix $R(\alpha)$ is an orthogonal rotation matrix depending on the \emph{flip angle} $\alpha_k \in (0,2\pi)$ and $E_k:\mathbb{R}^2 \to \mathbb{R}^{3 \times 3}, b_k: \mathbb{R} \to \mathbb{R}^{3}$ are given by
\begin{equation}
E_k(\mathbb{T}) = 
\begin{pmatrix}
    \exp(-\frac{TR_k}{T_2}) &0 &0 \\
    0 &\exp(-\frac{TR_k}{T_2})  &0 \\ 
    0 &0 &\exp(-\frac{TR_k}{T_1})  
\end{pmatrix},
\quad b_k(\mathbb{T}) = 
\left[ 1 - \exp\left(-\frac{TR_k}{T_1} \right)\right] 
\begin{pmatrix}
0 \\
0 \\
1
\end{pmatrix} ,\notag
\end{equation}
for $\mathbb{T} = (T_1,T_2) \in \mathbb{R}^2_{>0}$ and repetition times $\mathrm{TR}_1, \ldots, \mathrm{TR}_L > 0$, which are specified by the scanning protocol. In this way a mapping $m: \mathbb{R}^2 \to \mathbb{R}^{3 \times L}$ is defined,
\begin{equation}
    m(\mathbb{T}) :=
    \begin{bmatrix}
        m_{1}(\mathbb{T})\\
        \vdots \\
        m_{L}(\mathbb{T})
    \end{bmatrix}, \label{eq:def_m}
\end{equation}
that maps the relaxation times $\mathbb{T}=(T_1,T_2)$ to an approximation of the (time discrete) Bloch equation. We note that, with a slight abuse of notation, we refer to both the continuous solution and the time-discrete solution operator as \( m \), relying on context to distinguish between them. In most of the following sections, \( m \) refers to the time-discrete solution operator as defined in \eqref{eq:def_m}.
\begin{remark} \label{remark:extension_E_matrices}
    Note that $E_k: \mathbb{R}_{>0}^2 \to \mathbb{R}^{3 \times 3}$ and $b_k: \mathbb{R}_{>0}^2 \to \mathbb{R}^3$ are only well-defined for arguments $T_1,T_2 >0$. However the extension to $\mathbb{R}^2$ is obvious. We define
    \begin{equation}
        \Tilde{E}_k(\mathbb{T}) =
       \lim_{n \to \infty }E_k(\mathbb{T}_n) 
       \quad 
        \Tilde{b}_k(\mathbb{T}) = \lim_{n \to \infty}  b_k(\mathbb{T}_n), \notag
    \end{equation}
    where $\mathbb{T}_n \in (0,+\infty)^2$ for every $n\in \mathbb{N}$ and  $\mathbb{T}_n \to \mathrm{proj}_{[0,+\infty)^2}(\mathbb{T})$ as $n \to \infty$. 
    One can easily check that this extension is well defined and does not depend on the choice of the sequence $(\mathbb{T}_n)_n$.
    Additionally, in this way we extended $E_k$ and $b_k$ to $C^\infty$-functions that are defined everywhere on $\mathbb{R}^2$. We will from now on assume that $E_k$ and $b_k$ are defined everywhere and are smooth.
\end{remark}
\noindent
The two-step procedure of MRF has the drawback that it heavily depends on the accuracy of the precomputed Bloch simulation database. A preselected discretization of the parameter space can introduce significant errors in the final estimation and may even amplify errors originating from the initial reconstruction step used to obtain \( (\widehat{y}_{t_k})_{k=1}^L \). To overcome these limitations, we adopt an all-at-once, model-based approach, where the solution mapping of the Bloch equations is directly integrated into the forward model as proposed in \cite{dong2019quantitative}. This results in a nonlinear inverse problem, as described in the following subsection.
\noindent
\subsection{The inverse problem of qMRI and its mathematical properties} 
Throughout the following part, we use the notation  
\( u = (\rho, \mathbb{T})= (\rho, T_1, T_2)  \in L^2(\Omega, \mathbb{R}^3) \).
to summarize the parameters of interest. We now formulate the inverse problem of recovering these parameters from measurement data, closely following \cite{dong2019quantitative}. Using the discrete Bloch system \eqref{eq:discrete_bloch_system}, the full data generation process—from the true parameters \( u_{\text{true}} \in L^2(\Omega, \mathbb{R}^3) \) to the noisy measurements \( f^\delta \in L^2(\Omega, \mathbb{C}^L) \)—is modeled by the \emph{forward equation}:
\begin{equation}
    f^\delta = F(u_{\text{true}}) + \eta, \quad \|\eta\|_{L^2(\Omega, \mathbb{C}^L)} \leq \delta, \label{eq:def_datageneration_process}
\end{equation}
where \( \eta \) represents complex noise of level \( \delta \geq 0 \). The \emph{forward operator} \( F: L^2(\Omega, \mathbb{R}^3) \to L^2(\Omega, \mathbb{C}^L) \) is defined by
\begin{equation}
    F(u) := \left[ S_1 \mathcal{F}(\Pi(u)_1), \ldots, S_L \mathcal{F}(\Pi(u)_L) \right], \label{eq:definition_forward_operator}
\end{equation}
where again \( \mathcal{F} \) denotes the spatial Fourier transform and \( S_k \) are linear sampling operators selecting measured frequencies at each time step as in \autoref{fig:linear_MRI_example} and \eqref{eq:MRF_linear_IPs}. The \emph{Bloch solution operator} \( \Pi: L^2(\Omega, \mathbb{R}^3) \to L^2(\Omega, \mathbb{C}^L) \) maps the tissue parameters to the time-discrete transverse magnetization sequence:
\begin{equation}
    \Pi(u)(x) := \rho(x) 
    \begin{bmatrix}
        m_{1}(\mathbb{T}(x))_{12} \\
        \vdots \\
        m_{L}(\mathbb{T}(x))_{12}
    \end{bmatrix}, \label{eq:def_Pi}
\end{equation}
where $m(\cdot)$ denotes the mapping defined in \eqref{eq:def_m} as the solution operator of the time-discrete dynamical system in \eqref{eq:discrete_bloch_system}, and the subscript again indicates that we consider the first two components of this vector, interpreted as a complex-valued quantity in the transverse plane. Thus, the inverse problem of qMRI is to estimate \( u \approx u_{\text{true}} \) from noisy data \( f^\delta \).
\begin{remark}[Representation as a superposition operator]
The operator {\(\Pi: L^2(\Omega,\mathbb{R}^3) \to L^2(\Omega,\mathbb{C}^L)\)} defined in \eqref{eq:def_Pi} can be represented as a \emph{superposition operator}, also known as a Nemytskii operator. Specifically, we may define the function
\begin{equation}
\pi: \mathbb{R}^3 \to \mathbb{C}^L, \quad \pi(u) := \rho \cdot 
\begin{bmatrix}
m_1(\mathbb{T})_{12} \\
\vdots \\
m_L(\mathbb{T})_{12}\label{eq:def_pi}
\end{bmatrix},
\end{equation}
where now $u=(\rho, \mathbb{T}) \in \mathbb{R}^3$ and observe that
\begin{equation}
\Pi(u)(x) = \pi(u(x)) \quad \text{for a.e. } x \in \Omega. \notag
\end{equation}
\noindent In this formulation the map \(\pi(\cdot)\) acts pointwise on each voxel or pixel value {\(u(x) = (\rho(x), T_1(x), T_2(x)) \in \mathbb{R}^3\)}. The Bloch solution operator \(\Pi\) is then obtained by applying \(\pi\) to every spatial point \( u(x) \in \Omega\).
Such superposition operators are well-studied in nonlinear functional analysis; see \cite{goldberg1992nemytskij, appell1990nonlinear}.
\end{remark}
\noindent
Many properties of the function \( \pi \) extend to the associated superposition operator \( \Pi \). We therefore begin by collecting differentiability and stability properties of \( \pi \) in the following theorem. As usual, \( \mathcal{L}^k(X, Y) \) denotes the space of bounded \( k \)-linear operators between Banach spaces \( X \) and \( Y \), and \( f'(x) \in \mathcal{L}(X, Y) \), \( f''(x) \in \mathcal{L}^2(X, Y) \) denote the first and second Fréchet derivatives of a function \( f: X \to Y \) at \( x \in X \). We also use the norm \( \|f\|_{C(X,Y)} = \sup_{x \in X} \|f(x)\|_Y \) for continuous, bounded functions \( f: X \to Y \).
\begin{theorem}[Properties of the pointwise solution map]\label{theorem_properties_pi_map}
Let $\pi : \mathbb{R}^{3} \to \mathbb{C}^L$ and $m:\mathbb{R}^2 \to \mathbb{R}^{3 \times L}$ be defined as in \eqref{eq:def_pi} and \eqref{eq:def_m}. Then the following statements hold:
\begin{itemize}
\item[(i)] (Differentiability) The mappings $\pi: \mathbb{R}^3 \to \mathbb{C}^L$ and $m:\mathbb{R}^2 \to \mathbb{R}^{3\times L}$ infinitely differentiable. 
\item[(ii)](Boundedness) There is a $C = C(L,m_0)>0$ such that the following bounds hold true:
\begin{equation}
    \|m \|_{C(\mathbb{R}^2,\mathbb{R}^{3 \times L})} \leq C \quad \| m' \|_{C(\mathbb{R}^2,\mathcal{L}(\mathbb{R}^2,\mathbb{R}^{3 \times L})} \leq C \quad
    \| m'' \|_{C(\mathbb{R}^2,\mathcal{L}^2(\mathbb{R}^2,\mathbb{R}^{3 \times L})} \leq C. \notag
\end{equation}
\item[(iii)] (Lipschitz-properties) Denote \( u_i = (\rho_i, \mathbb{T}_i) \) for $i=1, 2$. If $|\rho_i| \leq b$ for $i=1,2$ and some real number $b>0$, then we find a constant $L = L(b)>0$ that depends on $b$, such that 
\begin{align*}
    \| \pi(u_1) - \pi(u_2) \|_2 
    &\leq L \| u_1 - u_2\|_2, \\
    \|  \pi'(u_1) - \pi'(u_2) \|_{\mathcal{L}(\mathbb{R}^3, \mathbb{C}^{L})} 
    &\leq L \| u_1 - u_2\|_2.
\end{align*}
\end{itemize}
\end{theorem} 
\begin{proof}
    As the proof uses only basic calculus we postpone it to the appendix, \autoref{Appendix_B_proofs_sec_2}.
\end{proof}
\noindent
The subsequent theorem establishes properties of the corresponding superposition operator in \eqref{eq:def_Pi} using the theory, which is developed in \cite{goldberg1992nemytskij}.
\begin{theorem}[Properties of the solution operator]\label{theorem_properties_solution_operator}
Consider the map $\pi: \mathbb{R}^3 \to \mathbb{C}^L$ for bounded $\Omega \subset \mathbb{R}^d$ with Lipschitz boundary. Then the corresponding superposition operator, or the solution operator of the time discrete Bloch equation, introduced in \eqref{eq:def_Pi}, satisfies the following properties:
\begin{itemize}
        \item[(i)] The operator $\Pi$ is a mapping from $L^p(\Omega,\mathbb{R}^3)$ to $L^q(\Omega,\mathbb{C}^L)$ for exponents $1 \leq p\leq q < \infty$. 
        \item[(ii)](Frechet-differentiability) The operator $\Pi: L^p(\Omega,\mathbb{R}^3) \to L^2(\Omega,\mathbb{C}^L)$  is Frechet differentiable for $p\geq 4$. The Frechet-derivative is given by
        \begin{equation}
            \Pi' :L^p(\Omega) \to \mathcal{L}(L^p(\Omega), L^2(\Omega)) \quad  D\Pi(u)[h](x) = \pi'(u(x))h(x) \quad \text{for $p \geq4$.} \notag
        \end{equation}
        \item[(iii)](2nd order Frechet-differentiability) The operator $\Pi$ is two times Frechet differentiable as a mapping $\Pi: L^p(\Omega,\mathbb{R}^3) \to L^2(\Omega,\mathbb{C}^L)$ for $p\geq 6$. The 2nd order Frechet-derivative is given by
        \begin{equation}
            \Pi'' :L^p(\Omega) \to \mathcal{L}(L^p(\Omega), L^2(\Omega)) \quad  \Pi''(u)[h_1,h_2](x) =  \pi''(u(x))h_1(x) h_2(x) \quad \text{for $p \geq 6$.} \notag
        \end{equation}
\end{itemize}
\end{theorem}
\begin{proof}
We employ the theory on abstract superposition operators developed in \cite{goldberg1992nemytskij}. The detailed proof is found again in the appendix, \autoref{Appendix_B_proofs_sec_2}.
 \end{proof}
 
\subsection{Regularization by dictionary learning}\label{sec:regularization_by_dict}
The core idea of dictionary learning is to learn a representation system directly from data, rather than using predefined bases like wavelets or radial bases. However, learning a  basis for an entire image is computationally intensive and prone to overfitting, so a common strategy is to work with small image patches. This reduces complexity while capturing local structure.
Before incorporating dictionary learning into the nonlinear inverse problem  \eqref{eq:def_datageneration_process}, we briefly review its role in regularizing linear inverse problems. The seminal work \cite{aharon2006k} introduced a two-step approach. First, a data matrix \(X = [p_1, \ldots, p_N]\), containing clean vectorized patches \(p_i \in \mathbb{R}^K\), is factorized by solving 
\begin{equation}
    \min_{D \in \mathcal{D},\, C \in \mathcal{C}} \frac{1}{2} \|DC - X\|_F^2 + \|C\|_0, \label{eq:matrix_factorization}
\end{equation}
where \(D = [\varphi_1, \ldots, \varphi_M] \in \mathbb{R}^{K \times M}\) is the dictionary of basis elements or atoms $\varphi_j$, and \(C\) holds sparse coefficients. The dictionary is often constrained (e.g., normalized columns) to avoid scaling ambiguity. The key assumption is that each patch admits an unknown sparse representation:
\begin{equation}
    p_i = \sum_{j=1}^M \varphi_j c_{ij}, \quad \text{with } \|c_{ij}\|_0 \ll M. \label{eq:assumption_dictionary_1}
\end{equation}
In the second step, the learned dictionary $D$ is used to regularize an inverse problem: Given measurements $f^\delta = F(u_{\text{true}}) + \eta$, the goal is to recover $u_{\text{true}}$ by solving the following minimization problem, while for the time being, setting aside concerns about mathematical rigor and well-posedness:
\begin{equation}
    \min_{u \in U_{ad},\, C \in \mathcal{C}} \| F(u) - f^\delta \|_{L^2(\Omega)}^2 + \frac{\lambda}{2} \left( \| P \mathbb{D}^\mathfrak{h} u - DC \|_F^2 + \beta \|C \|_0 \right), \label{eq:dict_learning_fixed}
\end{equation}
where \(\mathbb{D}^\mathfrak{h}\) discretizes \(u\), \(P\) extracts patches and $U_{ad}$ is some predefined set, see \autoref{sec:final_formulation_IP} for details.  This approach has seen wide use in linear inverse problem cf.  \cite{arridge2019solving} for an overview. 
\subsection{Final formulation of the nonlinear dictionary learning problem}\label{sec:final_formulation_IP}
While the two stage method, where a dictionary is learned before from training data, in \eqref{eq:dict_learning_fixed} already outperforms classical regularization techniques, more recent work \cite{ravishankar2012learning, ravishankar2010mr} has shown that results can even improve when the dictionary is learned \emph{on the fly}, i.e., jointly with image reconstruction, even without the use of training data. This approach—known as \emph{online dictionary learning} or \emph{blind compressed sensing}—formulates both tasks as a single optimization problem:
\begin{equation}
    \min_{u\in U_{ad}, D \in \mathcal{D}, C \in \mathcal{C}} \| F(u) - f^\delta \|_{L^2(\Omega)}^2 + \frac{\lambda}{2} \left( \| P \mathbb{
    D}^\mathfrak{h} u - DC \|_F^2 + \beta \|C \|_0 \right).\label{eq:dict_learning_variable} 
\end{equation}
Subsequent developments \cite{ravishankar2015efficient, ravishankar2013sparsifying, ravishankar2019image} have applied this framework, particularly in the context of \emph{classical} MRI. These data-adaptive methods have proven to be both interpretable and robust, while also mitigating common artifacts such as contrast bias—often introduced by classical regularization schemes like \eqref{eq:variational_approach_intro}. To ensure the well-posedness of the joint optimization problem, a small strongly convex regularization term \(\|\nabla\, \cdot\|_{L^2(\Omega)}^2\) is typically added. The final formulation of our dictionary-regularized qMRI reconstruction task is given by:
\begin{align}
\min_{(u,D,C)\in U_{ad} \times \mathcal{D} \times \mathbb{R}^{K \times M} }
\frac{1}{2} \| F(u) - f^\delta \|^2_{L^2(\Omega,\mathbb{C}^L)} + \frac{\alpha}{2} \|\nabla u \|^2_{L^2(\Omega)} +  \frac{\lambda} {2} \left( \| P \mathbb{D}^\mathfrak{h} u - D C \|_F^2 +  \beta \|C\|_1  \right), \tag{$P_0$}
\label{eq:p0}
\end{align}
for positive regularization parameters $\lambda,\alpha,\beta>0$. Here, we also replaced the seminorm \( \|\cdot\|_0 \) with its classical convex envelope \( \|\cdot\|_1 \). Let us now specify the components of the optimization problem \eqref{eq:p0}. For simplicity, we assume the spatial domain to be \(\Omega = (0,1)^2\), and define the problem setup as follows:
\begin{enumerate}
    \item \textbf{Parameter space:} The unknown parameter \(u = (\rho, T_1, T_2)\) is modeled as a vector-valued function in the Sobolev space \(H_0^1(\Omega, \mathbb{R}^3)\).
    \item \textbf{Admissible set:} The set of feasible parameters is given by
    \begin{equation}
    U_{\mathrm{ad}} = \left\{ u \in H_0^1(\Omega, \mathbb{R}^3) \,\middle|\, a \leq u_i(x) \leq b \quad \text{for a.e. } x \in \Omega,\ i = 1,2,3 \right\},
    \end{equation}
    where \(a, b \in \mathbb{R}\) are user-defined bounds ensuring physical plausibility.
    \item \textbf{Discretization operator:} For a mesh size \(\mathfrak{h} = 1/N\), \(N \in \mathbb{N}\), the discretization space is
    \begin{equation}
    U^N := \mathrm{span} \left\{ N^{-2} \mathbf{1}_{\Omega_{i,j}} : i,j = 1,\ldots,N \right\},
    \end{equation}
    where each set \(\Omega_{i,j} \subset \Omega\) is defined as \(\Omega_{i,j} = (i-1,j-1) + [0,1/N] \times [0,1/N]\). The projection \(\mathbb{D}^{\mathfrak{h}} : L^2(\Omega) \to U^N\) acts by averaging:
    \begin{equation}
    (\mathbb{D}^{\mathfrak{h}} u)|_{\Omega_{i,j}} = \frac{1}{N^2} \int_{\Omega_{i,j}} u(x)\, \mathrm{d}x. \label{eq:def_discretization_op}
    \end{equation}
    \item \textbf{Patch extraction operator:}\label{def_patch_extract} The operator \(P\) extracts (possibly overlapping) vectorized image patches from the discretized image \(\mathbb{D}^{\mathfrak{h}} u\). It maps from the space of discretized images to a matrix space of  \(X \in \mathbb{R}^{K \times N^2}\), where each column corresponds to a patch of dimension \(K\). The exact implementation of \(P\) may vary, but it is assumed to be linear and fixed throughout the reconstruction process.

    \item \textbf{Forward operator:} The forward map \(F: H_0^1(\Omega, \mathbb{R}^3) \to L^2(\Omega, \mathbb{C}^L)\) is defined by \eqref{eq:definition_forward_operator} and models the signal generation process via the Bloch dynamics and Fourier sampling.

    \item \textbf{Dictionary and coefficient spaces:} The dictionary is constrained to the Stiefel manifold
    \begin{equation}
    \mathcal{D} = O_K := \left\{ M \in \mathbb{R}^{K \times K} \,\middle|\, M^\top M = I \right\},
    \end{equation}
    enforcing orthonormal atoms. The coefficient matrix lies in \(\mathcal{C} = \mathbb{R}^{K \times N^2}\), without additional constraints.
\end{enumerate}
\noindent
Throughout this work, we make frequent use of classical Sobolev and Lebesgue spaces. For background on these function spaces and their properties, we refer to \cite{evans2022partial}.
\begin{remark}[Choice of the dictionary space]
We choose orthogonal matrices as the dictionary ansatz for their simple and efficient updates, while still ensuring good reconstruction quality. More advanced models can be incorporated, as proposed in \cite{ravishankar2019image} for linear inverse problems.\
\end{remark}
\section{A nested Levenberg-Marquardt type optimization algorithm} \label{sec:nested_optimization}
The goal of this section is to derive an optimization algorithm to find stationary points of the objective function \eqref{eq:p0} introduced earlier. To simplify the presentation, we focus on reconstructing a single function \( u \in H_0^1(\Omega) \) instead of working in \( H_0^1(\Omega, \mathbb{R}^3) \). Without loss of generality, we restrict the analysis to the case of a single measured image, i.e., $F(u) \in L^2(\Omega)$ instead of $F(u) \in L^2(\Omega, \mathbb{C}^L)$. The convergence analysis remains unchanged. As usual, we equip \( H_0^1(\Omega) \) with the inner product 
\( \langle u, v \rangle_{H_0^{1}(\Omega)} := \langle \nabla u, \nabla v \rangle_{L^2(\Omega)} \), 
which is equivalent to the standard \( H^1(\Omega) \) inner product by the Poincar\'e inequality; see \cite{attouch2014variational}. 
We remark that the entire convergence analysis can equivalently be carried out under the \( H^1(\Omega) \) norm. In fact, in the numerical experiments presented later, we will employ a discrete version of a weighted \( H^1 \)-norm. Let us now state the problem of the section: 
For regularization parameters \( \alpha, \lambda > 0 \), we consider the following class of optimization problems:
\begin{equation}
    \min_{u \in H_0^1(\Omega),\, z \in Z} J(u,z) := \frac{1}{2} \| F(u) - f^\delta \|^2_{L^2(\Omega)} + \frac{\alpha}{2} \| \nabla u \|^2_{L^2(\Omega)}+ \mathcal{I}_{U_{ad}}(u)  + h(u,z) + R(z),
    \tag{$P_1$} \label{eq:p0_general}
\end{equation}
where \( Z \) is a finite-dimensional Hilbert space. We also introduce the notation
\begin{equation}
    f^\alpha(u) := \frac{1}{2} \| F(u) - f^\delta \|_{L^2(\Omega)}^2 + \frac{\alpha}{2} \| \nabla u \|^2_{L^2(\Omega)}
    \label{eq:def_f}
\end{equation}
and we will also write $f := f^0$ at some points for simplicity. The following assumptions will be used in the analysis of \eqref{eq:p0_general}.
\begin{assumption}[Assumptions on problem class \eqref{eq:p0_general}]\label{assumptions_general}We impose the following assumptions for the subsequent analysis:
\begin{itemize}
    \item[(B1)] We assume $\Omega \subset \mathbb{R}^d$ bounded with Lipschitz boundary. The forward operator \( F: L^p(\Omega) \to L^2(\Omega) \) is twice continuously Fr\'echet differentiable for some $p>2$, such that $H_0^1(\Omega) \hookrightarrow L^p(\Omega)$ compactly.
    \item[(B2)] The derivates and function values are assumed to be bounded on $U_{ad}$, i.e. there are constants $M_i>0, i=1,2,3$ such that
    \begin{equation}
    \|F(u)\|_{L^2(\Omega)} \leq M_1 \quad 
        \|F'(u)\|_{\mathcal{L}} \leq M_2 \quad \|F''(u)\|_{\mathcal{L}} \leq M_3 \quad \text{for all $u \in U_{ad}$} \label{eq:bounded_derivatives_ass}
    \end{equation}
    \item[(B3)] The admissible set is
    \begin{equation}
    U_{ad} = \{ u \in H_0^1(\Omega) \mid a \leq u(x) \leq b \text{ a.e. in } \Omega \},
    \end{equation}
    for fixed \( 0 < a < b \).
    
    \item[(B3)] Let \( Z \) be a finite-dimensional Hilbert space. The function \( h: L^2(\Omega) \times Z \to \mathbb{R} \) satisfies:
    \begin{itemize}
        \item[(H1)] \( h \) is proper, bounded below, and twice continuously differentiable;
        \item[(H2)] For every \( z \in Z \), the mapping \( u \mapsto h(u,z) \) is convex;
        \item[(H3)] The gradient \( u \mapsto \nabla_z h(u,z) \) is Lipschitz continuous, uniformly in \( z \), i.e.,
        \begin{equation}
        \| \nabla_z h(u_1,z) - \nabla_z h(u_2,z) \|_Z \leq C \| u_1 - u_2 \|_{H_0^1(\Omega)} \quad \text{for all $z \in Z, u_1,u_2 \in U_{ad} $}.
        \end{equation}
    \end{itemize}
    \item[(B5)] The function \( R: Z \to \overline{\mathbb{R}} \) is proper, lower semicontinuous, and coercive (i.e., \( \|z\|_Z \to \infty \) implies \( R(z) \to \infty \)), but may be nonconvex.
\end{itemize}
\end{assumption}
\noindent
\begin{remark}
    It is clear, that for $\Omega \subset \mathbb{R}^d, d \leq 2$, the dictionary learning regularized problem in \eqref{eq:p0} is a specific case of the problem class presented in  \eqref{eq:p0_general} for the choice $h(u,z) = (1/2)\| P \mathbb{D}^\mathfrak{h} u - DC\|_F^2 $, $z = (D,C) \in \mathbb{R}^{K \times K} \times \mathbb{R}^{ K \times N}$ and $R(D,C) = \mathcal{I}_{O_K}(D) + \beta \|C\|_1$. The restriction on the space dimension is needed to ensure $H_0^1(\Omega) \hookrightarrow L^p(\Omega)$ compactly, cf. also  \autoref{theorem_properties_solution_operator}.
\end{remark}
\noindent
 Let us briefly investigate existence of solutions. 
\begin{theorem}[Existence of solutions for \eqref{eq:p0_general}]\label{theorem_existence_general}Let \autoref{assumptions_general} hold true. Then Problem \eqref{eq:p0_general} has a solution $(u^*,z^*) \in U_{ad} \times Z$.
\end{theorem}
\begin{proof}
We use the direct method in the calculus of variations. Let $(u_n, z_n)_n$ be an infimizing sequence, which exists since $J$ is bounded below. By the definition of $J$, the sequence $(\|\nabla u_n\|_{L^2(\Omega)})_n$ is bounded, and thus, by the Poincaré inequality \cite{attouch2014variational}, so is $(\|u_n\|_{H_0^1(\Omega)})_n$. The coercivity of $R$ implies boundedness of $(\|z_n\|_Z)_n$. Therefore, a subsequence (relabelled) converges weakly in $H^1(\Omega)$ and strongly in $L^p(\Omega)$ to some $(u^*, z^*)$ in $H^1_0(\Omega) \times Z$, due to the compact embedding $H_0^1(\Omega) \hookrightarrow L^p(\Omega)$ with the $p$ from \autoref{assumptions_general}. Since $U_{ad}$ is sequentially closed in the strong $L^p$-topology and $F: L^p(\Omega) \to L^2(\Omega)$ is continuous, classical lower semicontinuity arguments yield that $(u^*, z^*)$ is a global solution of \eqref{eq:p0_general}.
\end{proof}
\subsection{The algorithm and its global convergence analysis} Let us now consider the optimization scheme for finding stationary points of \eqref{eq:p0_general}. We assume the reader is familiar with basic concepts from variational analysis and stationarity. For a comprehensive treatment, we refer to the monographs \cite{mordukhovich2006variational,rockafellar2009variational}. For our purposes, the presentation in \cite{frankel2015splitting} will be sufficient.
\subsubsection{Description of the algorithm} As outlined in the introduction, we solve \eqref{eq:p0_general} by alternating between updates for the $z$-variable (\emph{$z$-step}) and the $u$-variable (\emph{$u$-step}). Both steps are detailed below.
\paragraph{Solution for the z-step}
Given the iterate $u_k \in H_0^1(\Omega)$, the goal of the $z$-step is to solve
\begin{equation}
\min_{z \in Z} g_k(z) := h(u_k,z) + R(z) \tag{$P_z^k$} \label{eq:P_z}
\end{equation}
up to limiting stationarity, cf. \cite{frankel2015splitting}. To this end, we assume access to an abstract algorithm specified by a sequence of transition maps $\mathcal{A}_k^n: Z \to Z$, defining the next iterate given the current one:
\begin{equation}
z_k^{n+1} = \mathcal{A}_k^n(z_k^n), \quad z_k^0 := z_k \in Z \notag
\end{equation}
for all $k \in \mathbb{N}$. We assume this algorithm generates a so-called \emph{descent sequence}, defined next.
\begin{definition}[Descent sequence]\label{def:descent_sequence} A descent sequence with parameters $(\sigma_1,\sigma_2) \in \mathbb{R}^2_{>0}$ for problem \eqref{eq:P_z} is a sequence of points $(z^n)_{n \in \mathbb{N}} \subset Z$ such that the following two inequalites hold true:
\begin{itemize}
    \item[(i)] (Sufficient descent of function values) There is a $\sigma_1 >0$  such that
        \begin{equation}
            g_k(z^{n+1}) \leq g_k(z^n) - \frac{\sigma_1}{2}\| z^{n+1} - z^n\|^2_Z  \quad \text{for all $n \in \mathbb{N}$}. \label{eq:descent_z}
        \end{equation}
    \item[(ii)](Gradient inequality) There is a $\sigma_2>0$, such that 
    \begin{equation}
        \dist(0, \partial g_k(z^{n+1})) \leq \sigma_2 \| z^{n+1} - z^{n}\|_Z \quad \text{for all $n \in \mathbb{N}$}.
        \label{eq:gradient_inequality_z}
    \end{equation}
\end{itemize}
\end{definition}
\noindent
\begin{remark}
The dictionary learning  \cref{alg:algorithm_dictionary} used to solve the subproblems generates descent sequences, as shown in \autoref{lem:properties_dict_learning_alg}. This property is common among many algorithms in the literature \cite{attouch2010proximal,attouch2013convergence,bolte2014proximal,bolte2016majorization}. In contrast, accelerated methods like FISTA exhibit non-monotone convergence and do not guarantee descent \cite{beck2017first}.
\end{remark}
\begin{assumption}[On the update algorithm in the $z$-step]\label{assumption_z_step}
We assume that the algorithm, defined by the update rule $\mathcal{A}_k^n: Z \to Z$, yields a descent sequence for the objective \eqref{eq:P_z} with parameters $(\sigma_1,\sigma_2) \in \mathbb{R}^2_{>0}$ that can be chosen independently of $k \in \mathbb{N}$.   
\end{assumption}
\noindent
Let us briefly formalize the algorithm for the $z$-step.
\begin{algorithm}[H]
\caption{Computation of a near stationary point for the $z$-step at the $k$-th outer loop iteration.}
\begin{algorithmic}[1]
\State Get $(u_k, z_k) \in H_{0}^1(\Omega) \times Z$, accuracy $\eta_k > 0$.
\State  Initialize with $z_k^0 = z_k$ and compute $z_k^1 = \mathcal{A}_k^0(z_k^0)$.
\State Set $n = 1$.
\While{$\|z_{k}^{n} - z_k^{n-1} \| \geq \eta_k $} 
	\State  Compute $z_k^{n+1} = \mathcal{A}_n(z_k^n) \in Z$. 
    \State Set $n = n+1$.
\EndWhile 
\State Set $z_{k+1} := z_k^{n_k}$ where $n_k \in \mathbf{N}$ is the first iterate such that 
$\|z_{k}^{n_k} - z_{k}^{n_k-1} \| < \eta_k$.  
\State Return $z_{k+1}=z_k^{n_k}$.
\end{algorithmic}
\end{algorithm}\label{alg:algorithm_z}
\noindent
  From the properties \eqref{eq:descent_z} and \eqref{eq:gradient_inequality_z} we directly infer the following lemma, which describes the fundamental behaviour of every descent sequence.
\begin{lemma} Let $(u_k, z_k) \in H_{0}^1(\Omega) \times Z$, and an accuracy $\eta_k > 0$ be given in \cref{alg:algorithm_z}. For the generated sequence the following properties hold:
\begin{itemize}
    \item[(i)] The sequence of function values converges monotonically to its infimum, i.e.
    \begin{equation}
        g_k(z_k^n) \searrow g^*_k := \inf_{n \in \mathbb{N}} g_k(z_k^n) \geq 0 \quad \text{as $n \to \infty$. }  \notag
    \end{equation}
    \item[(ii)] The following estimate holds true for every $k,N \in \mathbf{N}$
    \begin{equation}
        \sum_{n=1}^N \|z_k^{n} - z_k^{n-1} \|^2_Z \leq \frac{ 2(g_k(z_k^0) - g_k(z_k^{N}))}{\sigma_1} \leq \frac{ 2(g_k(z_k^0) - g_k^*)}{\sigma_1}. \notag
    \end{equation}
    In particular $\|z_k^{n} - z_k^{n-1} \|_Z \to 0$ and $\dist(0, \partial g_k(z_k^{n})) \to 0$ as $n \to \infty$ . 
    \item[(iii)] The following estimate holds for the lazy slope and every $k,N \in \mathbf{N}$
    \begin{equation}
        \min_{n=1, \ldots,N} \dist(0, \partial g_k(z_{k}^{n})) \leq \sqrt{ \frac{\sigma_2^2}{2\sigma_1} \left( \frac{g_k(z_k^0)-g_k(z_k^{N})}{N}\right)}. \notag
    \end{equation}
\end{itemize}
\end{lemma}
\begin{proof}
A proof can for instance be found in \cite{bolte2014proximal}.
\end{proof}
\noindent
\begin{remark}[Complexity of \cref{alg:algorithm_z}]\label{remark_complexity_z_algorithm}
Based on the above inequalities, we may also derive that
\begin{equation}
    \min_{n=1,\ldots,N}  \|z_k^{n} - z_k^{n-1} \|^2_Z \leq  \frac{ 2(g(z_k^0)-g(z_k^N))}{ \sigma_1 N}. \notag
\end{equation}
Consequently, if a tolerance $\eta_k>0$ is given, then this bound implies that 
$\|z_k^{n} - z_k^{n-1} \|_Z \leq \eta_k$ after at most $N(\eta_k)$-many steps, where
\begin{equation}
N(\eta_k) = \frac{2(g(z_k^0)-g(z_k^{n_k}))}{\sigma_1 \eta_k^2}. \notag
\end{equation} 
For example, for $\eta_k = k^{-\gamma}$ with some given $\gamma>0$ the stopping index $n_k$ satisfies 
\begin{equation}
n_k \leq \frac{2(g_k(z_k)-g_k(z_{k+1}))}{\sigma_1} k^{2\gamma} \leq \frac{2(J(u_k, z_k)-J(u_{k+1},z_{k+1}))}{\sigma_1} k^{2\gamma}.
\end{equation}
\end{remark}
\paragraph{Solution for the $u$-step.}
We update the physical parameter of interest, $u \in H_0^1(\Omega)$, by performing exactly one Levenberg-Marquardt step, i.e., for a given iterate \( (u_k, z_{k+1}) \in H_0^1(\Omega) \times Z \), we consider the  optimization problem:
\begin{equation}
    u_{k+1} =  \argmin_{u \in U_{ad}} g_{\lambda_k}(u,u_k) + \frac{\alpha}{2} \| \nabla u \|_{L^2(\Omega)}^2  +  h(u,z). \label{eq:P_u}
    \tag{$P_u$}
\end{equation}
Here, $g_{\lambda_k}(u,v)$ is the following \emph{model function} that approximates $f$ locally around $v$:
\begin{align}
    g_{\lambda_k}(u,v) 
    &:= 
    g(u,v) + \frac{\lambda_k}{2} \|u - v \|^2_{H_0^1(\Omega
    )} \notag \\
    &:= \frac{1}{2} \|F'(u_k)[u - u_k] + F(u_k) - f^\delta  \|^2_{L^2(\Omega)} 
    + \frac{\lambda_k}{2} \|u - v \|^2_{H_0^1(\Omega
    )}. \label{eq:def_model_function_u}
\end{align}  
For later use, we further define the \emph{approximation error} $$e_{\lambda_k}(u,v) := g_{\lambda_k}(u,v) - f(u).$$ Let us first ensure that the problem \eqref{eq:P_u} has a solution.
\begin{theorem}[Existence of solutions for the subproblems] Given $z_{k+1} \in Z$,  $u_k \in H_0^1(\Omega)$ and $\lambda_k>0$, the problem \eqref{eq:P_u} has a unique solution.
\end{theorem}
\begin{proof}
The proof follows similar arguments as the proof of \autoref{theorem_existence_general}. The uniqueness is a consequence of the $\lambda_k$-strong convexity of $g_{\lambda_k}(\cdot,u_k)$ for every $k \in \mathbb{N}$.
\end{proof}
\noindent
Next we establish some classical estimates for the model function defined in \eqref{eq:def_model_function_u}.
\begin{theorem}[Fundamental inequalities for the model function]\label{fundamental_inequality}
Let  \autoref{assumptions_general} hold true. Then there are $L_1, L_2>0$ such that the model function $g(\cdot,\cdot)$ as defined in \eqref{eq:def_model_function_u} satisfies
\begin{align}
    \| \nabla f(u) - \nabla f(v) \|_{H^{-1}(\Omega)} &\leq L_{1} \| u - v \|_{H_0^1(\Omega)}, \label{eq:fund_inequality_grad_ineq}\\
    |g(u,v) - f(u)| &\leq \frac{L_{2}}{2} \|u - v \|_{H_0^1(\Omega)}^2, \label{eq:fund_inequality_approx_inequa}
 \end{align}
 for every pair $u,v \in U_{ad}$. 
\end{theorem}
\begin{proof}
    We prove \eqref{eq:fund_inequality_grad_ineq} first. For this purpose let $u,v \in U_{ad}$. Chainrule yields 
    \begin{equation}
    \nabla f(u) = F'(u)^*(F(u) - f^\delta) \in H^{-1}(\Omega). \notag
    \end{equation}
    By definition, we obtain for $h \in H_0^1(\Omega)$ and $u,v \in U_{ad}$
    \begin{align}
        | \langle \nabla f(u) - \nabla f(v), h \rangle_{H^{-1}(\Omega)}| 
        &=
        |\langle F(u) - f^\delta, F'(u)[h] \rangle_{L^2(\Omega)} - \langle F(v) - f^\delta, F'(v)[h] \rangle_{L^2(\Omega)} | \notag\\
        &\leq (\| F(u)\|_{L^2(\Omega)} + \|f^\delta \|_{L^2(\Omega)})\| (F'(u) - F'(v))[h]\|_{L^2(\Omega)}  \notag\\
        &\quad + \| F(u) - F(v) \|_{L^2(\Omega)} \| F'(u)[h] \|_{L^2(\Omega)} . \label{eq:fund_ineq_1}
    \end{align}
    By \eqref{eq:bounded_derivatives_ass} we see that $\| F(u)\|_{L^2(\Omega)} + \|f^\delta \|_{L^2(\Omega)} \leq C_1$ for some $C_1>0$ independent of $u$. Moreover, again by \eqref{eq:bounded_derivatives_ass} we obtain Lipschitz continuity of $F:L^p(\Omega) \to L^2(\Omega)$ and $F': L^p(\Omega) \to L^2(\Omega)$ using classical mean-value type theorems. Hence we deduce 
    \begin{equation}
        \| F(u) - F(v) \|_{L^2(\Omega)} \leq  C_2 \|u - v \|_{H^1_0(\Omega)} \quad \text{and} \quad 
        \| F'(u) - F'(v)[h] \|_{L^2(\Omega)} 
        \leq C_2
          \| u - v \|_{H_0^1(\Omega)} \|h \|_{H_0^1(\Omega)},
        \notag
    \end{equation}
    for some generic $C_2>0$, which also involves  embedding constants.  Again by \eqref{eq:bounded_derivatives_ass}, we bound $\| F'(u)[h] \|_{L^2(\Omega)} 
    \leq M_1 \|h\|_{H_0^1(\Omega)}$.  Plugging these inequalities into \eqref{eq:fund_ineq_1}, we directly infer \eqref{eq:fund_inequality_grad_ineq}. \\[0.2cm]
    Let us now prove \eqref{eq:fund_inequality_approx_inequa}. For this purpose, we let $w_1,w_2 \in L^2(\Omega)$ be arbitrary and observe that 
    \begin{align*}
        \left| \frac{1}{2} \|w_1 -f^\delta \|^2_{L^2(\Omega)}- \frac{1}{2} \|w_2 - f^\delta \|^2_{L^2(\Omega)} \right| \leq 
        \left| \langle w_2- f^\delta , w_1 - w_2\rangle_{L^2(\Omega)} \right| + \frac{1}{2} \|w_1 - w_2 \|^2_{L^2(\Omega)} .
    \end{align*}
    Taking $w_2 = F(v) \in L^2(\Omega)$ and $w_1 = F'(u)[v-u] + F(u) \in L^2(\Omega)$ we deduce 
    \begin{align*}
        \left| \| F(u) +  F'(u)[v-u] - f^\delta \|_{L^2(\Omega)}^2 - \|F(v)-f^\delta\|_{L^2(\Omega)}^2 \right|
        &\leq 
        2\|F(v)-f^\delta\|_{L^2(\Omega)} B + B^2 ,
    \end{align*}
    with $B := \|F(v) - F(u) -  F'(u)[v-u] \|_{L^2(\Omega)}$. As in the first part, we use \eqref{eq:bounded_derivatives_ass} and mean value theorem to obtain
    $B \leq C_3 \| u - v\|^2_{L^p(\Omega)}$ for some $C_3>0$ and the $p>2$ from \autoref{assumptions_general}. Consequently we infer  
\begin{align}
  \left| \| F(u) +  F'(u)[v-u] -f^\delta \|_{L^2(\Omega)}^2 - \|F(v)-f^\delta\|_{L^2(\Omega)} \right| \notag
  &\leq
   2C_3 \|F(v)-f^\delta\|_{L^2(\Omega)}  \|v - u\|_{L^p(\Omega)}^2 \\
  & \quad + C_3^2 \| v - u\|_{L^p(\Omega)}^4 .\label{eq:theorem_fundamental_inequality_ineq_proof}
\end{align}
Note that again $\|u-v\|_{L^p(\Omega)} \leq C_5 (\|v\|_{L^\infty(\Omega)}+\|u\|_{L^\infty(\Omega)}) \leq 2C_5b $ for some constant $C_5>0$ and that $\| F(v) - f^\delta\|_{L^2(\Omega)} \leq C_6$ for $v \in U_{ad}$ by invoking \eqref{eq:bounded_derivatives_ass}. Combining these observations with the continuous embedding $H_0^1(\Omega) \hookrightarrow L^p(\Omega)$ in \eqref{eq:theorem_fundamental_inequality_ineq_proof}, we obtain a constant $C>0$ with
\begin{align*}
  \left| \| F(u) + F'(u)[v-u] -f^\delta \|_{L^2(\Omega)}^2 - \|F(v)-f^\delta\|_{L^2(\Omega)} \right|
  &\leq
  C \|v - u\|_{H_0^1(\Omega)}^2.
\end{align*}
Thus the proof is finished.
\end{proof}
\noindent
From \autoref{fundamental_inequality} we directly obtain the following majorization property:
\begin{equation}
    f(u) \leq g(u,u_k) +\frac{L_{2 }}{2} \|u - u_k \|^2_{H_0^1(\Omega)} \quad \text{for all $u \in U_{ad}$}. \notag
\end{equation}
Given a positive descent constant $C_{\mathrm{desc}}>0$, we may chose step-size $\lambda_k \geq L_2 + C_{\mathrm{desc}}$ and deduce
\begin{align}
    J(u_{k+1},z_{k+1}) 
    &\leq 
    J(u_k,z_{k+1}) + \frac{L_2- \lambda_k}{2} \| u_k - u_{k+1} \|_{H_0^1(\Omega)}^2 \notag 
    \leq 
    J(u_k,z_{k+1}) - \frac{C_\mathrm{desc}}{2} \| u_k - u_{k+1} \|_{H_0^1(\Omega)}^2,
    \label{eq:descent_u_remark}
\end{align}
where we also used the minimization property of $u_{k+1} \in U_{ad}$.
\paragraph{The overall algorithm.} Let us now describe  the overall algorithm to approach stationary points of Problem \eqref{eq:p0_general}.
\begin{algorithm}[H]
\caption{Computation of a stationary point of \eqref{eq:p0_general}.}
\begin{algorithmic}[1]
\State Choose initial values $(u_0, z_0) \in H_{0}^1(\Omega) \times Z$, stopping tolerance $\varepsilon_1, \varepsilon_2 > 0$, accuracies of the nested subroutine $(\eta_k)_{k \in \mathbb{N}}$ and descent parameter $C_{\mathrm{decs}}>0$. 
\State Set k = 0.
\While{$\| u_{k} - u_{k-1} \|_{H_0^1(\Omega)} \geq  \varepsilon_1 $ or $\| z_{k} - z_{k-1} \|_{Z} \geq \varepsilon_2$}
	\State \textbf{z-step}: Given $(u_k,z_k) \in H_0^1(\Omega) \times Z$ and accuracy $\eta_k>0$, compute $z_{k+1}\in Z$ 
    \State using the abstract descent \cref{alg:algorithm_z}. 
    \State \textbf{u-step}: Compute a global solution $\widehat{u}(\lambda_k) \in U_{ad}$ of 
    \begin{equation}
        \min_{u \in U_{ad}} g_{\lambda_k}(u,u_k) + \frac{\alpha}{2} \| \nabla u \|_{L^2(\Omega)}^2 +  h(u,z), \label{eq:gen_algorithm_p_u}
    \end{equation}
    \State that satisfies the sufficient descent condition
    \begin{equation}
        J(\widehat{u}(\lambda_k),z_{k+1}) \leq J(u_k,z_{k+1}) - \frac{C_{\mathrm{desc}}}{2}\|\widehat{u}(\lambda_k) - u_k\|_{H_0^1(\Omega)}^2. \label{eq:alg_u_descent}
    \end{equation}
    \State Set $u_{k+1} = \widehat{u}(\lambda_k)$. 
    \State Set $k \leftarrow k+1$.
\EndWhile 
\State Return $(u_{k}, z_{k})$.
\end{algorithmic}
\end{algorithm}\label{alg:algorithm_general}
\noindent
The algorithm accepts the next iterate $\hat{u}(\lambda_k) \in U_{ad}$ in the $u$-step \eqref{eq:gen_algorithm_p_u} only if the descent condition \eqref{eq:alg_u_descent} is satisfied. If $L_2 > 0$ from \autoref{fundamental_inequality} is known beforehand, one can choose $\lambda_k \geq L_2 + C_{\mathrm{desc}}$ to guarantee this condition. However, since this constant is rarely known in practice, we rely on an adaptive backtracking strategy to determine a suitable $\lambda_k$, as described in the following algorithm.
\begin{algorithm}[H]
\caption{Backtracking search for \cref{alg:algorithm_general}}
\label{alg:backtracking1}
\begin{algorithmic}[1]
\State Get $\sigma_{BT} \in (0,1), \tau > 1$ and  $\lambda_0>0$ such that $C_{\mathrm{desc}} = \sigma_{BT}\lambda_0$.
\State Set $\lambda=\lambda_0$ and solve \eqref{eq:gen_algorithm_p_u} with parameter $\lambda>0$ to obtain $u(\lambda)$. 
\While{ $J(u(\lambda),z_{k+1}) > J(u_k,z_{k+1})
 - 	\frac{\lambda \sigma_{BT} }{2} \| u(\lambda) - u_k \|_{H_0^1(\Omega)}^2$} 
	\State  Set $\lambda \gets \tau \lambda $.
	\State Recompute $u = u(\lambda)$ by solving \eqref{eq:gen_algorithm_p_u} with the new $\lambda > 0$.
\EndWhile 
\State Return $\widehat{u} = u(\lambda)$ and step size $\lambda_k := \lambda$.
\end{algorithmic}
\end{algorithm}
\noindent
We argue that the line search strategy terminates after finitely many iterations.
\begin{lemma}Consider the \cref{alg:algorithm_general} above together with the line search strategy \cref{alg:backtracking1} at the $k$-th iterate. Given $u_k$ and the line search parameters $(\sigma_{BT},\lambda_0)$, the iterates of the backtracking algorithm converge after at most $j_* = \log_\tau(L_2/(1-\sigma_{BT})\lambda_0)$ iterations. In particular the sequence of step sizes satisfies $C_{\mathrm{desc}} := \sigma_{BT} \lambda_0 \leq \lambda_k \leq \lambda_0 \tau^{j_*}$.
\end{lemma}
\begin{proof} The proof is standard but we provide it here for the convenience of the reader. Given a current step size parameter $\lambda>0$ consider the unique solution $u(\lambda) \in H_0^1(\Omega)$ of \eqref{eq:gen_algorithm_p_u}. Using \autoref{fundamental_inequality}, we directly obtain
\begin{align*}
    J(u(\lambda),z_{k+1}) 
    \leq 
    J(u_k,z_{k+1}) - \left(\frac{\lambda-L_2}{2} \right)\| u(\lambda) - u_{k} \|_{H_0^1(\Omega)}^2 
\end{align*}
with the constant $L_2>0$ from \autoref{fundamental_inequality}. Hence, whenever $\sigma_{BT} \lambda \leq (\lambda - L_2)$ we also have
\begin{equation}
    J(u(\lambda),z_{k+1}) 
    \leq 
    J(u_k,z_{k+1}) - \left(\frac{\sigma_{BT} \lambda}{2} \right)\| u(\lambda) - u_{k} \|_{H_0^1(\Omega)}^2, \notag
\end{equation}
which is the condition for acceptance in the backtracking algorithm above. Thus when initialized with $\lambda_0$ then the line-search stops if $\sigma_{BT}
\tau^j \lambda_0 \leq (\tau ^j\lambda_0 - L_2)$ which happens after at most $j^* := \log_\tau(L_2/(1-\sigma_{BT})\lambda_0)$ steps.
\end{proof}
\noindent
Next we state a basic global convergence result for \cref{alg:algorithm_general} guaranteeing global sublinear rate.
\noindent
\begin{theorem}[Global convergence to stationarity]\label{thm_global_convergence}
Let \autoref{assumptions_general} and  \autoref{assumption_z_step} hold true and let $(u_k,z_k)_{k \in \mathbb{N}} \subset H_{0}^1(\Omega) \times Z$ be a sequence that is generated by \cref{alg:algorithm_general} with step size parameters with $\lambda_k - L_2 \geq C_{\mathrm{desc}} >0$ for $k \in \mathbb{N}$. In addition assume that the sequence of accuracy parameters are square summable, i.e. $\sum_k \eta_k^2 < + \infty $. Then the following statements hold:
\begin{itemize}
    \item[(i)] The sequence $(u_k,z_k)_{k\in \mathbb{N}}$ is bounded in $H_0^1(\Omega) \times Z$.
    \item[(ii)] The function values converge monotonically to its infimum and the slope converges at a globally sub-linear rate, i.e., there is a constant $C>0$ such that
    \begin{equation}
    \min_{k=1,\ldots,N}  
    \dist(0,\partial J(u_{k},z_{k}))
    \leq
     C \sqrt{ \frac{\left( J(u_0,z_0) - J(u_{N},z_{N}) \right) + \sum_{k=0}^\infty \eta_k^2 }{N}}, \notag
    \end{equation}
    \item[(iii)]If the sequence \((u_k)_{k \in \mathbb{N}}\) is bounded in \(H^2(\Omega)\), then any limit point \((u^*, z^*)\) at which the functional \(J\) is continuous is a stationary point.
\end{itemize}
\end{theorem}
\begin{proof} 
For the proof we denote $R_1(u) = \mathcal{I}_{U_{ad}}(u)$ for convenience and start with (i). As in the remark after \autoref{fundamental_inequality} we obtain
    \begin{equation}
    J(u_{k+1},z_{k+1}) \leq J(u_k,z_{k+1}) - \frac{C_{\mathrm{desc}}}{2}\|u_{k+1} - u_k\|_{H_0^1(\Omega)}^2 \notag
    \end{equation}
    for every $k \in \mathbb{N}$. Taking into account the descent property of \cref{alg:algorithm_z}, i.e. \autoref{assumption_z_step}, we infer 
    $J(u_k,z_{k+1}) \leq J(u_k,z_{k})$ and consequently
     \begin{equation}
    J(u_{k+1},z_{k+1}) \leq J(u_k,z_{k}) - \frac{C_{\mathrm{desc}}}{2}\|u_{k+1} - u_k\|_{H_0^1(\Omega)}^2. \notag
    \end{equation}
    As $J$ is bounded from below, the sequence $(J(u_k,z_k))_{k\in\mathbb{N}}$ converges monotonically to its infimum. By the same arguments as in the proof of \autoref{theorem_existence_general}, we infer boundedness of $(u_k,z_k)_{k\in\mathbb{N}}$ in $H_0^1(\Omega) \times Z$. \\
    To show (ii) we deduce from  \autoref{fundamental_inequality} that the approximation error $e_{\lambda_k}$ satisfies
    \begin{align*}
    e_{\lambda_k}(u,u_k) 
    &=
    g_{\lambda_k}(u,u_k) - f(u) \notag\\
    &=
    \frac{1}{2} \| F'(u_k)[u - u_k] + F(u_k) - f^\delta  \|^2_{L^2(\Omega)}  + \frac{\lambda_k}{2} \| u - u_k\|^2_{H_0^1(\Omega)} - \frac{1}{2}\|F(u) - f^\delta\|^2_{L^2(\Omega)} \notag\\
    &\geq 
    \frac{\lambda_k - L_2}{2} \| u - u_k\|^2_{H_0^1(\Omega)} 
    \geq \frac{C_{\mathrm{desc}}}{2} \| u - u_k\|^2_{H_0^1(\Omega)},
\end{align*}
where $L_2>0$ is the constant in \autoref{fundamental_inequality}. The optimality system for $u_{k+1}$ given $z_{k+1}$ gives:
\begin{equation}
    0 \in \nabla_u g(u_{k+1},u_k)  - \alpha \Delta u_{k+1} + \nabla_u h(u_{k+1},z_{k+1}) - \lambda_k \Delta 
(u_{k+1}-u_k) +   \partial R_1(u_{k+1}), \label{eq:global_conv_eq_1} 
\end{equation}
which follows from standard (convex) subdifferential calculus in the space $H^{-1}(\Omega)$.  We set 
\begin{equation}
  A_k:= \nabla  f(u_{k+1}) - \nabla f(u_k) -   F'(u_k)^*(F'(u_k)[u_{k+1}-u_k]  + \lambda_k \Delta (u_{k+1}-u_k) \in H^{-1}(\Omega) \notag 
\end{equation}
and rewrite \eqref{eq:global_conv_eq_1} as
\begin{equation}
    A_k \in \nabla f(u_{k+1}) - \alpha \Delta u_{k+1} + \nabla_u h(u_{k+1},z_{k+1}) +  \partial R_1(u_{k+1}) = \partial_u J(u_{k+1},z_{k+1}). \notag
\end{equation}
Using \autoref{fundamental_inequality} and the chain rule it is not difficult to see that $\| A_k \|_{H^{-1}(\Omega)} \leq C_1\|u_{k+1} - u_k \|_{H_0^1(\Omega)}$ for some $C_1>0$. For the $z$-step we obtain for arbitrary $w \in \partial R (z_{k+1})$
\begin{align*}
    \| \nabla_z   h(u_{k+1},z_{k+1}) + w \|_Z
     &= \| \nabla_z  h(u_{k+1},z_{k+1}) - \nabla_z  h(u_{k},z_{k+1}) \|_Z +
    \| \nabla_z  h(u_{k},z_{k+1}) + w \|_Z. \\
     &\leq C_2 \| u_{k+1} - u_{k}\|_{H_0^1(\Omega)} +
    \| \nabla_z  h(u_{k},z_{k+1}) + w \|_Z,
\end{align*}
where we used \autoref{assumptions_general} (H3) in the last inequality to find the constant $C_2>0$. We further observe by the minimization property, that 
\begin{align*}
e_{\lambda_k}(u_{k+1},u_k) 
&= g_{\lambda_k}(u_{k+1},u_k) - f(u_{k+1}) \\
& \leq 
f(u_{k}) + \frac{\alpha}{2} \|\nabla u_k \|^2_{L^2(\Omega)}  + h(u_{k},z_{k+1}) + R(z_{k+1}) + R_1(u_{k})  \\ 
& \quad - \frac{\alpha}{2} \|\nabla u_{k+1} \|^2_{L^2(\Omega)} - h(u_{k+1},z_{k+1}) - R_1(u_{k+1}) - R(z_{k+1}) - f(u_{k+1}).
\end{align*}
Using the fact that $h(u_{k},z_{k+1}) +  R(z_{k+1}) \leq h(u_{k},z_k) + R(z_{k})$ by the property \eqref{eq:descent_z},  we infer 
\begin{equation}
    e_{\lambda_k}(u_{k+1},u_k)
    \leq
     J(u_{k}, z_{k}) - J(u_{k+1},z_{k+1}). \notag
\end{equation}
Consequently, from previous inequalities, we deduce for every $k \in \mathbb{N}$
\begin{align}
    \|A_k\|^2_{H^{-1}(\Omega)} + \| \nabla_z   h(u_{k+1},z_{k+1}) + w \|^2_Z 
&\leq 
 (C_1^2+2C_2^2)\|u_{k+1} - u_k \|_{H_0^1}^2 + 2\| \nabla_z   h(u_{k},z_{k+1}) + w \|_Z^2  \notag \\
&\leq \frac{2(C_1^2+2C_2^2)}{C_{\mathrm{desc}}}  e_{\lambda_k}(u_{k+1},u_k) +  2\| \nabla_z   h(u_{k},z_{k+1}) + w \|_Z^2  \notag \\
&\leq  \frac{2(C_1^2+2C_2^2)}{C_{\mathrm{desc}}}
   \left( J(u_{k}, z_{k}) - J(u_{k+1},z_{k+1}) \right) \notag\\
&\quad 
+  2\| \nabla_z   h(u_{k},z_{k+1}) + w \|_Z^2. \label{eq:global_conv_eq_2} 
\end{align}
By the sum-rule for the convex subdifferential (cf. \cite{bauschke2017correction})  including the remark below, we deduce
\begin{equation}
    (A_k ,  \nabla_z   h(u_{k+1},z_{k+1}) + w ) \in \partial J(u_{k+1}, z_{k+1}) \notag
\end{equation}
and therefore, by definition of the slope, cf. \cite{frankel2015splitting} also 
\begin{equation}
    \dist(0,\partial J(u_{k+1},z_{k+1}))^2 \leq  
\ \frac{2(C_1^2+2C_2^2)}{C_{\mathrm{desc}}}
   \left( J(u_{k}, z_{k}) - J(u_{k+1},z_{k+1}) \right)
+  2\| \nabla_z   h(u_{k},z_{k+1}) + w \|_Z^2. \notag
\end{equation}
Taking the infimum over all $w \in \partial R(z_{k+1})$ we deduce  
\begin{equation}
   \dist(0,\partial J(u_{k+1},z_{k+1}))^2 \leq  
 \frac{2(C_1^2+2C_2^2)}{C_{\mathrm{desc}}}
   \left( J(u_{k}, z_{k}) - J(u_{k+1},z_{k+1}) \right)
+  2 \eta_k. \notag
\end{equation}
Summing from $0$ to $N-1$ leads to 
\begin{equation}
   \sum_{k=0}^{N-1} \dist(0,\partial J(u_{k+1},z_{k+1}))^2
    \leq \frac{2(C_1^2+2C_2^2)}{C_{\mathrm{desc}}}\left( J(u_0,z_0) - J(u_{N},z_{N})\right) + 2\sum_{k=0}^{N-1} \eta_k^2. \notag
\end{equation}
Consequently, with $C := \sqrt{\max(2, 2(C_1^2 + C_2^2)/C_{\mathrm{desc}})}>0$, we obtain 
\begin{equation}
    \min_{k=1,\ldots,N}  
    \dist(0,\partial J(u_{k},z_{k}))
    \leq
     C \sqrt{ \frac{\left( J(u_0,z_0) - J(u_{N},z_{N}) \right) + \sum_{k=0}^{N-1} \eta_k^2 }{N}}, \label{eq:thm_global_final_eq}
\end{equation}
which proves the assertion in (ii). \\
To show (iii) assume an accumulation point ${x^*:=(u^*,z^*) \in U_{ad} \times Z}$. We recall that $H^2(\Omega) \hookrightarrow H^1(\Omega)$ compactly, cf. \cite{attouch2014variational}. Thus we may extract a $H^1$-strongly convergent subsequence such that $u_{n_k} \to u^*$ and $z_{n_k} \to z^*$ after selection of a suitable subsequence in $H^1(\Omega) \times Z$. By \eqref{eq:thm_global_final_eq} we know that for any $v_{k} \in \partial J(u_{n_k},z_{n_k})$ ,we have $v_k \to 0$ as $k \to \infty$.
Moreover by exploiting continuity of $J$ we obtain
\begin{equation}
    J(x_{n_k}) \to J(x^*) \quad v_{k} \to 0 \quad x_{n_k} \to x^*.
\end{equation}
That is the definition of the so called $J$-attentive convergence. By the closedness of the limiting subdifferential under J-attentive convergence, cf. \cite{frankel2015splitting}, we infer that $0 \in \partial J(x^*)$. 
\end{proof}

\noindent 
\subsection{Local convergence analysis under the KL-inequality}
In this section we analyze local strong convergence of \cref{alg:algorithm_general} near a global minimizer $x^* = (u^*, z^*) \in U_{ad} \times Z$, focusing on conditions that ensure fast local convergence of the Levenberg–Marquardt method. During this section we will often use the notation $x=(u,z)$ and the product norm defined by $\|x\|_X^2 := \|u\|_{H_0^1(\Omega)}^2 + \|z\|_Z^2$.
\begin{assumption}[On a global minimum]\label{assumption_local}
 We impose the following conditions:
\begin{itemize}
\item[(C1)] $x^*= (u^*,z^*) \in U_{ad} \times Z$ is a global minimizer of \eqref{eq:p0_general} at which \(J\) is continuous in \(\|\cdot\|_X\).
\item[(C2)] There exists $\kappa > 0$ such that
\begin{equation}
    \langle \nabla^2 f^\alpha(u^*)[h], h \rangle_{H^{-1}(\Omega)} \geq \frac{\kappa}{2} \|h\|^2_{H^1_0(\Omega)} \quad \forall h \in H^1_0(\Omega). \label{eq:local_assumption_hessian}
\end{equation}
\item[(C3)] The function $g(z) := h(u^*,z) + R(z)$ satisfies a Kurdyka–Łojasiewicz (KL) inequality at $z^*$ with exponent $\beta = 1 - 1/q$, $q > 2$, i.e., there exist constants $\eta, C > 0$ and $\varepsilon > 0$ such that
\begin{equation}
    g(z) - g(z^*) \leq C \dist(0, \partial g(z))^q \quad \text{for $z \in B_{\varepsilon}(z^*)$ with $g(z^*) < g(z) < g(z^*) + \eta$}.\label{eq:KL_inequality_exponent}
\end{equation}
  The constants $q$, $\varepsilon_{}$, $C_{}$, and $\eta$ may depend on $u^*$.
\end{itemize}
\end{assumption}
\noindent
\begin{remark}[General KL-inequality] The KL-inequaltity with a given exponent is only a special case of the general KL inequality which is investigated in \cite{attouch2010proximal,attouch2013convergence} and has the form
\begin{equation}
    \varphi' (g(z) - g(z^*)) \dist(0,\partial g(z)) \geq 0. \label{eq:KL_inequality_general}
\end{equation}
 For a certain concave $C^1$-function $\varphi:(0,\eta) \to \mathbb{R}_+$, also called \emph{desingularization function}. We directly see that \eqref{eq:KL_inequality_exponent} is obtained from \eqref{eq:KL_inequality_general} by using $\varphi(t) = C_{}^{1-\beta} t ^\beta$.  We will not go further into details here, but rather refer to \cite{attouch2013convergence} for details.   
\end{remark}
\begin{remark}[Discussion of \autoref{assumption_local}.]
   Condition (C2) is needed to apply Robinson’s generalized implicit function theorem \cite{dontchev2009implicit,ito2008lagrange}, which yields a locally Lipschitz continuous solution mapping $z \mapsto S(z)$ around $z \in Z$. Condition (C3), based on the KL-inequality, has become standard in finite-dimensional optimization \cite{attouch2010proximal,attouch2013convergence,bolte2014proximal}. The KL property holds for many functions, especially those definable in o-minimal structures \cite{attouch2013convergence,bolte2010characterizations,ochs2015long}. A large KL-exponent $\beta \in (0,1]$ (i.e., $q \to \infty$) typically indicates faster local convergence due to greater sharpness near stationary points. While the existence of some $\beta$ is often easy to establish via real algebraic geometry, computing an explicit value is typically difficult; see \cite{li2018calculus,yu2022kurdyka} for some results. Extensions to infinite dimensions exist \cite{chill2003lojasiewicz,bolte2010characterizations,hauer2019kurdyka}, though KL-exponents are harder to compute due to the lack of corresponding  tools from real algebraic geometry.
\end{remark}
\noindent
For the subsequent part, we need uniform growth properties, as proven in  the following lemma .
\begin{lemma}[Uniform Taylor bound]\label{lemma_quadratic_growth} Let $u^* \in H_0^1(\Omega)$ such that
\begin{equation}
    \langle \nabla f^\alpha (u^*) [h], h \rangle_{H^{-1}(\Omega)} \geq \frac{\kappa}{2} \| h \|^2_{H_0^1(\Omega)} \quad \text{for all $h \in H_0^1(\Omega)$}. \notag
\end{equation}
Then there is an $\varepsilon>0$ such that 
\begin{equation}
f^\alpha(u) \geq f^\alpha(u_0) + \langle \nabla f^\alpha(u_0) , u - u_0 \rangle_{H^{-1}(\Omega)} + \frac{\kappa}{8} \|u - u_0\|_{H_0^1(\Omega)}^2, \notag
\end{equation}
whenever $u,u_0 \in B_\varepsilon(u^*)$.
\end{lemma}
\begin{proof}
The result follows from Taylor's formula with integral remainder:
\begin{align*}
    f^\alpha(u) &= f^\alpha(u_0) + \langle \nabla f^\alpha(u_0), u - u_0 \rangle_{H^{-1}(\Omega)} 
    + \frac{1}{2} \langle \nabla^2 f^\alpha(u_0)[u - u_0], u - u_0 \rangle_{H^{-1}(\Omega)} \\
    &\quad + \int_0^1 (1-\tau) \langle [\nabla^2 f^\alpha(u_0 + \tau(u - u_0)) - \nabla^2 f^\alpha(u_0)][u - u_0], u - u_0 \rangle_{H^{-1}(\Omega)} \, d\tau.
\end{align*}
Adding and subtracting \(\frac{1}{2} \langle \nabla^2 f^\alpha(u^*)[u - u_0], u - u_0 \rangle_{H^{-1}(\Omega)}\), and applying the assumptions, we obtain:
\begin{align}
    f^\alpha(u) &\geq f^\alpha(u_0) + \langle \nabla f^\alpha(u_0), u - u_0 \rangle_{H^{-1}(\Omega)} 
    + \frac{\kappa}{4} \|u - u_0\|_{H_0^1(\Omega)}^2 \notag \\
    &\quad - \frac{1}{2} \| \nabla^2 f^\alpha(u_0) - \nabla^2 f^\alpha(u^*) \| \cdot \|u - u_0\|_{H_0^1(\Omega)}^2 \notag \\
    &\quad - \int_0^1 (1 - \tau) \| \nabla^2 f^\alpha(u_0 + \tau(u - u_0)) - \nabla^2 f^\alpha(u_0) \| \cdot \|u - u_0\|_{H_0^1(\Omega)}^2 \, d\tau. \label{eq:eq_uniform_taylor_eq1_short}
\end{align}
By continuity of \(\nabla^2 f^\alpha\), choosing \(\varepsilon_1 > 0\) such that for all \(v_1, v_2\) with \(\|v_i - u^*\|_{H_0^1(\Omega)} < \varepsilon_1\) we have
\begin{equation}
\| \nabla^2 f^\alpha(v_1) - \nabla^2 f^\alpha(v_2) \| \leq \frac{\kappa}{8}.
\end{equation}
 Setting \(\varepsilon := \varepsilon_1/2\) we deduce from  \eqref{eq:eq_uniform_taylor_eq1_short} that for all \(u, u_0\) with \(\|u - u^*\| < \varepsilon\) and \(\|u_0 - u^*\| < \varepsilon\), 
\begin{equation}
f^\alpha(u) \geq f^\alpha(u_0) + \langle \nabla f^\alpha(u_0), u - u_0 \rangle_{H^{-1}(\Omega)} + \left( \frac{\kappa}{4} - \frac{\kappa}{16} - \frac{\kappa}{16} \right) \|u - u_0\|_{H_0^1(\Omega)}^2,
\end{equation}
which gives the claim.
\end{proof}
\noindent
We will now show that the set-valued map, which sends $z$ to the set of stationary points of $J(\cdot,z)$ is in fact locally  single-valued and Lipschitz continuous.
\begin{lemma}[Local parameterization of the solution map]\label{lemma_ift}
Let \autoref{assumption_local}, (C2) hold at a stationary point $ x^* \in H_0^1(\Omega) \times Z$ with $\kappa>0$.  Denote the partial solution map by 
\begin{equation}
\omega: Z \rightrightarrows H_0^1(\Omega), \quad \omega(z) = \partial_u J(\cdot, z)^{-1}(0).
\end{equation}
Then the following two assertions are true:
\begin{itemize}
    \item[(i)]There is an $\varepsilon_1>0$ such that $\omega: Z \rightrightarrows H_0^1(\Omega)$ admits a single-valued, \(L_S = L_S({z^*})\)-Lipschitz continuous localization \(S: B_\varepsilon(z^*) \to H_0^1(\Omega)\) with $L \leq \mathcal{O}(1/\kappa)$ and {\(S(z^*) = u^*\)}. \\
    \item[(ii)] There are $0<\varepsilon_2 \leq \varepsilon_1$ with $\varepsilon_1$ from (i) and \(C_1 > 0\) independent of \(k\) such that, 
    \begin{align}
    \|S(z_{k+1}) - u_k\|_{H_0^1(\Omega)} &\leq C_1 \|u_{k+1} - u_k\|_{H_0^1(\Omega)}, \label{eq:ift_1} \\
    J(S(z_{k+1}), z_{k+1}) + \frac{\kappa}{8} \|S(z_{k+1}) - S(z^*)\|_{H_0^1(\Omega)}^2 &\leq J(S(z^*), z_{k+1}), \label{eq:ift_2} 
    \end{align}
    whenever $\|x_{k+1} - x^*\|_X \leq \varepsilon_2$. 
\end{itemize}
\end{lemma}
\begin{proof} 
For the proof, we again use the notation $R_1(u) =  \mathcal{I}_{U_{ad}}(u)$ for simplicity. By assumption, $u^* \in \omega(z^*)$ or equivalently, $u^*$ solves the generalized equation
       $0 
         \in G(u,z^*) + \partial R_1(u) := \nabla f(u) - \alpha \Delta u  + \nabla_u h(u,z^*) + \partial R_1(u)$, 
    in $H^{-1}(\Omega)$. We aim to apply the generalized implicit function theorem by Robinson in the version \cite{dontchev2021lectures}, Theorem 8.5, see also \cite{dontchev2009implicit}. For this purpose consider an approximated generalized equation around $z^* \in Z$, namely $0 \in G_{z^*}(u)   + \partial R_1(u)$, with $G_{z^*}: H_0^1(\Omega) \to H^{-1}(\Omega)$ defined as
    \begin{equation}
        G_{z^*}(u) :=  \nabla f(u^*) +  \nabla^2f(u^*[u-u^*]  - \alpha \Delta u + \nabla_u h(u,z^*). \notag
    \end{equation}
     To apply the generalized implicit function theorem, we need to show strong regularity of the linearized equation. For this, it suffices to verify that the solution map of the linearized equation,
    \begin{equation}
    S_{z^*}:H^{-1}(\Omega) \to H_0^1(\Omega) \quad  S_{z^*}(p) =  (G_{z^*} + \partial R_1)^{-1}(p), \notag
    \end{equation}
    has a Lipschitz-continuous single valued localization around $p = 0$. This is a consequence of the fact that $p \in G_{z^*}(u) + \partial R_1(u)$ if and only if
    \begin{equation}
        S_{z^*}(p) \in \argmin_{u \in H_0^1(\Omega)} \; \langle \nabla^2 f(u^*)[u - u^*] , u-u^* \rangle_{H^{-1}(\Omega)} + \frac{\alpha}{2}\| \nabla u \|^2_{L^2(\Omega)}  + h(u,z^*)  - \langle p, u \rangle_{H^{-1}(\Omega)} + R_1(u).  \notag
    \end{equation}
    As this problem is $\kappa$-strongly convex, with $\kappa>0$ from \autoref{assumption_local}, it follows from standard arguments that the associated solution map $S_{z^*}$ is single-valued and $1/\kappa$ -Lipschitz continuous. Introducing the approximation error 
    \begin{equation}
        e(u,z) := G_{z^*}(u) - G(u,z) = \nabla f(u) - \nabla f(u^* - \nabla^2f(u^*)[u-u^*], \notag  
    \end{equation}
    the partial uniform Lipschitz-modulus $\widehat{\mathrm{Lip}}_{u^*}(G_{z^*} - G;(u^*,z^*) ) $ 
    as introduced in \cite[p.8]{dontchev2021lectures}, satisfies
    \begin{align*}
    \widehat{\mathrm{Lip}}_{u^*}(G_{z^*} - G;(u^*,z^*) ) 
    &= 
    \mathop{\limsup_{u_1, u_2 \to u^*}}_{z \to z^*} \frac{\| e(u_1,z) - e(u_2,z) \|_{H^{-1}(\Omega)}} { \| u_1 - u_2\|_{H_0^1(\Omega)}}  \\
     &= 
    \limsup_{u_1, u_2 \to u^*} \frac{\| \nabla f(u_1) - \nabla f(u_2) - \nabla^2 f(u^*[u_1-u_2] \|_{H^{-1}(\Omega)}} { \| u_1 - u_2\|_{H_0^1(\Omega)}} \\
    &\leq \limsup_{u_1, u_2 \to u^*}
         \left(\int_0^1 \| \nabla^2 f(u_2 + \tau(u_1-u_2)) - \nabla^2f(u^*)\|_{H^{-1}(\Omega)} \diff \tau \right) = 0,
    \end{align*}
 where we used the mean value theorem in the last inequality. By \cite[Theorem 8.5]{dontchev2021lectures}, the mapping $\omega:Z \rightrightarrows H_0^1(\Omega)$ has a single-valued Lipschitz continuous localization $S:Z \to H_0^1(\Omega)$ in a ball $B_\varepsilon(z^*)$ with a Lipschitz-constant $0 < L_S \lesssim 1/\kappa$, showing the first part of \autoref{lemma_ift}. Now  consider the first-order optimality condition for deriving $u _{k+1}$ given $z = z_{k+1}$, as in \eqref{eq:P_u}:
    \begin{equation}
        R_1(u) \geq  - \langle \nabla f(u_{k+1}) - \alpha \Delta u_{k+1}  + \nabla_u h(u_{k+1},z_{k+1}) - A_k,u-u_{k+1} \rangle_{H^{-1}(\Omega)} + R_1(u_{k+1}), \label{eq:ift_eq1}
    \end{equation}
    for every $u \in H_{0}^1(\Omega)$, where we defined $A_k \in H^{-1}(\Omega)$ as
    \begin{equation}
           A_k:= \nabla f(u_{k+1}) - \nabla f(u_k)  -  F'(u_k)^* F'(u_k)[u_{k+1}-u_k]) + \lambda_k\Delta (u_{k+1}-u_k). \notag
    \end{equation}
      Similarly the first order conditions for the stationary point $S(z_{k+1}) \in H_0^1(\Omega)$ are given by:
     \begin{equation}
        R_1(u) \geq  - \langle \nabla f(S(z_{k+1})) - \alpha \Delta S(z_{k+1})  + \nabla_u h(S(z_{k+1}),z_{k+1}),u-S(z_{k+1})\rangle + R_1(S(z_{k+1})), \label{eq:ift_eq2}
    \end{equation}
    for all $u \in H_0^1(\Omega)$. Adding \eqref{eq:ift_eq1} with $u = S(z_{k+1})$ and \eqref{eq:ift_eq2} with $u = u_{k+1}$ yields
    \begin{align*}
        0 
        &\geq  \langle \nabla f^\alpha(S(z_{k+1})) - \nabla f^\alpha(u_{k+1}),S(z_{k+1}) - u_{k+1}\rangle_{H^{-1}(\Omega)}\\
        & \quad + \langle \nabla_u h(S(z_{k+1}),z_{k+1}) - \nabla_u h(u_{k+1},z_{k+1}) , S(z_{k+1}) -u_{k+1} \rangle_{H^{-1}(\Omega)} - \langle q_k, S(z_{k+1}) - u_{k+1} \rangle_{H^{-1}(\Omega)}.
    \end{align*}
    Note that by \autoref{lemma_quadratic_growth}, we infer
    \begin{equation}
        \langle \nabla f^\alpha(S(z_{k+1})) - \nabla f^\alpha(u_{k+1}),S(z_{k+1}) - u_{k+1}\rangle_{H^{-1}(\Omega)} \geq \frac{\kappa}{8}\|S(z_{k+1}) - u_{k+1}\|^2_{H_0^1(\Omega)}
    \end{equation}
    whenever $\|x_{k+1} - x^*\|_X \leq \varepsilon$ with a possibly smaller $\varepsilon$.
    By shifting the term with $A_k$ on the left hand side and use \autoref{lemma_quadratic_growth}, we obtain the bound:
    \begin{align*}
        \frac{\kappa}{8}\| S(z_{k+1}) - u_{k+1} \|^2_{H_0^1(\Omega)}  &\leq \| q_k\|_{H^{-1}(\Omega)} \| S(z_{k+1}) - u_{k+1}\|_{H^1_0(\Omega)} \\ 
        &\leq 
        \tilde{C} \| u_{k+1} - u_k \|_{H_0^1(\Omega)} \| S(z_{k+1}) - u_{k+1}\|_{H_0^1(\Omega)} 
    \end{align*}
    for a constant $\tilde{C}>0$. The estimate $\| A_k\|_{H^{-1}(\Omega)} \leq \tilde{C}\| u_{k+1} - u_k \|_{H_0^1(\Omega)} $ follows as in the proof of \autoref{thm_global_convergence}.
    If $\|S(z_{k+1}) - u_{k+1}\|_{H_0^1(\Omega)} = 0$, \eqref{eq:ift_1} is obvious. For the other case we obtain from the triangle inequality
    \begin{equation}
        \| S(z_{k+1}) - u_{k} \|_{H_0^1(\Omega)} \leq (\tilde{C} + 1) \| u_{k+1} - u_k \|_{H_0^1(\Omega)} ,
    \end{equation}
    thus \eqref{eq:ift_1} follows with $C_1 = \tilde{C} + 1$. To show \eqref{eq:ift_2}, consider again  \autoref{lemma_quadratic_growth}, which guarantees that     \begin{align}
       f^\alpha(u^*) \geq f^\alpha(S(z_{k+1})) + \langle \nabla f^\alpha(S(z_{k+1})) , u^* - S(z_{k+1}) \rangle_{H^{-1}(\Omega)} + \frac{\kappa}{8} \|u^* - S(z_{k+1})\|_{H_0^1(\Omega)}^2, \notag
    \end{align}
    whenever $ \|x^* - x_{k+1} \|_{X} \leq \varepsilon$. Moreover by convexity of $h(\cdot,z_{k+1}) + R_1$
    \begin{equation}
        h(u^*,z_{k+1}) + R_1(u^*) \geq \langle\nabla_u h(S(z_{k+1}),z_{k+1}) + v,u^* - S(z_{k+1}) \rangle + h(S(z_{k+1}), z_{k+1}) + R_1(S(z_{k+1})) \notag
    \end{equation}
    for all  $v\in \partial R_1(S(z_{k+1}))$. Adding both previous inequalites and using first order optimality of $S(z_{k+1})$, we infer \eqref{eq:ift_2} with $\varepsilon_1 = \varepsilon$.
\end{proof}
\noindent 
\begin{remark}
    In the subsequent part of the article, we will always use the notation
        $S: B_\varepsilon(z^*)\cap Z \to H_0^1(\Omega)$ 
    for the single-valued localization of $\omega$ from \autoref{lemma_ift}.
\end{remark}
\noindent
Before presenting the main theorem of the section let us recall the following descent lemma:
\begin{lemma}[General descent inequality]\label{lemma_general_descent}
For arbitrary $u \in U_{ad}$ and step sizes $\lambda_k - L_2 \geq C_{desc}$ the following inequalities hold true for any $k \in \mathbb{N}$ and $u \in H_0^1(\Omega)$:
\begin{align}
    J(u,z_{k+1})
    &\geq 
    J(x_{k+1}) + \left( \frac{\lambda_k - L_2}{2} \right) \| u_{k+1} - u_{k} \|^2_{H_0^1(\Omega)} - \left( \frac{\lambda_k + L_2}{2}\right) \|u - u_k\|^2_{H_0^1(\Omega)},  \label{eq:gen_descent_local_1}\\
    J(x_k) 
    &\geq J(x_{k+1}) 
    + \left( \frac{\lambda_k - L_2}{2} \right) \| u_{k+1} - u_{k} \|^2_{H_0^1(\Omega)} + \sigma_1 \sum_{i=0}^{n_k - 1} \|z_k^{i+1} -z_k^{i}\|^2_Z. \label{eq:gen_descent_local_2}
\end{align}    
Here, $L_2 >0$ denotes again the constant in \autoref{fundamental_inequality}.
\end{lemma}
\begin{proof}
    As the proof is essentially the same as the proof of the first part of \autoref{thm_global_convergence}, we postpone it to the appendix.
\end{proof}
\noindent
We now show the main theorem of this section, which states that if \cref{alg:algorithm_general} is initialized near a global minimizer $x^*=(u^*,z^*) \in U_{ad} \times Z$ that satisfies \autoref{assumption_local}, then the iterates converge to $x^*$ with linear speed. The proof is inspired by the analysis in \cite{drusvyatskiy2018error}.
\begin{theorem}[Local convergence of \cref{alg:algorithm_general} to global minima] \label{theorem_local_conv}
    Let the sequence $(x_k)_{k\in\mathbb{N}}$ in $H^1_0(\Omega) \times Z$ be generated by \cref{alg:algorithm_general} and let \autoref{assumption_local} holds true at a global minimum point $x^* \in U_{ad} \times Z$ of $J$. Moreover, assume that the accuracies of the inner loop are given by $\eta_k \lesssim k^{-\gamma}$, $\gamma > 1/2 $. Then, there is a radius $r>0$ such that if $x_0 \in B_r(x^*)$, the following statements hold true:
    \begin{itemize}
    \item[(i)] The function values converge linearly, i.e., there is $Q_1 \in (0,1)$ such that 
    \begin{equation}
        J(x_{k+1}) - J(x^*) \leq Q_1 \left( J(x_{k}) - J(x^*)  \right) \quad \text{for all $k \in \mathbb{N}$}.
    \end{equation}
    \item[(ii)] The iterates converge at a linear rate, i.e. there exists a $C>0$ and a $Q_2\in (0,1)$ such that
    \begin{equation}
        \|x_{k} - \tilde{x}\|_{X} \leq C  Q_2^k \left( J(x_0) - J(x^*) \right), \quad \text{for all $k \in \mathbb{N}$},\label{eq:global_theorem_statement_2}
    \end{equation}
    where $\tilde{x} \in B_r(x_0)$ such that $J(\tilde{x}) = J(x^*)$.
    \end{itemize}
\end{theorem}
\begin{proof} We split the proof into three parts. \\[0.2cm]
Step I: There is an $\varepsilon>0$ and $Q_1 \in (0,1)$ independent of $k \in \mathbb{N}$ such that $\|x_{k+1} - x^*\|_X < \varepsilon$ implies 
\begin{equation}
        J(x_{k+1}) - J(x^*) \leq Q_1 \left( J(x_{k}) - J(x^*)  \right).\label{eq:thm_local_step_1}
\end{equation}
To prove \eqref{eq:thm_local_step_1} let us fix $\varepsilon_1 > 0$ according to \autoref{lemma_ift} such that $z \mapsto S(z) \in H_0^1(\Omega)$ is well defined and single valued for all $z \in B_{\varepsilon_1}(z^*)$, such that the inequalities \eqref{eq:ift_1}, \eqref{eq:ift_2} hold true for $\| x_{k+1} - x^*\| \leq \varepsilon_1$. An application of the general descent \autoref{lemma_general_descent} with $u = S(z_{k+1})$ yields
\begin{equation}
    J(S(z_{k+1}),z_{k+1}) \geq J(u_{k+1},z_{k+1})
    + \left( \frac{\lambda_k - L_2}{2} \right) \| u_{k+1} - u_{k} \|^2_{H_0^1(\Omega)} 
     - \left( \frac{\lambda_k + L_2}{2}\right) \|S(z_{k+1}) - u_k\|^2_{H_0^1(\Omega)}. \notag
\end{equation}
Rearranging and using \eqref{eq:ift_1}, we obtain from the latter inequality 
\begin{align}
    J(u_{k+1},z_{k+1}) - J(S(z_{k+1}),z_{k+1}) 
    &\leq \| u_{k+1} - u_k\|^2_{H_0^1(\Omega)} \left( \left( \frac{\lambda_k+L_2}{2} \right) C_1^2  - \left( \frac{\lambda_k - L_2}{2}\right)\right)  \notag\\
    &=:C_{L,\lambda_k} \| u_{k+1} - u_k\|^2_{H_0^1(\Omega)}, \label{eq:thm_local_eq_0}
\end{align}
where $C_1>0$ is the constant in \eqref{eq:ift_1}. We have $J(u^*,z_{k+1}) \geq  J(x^*)$ by the minimization property. Hence, by KL-inequality with exponent $q>2$ and continuity of $J(\cdot)$ at $x^*$, we find $ 0 < \varepsilon_2 $ and $C_2>0$ with $J(x^*) \leq  J(u^*,z_{k+1})  < J(x^*) + \eta$ and
\begin{align*}
    J(u^*,z_{k+1})  - J(x^*) 
    &\leq C_2
      \dist(0, \partial_z J(u^*,z_{k+1}))^q,
\end{align*}
whenever $\| x_{k+1} - x^* \|_Z \leq \varepsilon_2$.
An application of the triangle inequality yields 
\begin{align*}
    J(u^*,z_{k+1})  - J(x^*)
     &\leq C_{3} 
     (
     \dist(0, \partial_z J(u_{k},z_{k+1})
     +  \|u^* - S(z_{k+1})\|_{H_0^1(\Omega)}  \\
     & \quad + \| S(z_{k+1})  - u_{k}\|_{H_0^1(\Omega)} 
     )^q ,
 \end{align*}
 for a constant $C_3>0$ that involves the global Lipschitz constant of $\nabla_z h(\cdot,z_{k+1})$. Recall from \eqref{eq:gradient_inequality_z}, that  $\dist(0, \partial_z J(u_{k},z_{k+1})) \leq \sigma_2 \| z_{k}^{n_k} - z_{k}^{n_k-1} \|_Z$. Moreover, by \eqref{eq:ift_1} and also by  using $a_1+a_2+a_3  \leq \sqrt{3} (a_1+a_2+a_3)^{1/2}$ for $a_i\geq 0$ we estimate 
 \begin{align*}
     J(u^*,z_{k+1})  - J(x^*)
     &\leq
     C_{3}\sqrt{3}
     \Big( \sigma_2^2  \sum_{i=0}^{n_k-1}\| z_{k}^{i+1} - z_{k}^{i} \|_Z^2
     +  C_1^2 \|u_{k+1} - u_{k}\|^2_{H_0^1(\Omega)}  \\
     &\quad   +  \|S(z_{k+1}) - u^*\|^2_{H_0^1(\Omega)}
     \Big)^\frac{q}{2}.
\end{align*}
Invoking the convexity of $x \mapsto x^{q/2}$ as $q>2$ and summarizing all constants, we infer the existence of a constant $C_4>0$ such that
\begin{align*}
    J(u^*,z_{k+1})  - J(x^*)
     &\leq 
     C_4
     \left(
     \sum_{i=0}^{n_k-1} \|z_{k}^{i+1} -  z_{k}^{i} \|^2_Z + \|u_{k+1} - u_{k}\|_{H_0^1(\Omega)}^2 \right)^{\frac{q}{2}} + 
      C_4   \|S(z_{k+1}) - u^*\|_{H_0^1(\Omega)}^{q}.    \notag 
\end{align*}
By the descent inequality \eqref{eq:gen_descent_local_2} and using the Lipschitz constant $L_S>0$ of $S$ from \autoref{lemma_ift},  we deduce with a possibly enlarged $C_4>0$ that
\begin{align}    
J(u^*,z_{k+1})  - J(x^*)
     &\leq 
     C_{4}
     \left(
     J(x_{k}) - J(x_{k+1})
     \right)^{\frac{q}{2}} +  C_4  \|S(z_{k+1}) - u^*\|_{H_0^1(\Omega)}^q \notag\\
     &\leq C_{4}
     \left(
     J(x_{k}) - J(x_{k+1})
     \right)^{\frac{q}{2}} \notag \\
     &\quad + C_4 L_S^{q-2} \| z_{k+1} - z^*\|^{q-2} (J(u^*, z_{k+1}) - J(S(z_{k+1}), z_{k+1}) ),
      \label{eq:thm_local_eq_1}
\end{align}
where we also used \eqref{eq:ift_2}. Shrinking $\varepsilon_2$ such that also 
 $C_4 L_S^{q-2} \| z_{k+1} - z^*\|^{q-2} \leq 1$ for all $\|z_{k+1} - z^*\|_Z <\varepsilon_2$, we obtain 
from \eqref{eq:thm_local_eq_1}:
\begin{equation}    
J(S(z_{k+1}),z_{k+1}) - J(x^*) 
\leq  C_4 (
     J(x_{k}) - J(x_{k+1}))^{\frac{q}{2}} \leq C_5 (J(x_{k}) - J(x_{k+1})),\label{eq:thm_local_eq_4} 
\end{equation}
for some constant $C_5>0$ using also $q>2$. Now adding \eqref{eq:thm_local_eq_0} and \eqref{eq:thm_local_eq_4}, we obtain 
\begin{equation}
    J(x_{k+1}) - J(x^*) \leq
    C_{L,\lambda_k} \| u_{k+1} - u_k\|^2_{H_0^1(\Omega)} + 
    C_5
    \left(
     J(x_{k}) - J(x_{k+1})
     \right). \notag
\end{equation}
 Again invoking  the descent inequality \eqref{eq:gen_descent_local_2} for the first summand, we find a generic constant $C>0$, such that 
\begin{equation}
    J(x_{k+1}) - J(x^*) \leq
    C
    \left(
     J(x_{k}) - J(x_{k+1})
     \right). \label{eq:thm_local_eq_5}
\end{equation}
 Inequality \eqref{eq:thm_local_eq_5} can be equivalently rephrased as \eqref{eq:thm_local_step_1} with $Q_1 = C/(1+C)$. \\[0.2cm] 
Step II: Show, that there is an $0 < r \leq \varepsilon$, where $\varepsilon>0$ from Step I, such that $x_0 \in B_r(x^*)$ implies $x_k \in B_r(x^*)$ for all $k \in \mathbb{N}$. \\[0.2cm] 
We start by choosing $r := \varepsilon>0$, with $\varepsilon>0$ from Step I.  We assume that $z_j \in B_r(z^*)$ for $j = 0,\ldots,k$ for some $k \in \mathbb{N}$. We then set $\varphi(t) := t^{1/2}$ and deduce by concavity:
\begin{equation}
    \varphi( J(x_j) - J(x^*) ) - \varphi( J(x_{j+1}) - J(x^*) ) \geq \varphi'( J(x_j) - J(x^*) )(J(x_j) - J(x_{j+1})). \label{eq:thm_local_eq_8}
\end{equation}
Note that this is defined, as $x^*$ is a global minimum. As  $x_j \in B_r(x^*)$ for $j=1,\ldots,k$, we obtain with help of step I,  \eqref{eq:thm_local_step_1}:
\begin{align}
    \varphi'( J(x_j) - J(x^*) ) 
    &= \frac{1}{2(J(x_j) - J(x^*))^{\frac{1}{2}}} 
    \geq  
    \frac{1}{2Q_1^\frac{j}{2}(J(x_{0}) - J(x^*))^{\frac{1}{2}}},\label{eq:thm_local_eq_9}
\end{align} 
From the descent inequality \eqref{eq:gen_descent_local_2} we infer that there exists a constant $C_6>0$ with
\begin{align}
    J(x_j) - J(x_{j+1}) 
    &\geq C_6
    \|u_{j+1} - u_j \|^2_{H_0^1(\Omega)} + C_6 \sum_{i=0}^{n_j-1}\|z_j^{i+1} - z_j^{i} \|_Z^2 \notag\\
    &\geq C_6
    \|u_{j+1} - u_j \|^2_{H_0^1(\Omega)} + \frac{C_6}{n_j} \|z_{j+1} - z_j \|_Z^2 \notag\\
    &\geq
    \frac{C_6}{n_j}\left( \|u_{j+1} - u_j \|^2_{H_0^1(\Omega)} + \|z_{j+1} - z_{j}\|^2_Z \right) = \frac{C_6}{n_j} \|x_{j+1} - x_j \|^2_X. \label{eq:thm_local_eq_10}
\end{align}
Then we conclude from \eqref{eq:thm_local_eq_8}, \eqref{eq:thm_local_eq_9} and \eqref{eq:thm_local_eq_10} that 
\begin{align}
    \frac{C_6}{n_j} \|x_{j+1} - x_j \|^2_X 
    \leq 2Q_1^{\frac{j}{2}}(J(x_{0}) - J(x^*))^{\frac{1}{2}} \left(
    \varphi( J(x_j) - J(x^*) ) - \varphi( J(x_{j+1}) - J(x^*) ) \right),
\end{align}
and henceforth
\begin{align}
     \|x_{j+1} - x_j \|_X
    &\lesssim
    \sqrt{n_j} Q_1^{\frac{j}{4}}(J(x_{0}) - J(x^*))^{\frac{1}{4}} \left(
    \varphi( J(x_j) - J(x^*) ) - \varphi( J(x_{j+1}) - J(x^*) ) \right)^{\frac{1}{2}} \notag \\
    &\leq 
    2 n_j Q_1^{\frac{j}{2}}(J(x_{0}) - J(x^*))^{\frac{1}{2}}  + 2 
    \left(
    \varphi( J(x_j) - J(x^*) ) - \varphi( J(x_{j+1}) - J(x^*) ) \right) \notag \\
    &\lesssim 
     2 j^{2\gamma} Q_1^{\frac{j}{2}}(J(x_{0}) - J(x^*))^{\frac{1}{2}}  + 2\left(
    \varphi( J(x_j) - J(x^*)) - \varphi( J(x_{j+1}) - J(x^*) ) \right), \label{eq:thm_local_eq_11}
\end{align}
where we used Young's inequality in the second step and \autoref{remark_complexity_z_algorithm} for the estimate of $n_j \lesssim j^{2\gamma}$ in the last.  We deduce by summing from $j =0,\ldots,k$ and using \eqref{eq:thm_local_eq_11} that
\begin{align*}
    \| x_{k+1} - x^* \|_Z 
    &\leq
    \| x_{k+1} - x_{0} \|_X + \| z_{0} - z^* \|_Z   
    \leq
    \sum_{j=0}^{k}
    \| x_{j+1} - x_{j} \|_X +\| z_{0} - z^* \|_Z\\
    &\lesssim (J(x_{0}) - J(x^*))^{\frac{1}{2}} 
    \sum_{j=0}^{k}
    j^{2\gamma} Q_1^{\frac{j}{2}}   
    + 2 \sum_{j=0}^{k} \left( \varphi( J(x_{j}) - J(x^*) ) - \varphi( J(x_{j+1}) - J(x^*) ) \right)  + \| z_{0} - z^* \|_Z \\
    &\lesssim (J(x_{0}) - J(x^*))^{\frac{1}{2}} 
    \sum_{j=0}^{ \infty}
    j^{2\gamma} Q_1^{\frac{j}{2}}  
    + \varphi( J(x_{0}) - J(x^*))   + \| z_{0} - z^* \|_Z,
\end{align*}
where  $ \sum_{j=0}^{\infty}
    j^{2\alpha} Q_1^{j} < + \infty$. Let $C_7>0$ denote the constant hidden in  $"\lesssim"$, and if 
\begin{equation}
   C_7  (J(x_{0}) - J(x^*))^{\frac{1}{2}}  
    \sum_{j=0}^{+ \infty}
    j^{2\gamma} Q_1^{\frac{j}{2}} + C_7 \varphi( J(x_{0}) - J(x^*) )  + C_7 \| z_{0} - z^* \|_Z < r, \label{eq:thm_local_eq_11_new}
\end{equation}
then also $z_{k+1} \in B_r(z^*)$,
which is possible by choosing a smaller $r>0$ due to continuity of $J$.
\\[0.3cm]
Step III: 
Using the results from steps I and II, the linear convergence of function values in (i) follows immediately. Now, let us prove the strong convergence of the sequence $(x_k)_{k\in\mathbb{N}}$. By Step I and Step II, we have that $z_k \in B_r(z^*)$  for all $k \geq 0$ and therefore for arbitrary $l \in \mathbb{N}$ 
\begin{align} 
    \|x_{k+l} - x_{k} \|_X^2   
    &\leq 
    \sum_{j=k}^{k+l} n_j \left( \sum_{i=0}^{n_{j}-1}  \|z_{j}^{i+1} - z_{j}^{i} \|_Z^2 + \|u_{j+1} - u_{j}\|_{H_0^1(\Omega)}^2 \right) \notag\\
    &\leq
    \sum_{j=k}^{k+l}n_j \left(J(x_j) - J(x^*)\right) \notag\\
    &\lesssim \left(J(x_{0}) - J(x^*)\right)
    \sum_{j=k}^{k+l}j^{2\gamma} Q_1^{j}  \notag \\
    &= \left(J(x_{0}) - J(x^*)\right) Q_1^{k}
    \sum_{j=0}^{l}(j + k)^{2\gamma} Q_1^{j} \notag \\
    & \lesssim \left(J(x_{0}) - J(x^*)\right) Q_1^{\frac{k}{2}} Q_1^{\frac{k}{2}} k^{2\gamma}
    \sum_{j=0}^{+\infty}(j + 1)^{2\gamma} Q_1^{j} \to 0 \quad \text{as $k \to \infty$}, \label{eq:linear_conv_local}
\end{align}
where we used again the definition of the accuracy $\eta_k \lesssim k^{-\gamma} $. Hence $(x_k)_{k \in \mathbb{N}}$ is a Cauchy sequence in $H_0^1(\Omega) \times Z$ and converges to some element $\tilde{x} \in U_{ad} \times Z$ strongly. By continuity of $J$ we deduce $J(\tilde{x}) = J(x^*)$. The linear convergence in (iii) follows from \eqref{eq:linear_conv_local}, by sending $l \to \infty$ and using $C = \sup_k Q_1^{\frac{k}{2}} k^{2\gamma}
    \sum_{j=0}^{+\infty}(j + 1)^{2\gamma} Q_1^{j}$ and $Q_2 = Q_1^{1/2}$ in \eqref{eq:global_theorem_statement_2}.
\end{proof}
\begin{remark}
Let us briefly comment on the implications of \autoref{theorem_local_conv}, which suggest avenues for future research and potential algorithmic improvements. The assumption of strong monotonicity of the Hessian of $f^\alpha$, as stated in \autoref{assumption_local}, is classical yet restrictive, though it is a common requirement in second-order and Gauss–Newton methods to guarantee fast local convergence \cite{bertsekasnonlinear,nocedal1999numerical}. 
In the absence of the learning term (i.e., when $h(u, z) + R_2(z) = 0$), local linear convergence is ensured by results such as those in \cite{drusvyatskiy2018error}. However, when the learning term is included, maintaining fast convergence necessitates a larger KL exponent $q > 2$. Practically, the use of alternative discrepancy terms $h(\cdot, \cdot)$ that exhibit sharper local curvature could enhance convergence, although exploring this lies beyond the scope of the present work.
\end{remark}

\section{Regularization and basic stability properties} \label{sec:regularization}
Since data is never perfect, a key question is: in what sense do errors in the input data propagate to the output? This is commonly referred to as data stability. Another important question is whether the expected solution can be recovered under vanishing noise level, $\delta \to 0_+$. Without aiming to provide a complete answer to these questions, we argue that the proposed method yields a stable regularization scheme under rather strong assumptions. A more detailed analysis is deferred to future work. For a comprehensive account on these topics in the setting of imaging we refer to \cite{scherzer2009variational}.
We also note that even the class of bilinear inverse problems—such as blind deconvolution \cite{burger2001regularization} and classical linear dictionary learning \cite{ravishankar2010mr,ravishankar2013sparsifying,ravishankar2015efficient}— as well as the general topic of regularization via unsupervised learning \cite{tachella2022unsupervised,tachella2023sensing}, form an active and intriguing area of research and are not fully understood, even when dealing with linear inverse problems. Thus, we do not attempt a detailed discussion of these topics here.
The following two theorems represent a first step toward showing that the methodology presented here yields a convergent and stable regularization scheme. For these results recall the functional from \eqref{eq:p0}:
\begin{align}
J(x;f,\alpha,\lambda,\beta) :=
\frac{1}{2} \| F(u) - f^\delta \|^2_{L^2(\Omega,\mathbb{C}^L)} + \frac{\alpha}{2} \|\nabla u \|^2_{L^2(\Omega)} +  \frac{\lambda} {2} \left( \| P \mathbb{D}^\mathfrak{h} u - D C \|_F^2 +  \beta \|C\|_1  \right)
\label{eq:def_functional_reg}
\end{align}
for $x = (u,D,C)$. Also recall the definition of the Moreau--Yosida approximation of a convex function $g:H \to \mathbb{R}$ defined on a Hilbert space $H$, cf. \cite{attouch2014variational}:
\begin{equation}
g_\gamma(u) := \min_{v \in X} \frac{1}{2\gamma}\| u - v \|^2_2 + g(v). \label{eq:def:moreau_yosida}\end{equation}
It is a smooth approximation of $g$,  which is also referred to as Huber regularization (see \cite{huber1992robust}) in the context of $g = \|\cdot \|_1$. In the latter case we will write $\|\cdot\|_{1,\gamma}$ for the Moreau Yosida regularized $\ell_1$ norm. Then the following theorem holds true:
\begin{theorem}[Regularization]\label{thm:regularization_thm}
Consider Problem \eqref{eq:p0} in space dimension $d=2$ and a vanishing noise level, i.e., $\delta_n \to 0_+$. Let ${x_n=(u_n, D_n, C_n)_n \subset U_{ad} \times \mathcal{D} \times \mathcal{C}}$ be global solutions of the optimization problem $\eqref{eq:p0}$:
\begin{equation}
    (u_n,D_n,C_n) \in \argmin_{x=(u, D, C) \in U_{ad} \times \mathcal{D} \times \mathcal{C}} J(x;f^{\delta_n},\alpha_n,\lambda_n,\beta) \quad \forall \, n \in \mathbb{N}
\end{equation}
for data $f^{\delta_n}$ with  regularization parameters $\alpha_n, \lambda_n > 0$ and constant $\beta>0$. Suppose the sequence of regularization parameters satisfies
\begin{equation} 
\lambda_n = \kappa\alpha_n, \quad \alpha_n \to 0, \quad \frac{\delta_n^2}{\alpha_n} \to 0 \quad \text{as } n \to \infty, \end{equation} 
for some $\kappa > 0$. Then, after possibly passing to a subsequence, we have $u_n \rightharpoonup u^* \in H^1(\Omega)$ and $D_n \to D^*$, where $(u^*, D^*)$ is a solution to the problem \begin{equation}
\min_{D \in \mathcal{D}, u \in U_{ad}} \frac{1}{2}\|\nabla u \|^2_{L^2(\Omega)} + \kappa \beta \| D^T \mathbb{D}^\mathfrak{h} u \|_{1,\beta} \quad \text{s.t. } \, F(u) = F(u_{true}). \label{eq:minimum_norm_problem} 
\end{equation}
\end{theorem}
\begin{proof}
During the proof we denote $v = (u,D)$. The proof uses  techniques as  presented in \cite{scherzer2009variational}.
We first argue that \eqref{eq:minimum_norm_problem} has a solution. Denote $\mathcal{R}(v) := \frac{1}{2}\|\nabla u \|^2_{L^2} + \kappa \beta \|D^T  \mathbb{D}^\mathfrak{h} u \|_{1,\beta}$. Since $\mathcal{R}\geq 0 $, we may pick an infimizing sequence $v_n \in U_{ad} \times \mathcal{D}$. As $\mathcal{D}$ is compact and clearly $\|\nabla u_n \|_{L^2(\Omega)}$ is bounded, we find by standard compact embedding theorems that $u_n \rightharpoonup u^*$ in $H^1(\Omega)$, $u_n \to u^*$ in $L^p(\Omega)$ for some $p >2$ and $D_n \to D^*$ in $\mathcal{D}$ after passing to a suitable subsequence and relabeling. Moreover by the continuity of $F:L^p(\Omega) \to L^2(\Omega)$ we infer $F(u^*)=F(u_{true})$ as all $v_n$ satisfy the constraints. In addition, as $U_{ad}$ is clearly closed in $L^p(\Omega)$, it follows that $(u^*,D^*)$ is a global solution of \eqref{eq:minimum_norm_problem} by classical lower semi-continuity arguments. \\
Let us now prove the convergence result. First note that as $(u_n,D_n,C_n)$ is a global solution for every $n$, we may eliminate $C_n$ from the objective and obtain 
\begin{align*}
    (u_n,D_n) \in \argmin_{(u,D) \in U_{ad} \times \mathcal{D}} J_n(u,D) := \frac{1}{2} \| F(u) - f^{\delta_n} \|_{L^2(\Omega,\mathbb{C}^L)}^2 + \frac{\alpha_n}{2} \| \nabla u \|_{L^2(\Omega)}^2  +  \lambda_n \beta \|D^T \mathbb{D}^\mathfrak{h} u\|_{1,\beta}. 
\end{align*}
for every $n\in \mathbf{N}$. Here, we used the orthogonality of $D_n$ and the definition of the Moreau envelope \eqref{eq:def:moreau_yosida}. We then derive the bounds
\begin{align}
\begin{aligned}
    J_n(v_n) 
    &=
    \frac{1}{2} \| F(u_n) - f^{\delta_n} \|_{L^2(\Omega,\mathbb{C}^L)}^2 + \frac{\alpha_n}{2} \| \nabla u_n \|_{L^2(\Omega)}^2  +  \lambda_n \beta \| D_n^\top P \mathbb{D}^\mathfrak{h}u_n\|_{1,\beta} \\
     &\leq \frac{1}{2} \| F(u^\dagger) - f^{\delta_n} \|_{L^2(\Omega,\mathbb{C}^L)}^2 + \frac{\alpha_n}{2} \| \nabla u^\dagger \|_{L^2(\Omega)}^2  +  \lambda_n \beta \| D^{\dagger,\top} P \mathbb{D}^\mathfrak{h}u^\dagger\|_{1,\beta} \\
     &\leq \frac{1}{2}\delta_n^2 + \alpha_n\left( \frac{1}{2} \| \nabla u^\dagger \|_{L^2(\Omega)}^2  +  \kappa \beta \| D^{\dagger,\top} P \mathbb{D}^\mathfrak{h}u^\dagger\|_{1,\beta}\right) \to 0 \quad \text{as $n \to \infty$}, \label{eq:reg_thm_1}
    \end{aligned}
\end{align}
for every, so called \emph{minimum norm solution}  $v^\dagger$ of \eqref{eq:minimum_norm_problem},  where we used the assumption of the theorem. Henceforth, we obtain by \eqref{eq:reg_thm_1} and triangle inequality 
\begin{align}
    \| F(u_n) - F(u_{true})\| &\leq \|F(u_n) - f^{\delta_n}\| + \delta_n \to 0   \quad  (n \to \infty) \label{eq:reg_thm_2_1}  \\
    \left( \frac{1}{2} \| \nabla u_n \|_{L^2(\Omega)}^2  +  \kappa \beta \| D^\top P \mathbb{D}^\mathfrak{h}u_n\|_{1,\beta}\right) 
    &\leq 
    \frac{\delta_n^2}{2\alpha_n} + 
    \left( \frac{1}{2} \| \nabla u^\dagger \|_{L^2(\Omega)}^2  +  \kappa \beta \| D^{\dagger,\top} P \mathbb{D}^\mathfrak{h}u^\dagger\|_{1,\beta}\right). \label{eq:reg_thm_2_2}
\end{align}
As before, we conclude from \eqref{eq:reg_thm_1} that $\|u_n\|_{H_0^1(\Omega)}$ is bounded. By compactness of $\mathcal{D}$, after passing to a subsequence, using suitable embedding theorems for dimension $d=2$ we may assume that $u_n \to \tilde{u}$ in $L^p(\Omega)$ for some $p>2$, $u_n \rightharpoonup \tilde{u}$ in $H^1(\Omega)$ and $D_n \to \tilde{D}$ for some $\tilde{v}:=(\tilde{u},\tilde{D}) \in U_{ad} \times \mathcal{D}$, where we also used the $L^p(\Omega)$ closedness of $U_{ad}$. Using the the continuity of $F:L^p(\Omega) \to L^2(\Omega)$ we infer from \eqref{eq:reg_thm_2_1} that $F(\tilde{u}) = F(u_{true})$ . Moreover by sequential lower semi continuity of $v \mapsto \mathcal{R}(v)$ we deduce
\begin{align*}
    \mathcal{R}(\tilde{v}) \leq \liminf_{n \to \infty} \mathcal{R}(v_n) 
    &\leq \limsup_{n \to \infty } \mathcal{R}(v_n) \\
    &\leq\limsup_{n \to \infty}   \left(\frac{\delta_n^2}{2\alpha_n} + 
    \frac{1}{2} \| \nabla u^\dagger \|_{L^2(\Omega)}^2  +  \kappa \beta \| D^{\dagger,\top} P \mathbb{D}^\mathfrak{h}u^\dagger\|_{1,\beta}\right) \\ 
    &= \mathcal{R}(v^\dagger)
\end{align*}
for any solution $v^\dagger$ of  \eqref{eq:minimum_norm_problem}. Here we used \eqref{eq:reg_thm_2_2} in the second step. In conclusion, $\tilde{x}$ must be a solution of \eqref{eq:minimum_norm_problem} as well.
\end{proof}
\noindent
The following theorem provides a stability result with respect to the measured noise data $f^\delta$.

\begin{theorem}\label{thm:stability_thm}
    As before, fix space dimension $d = 2$, and consider Problem~\eqref{eq:p0} with a sequence of noisy data  $(f_n)_n \subset L^2(\Omega, \mathbb{C}^L)$ such that $f_n \to f$ in $L^2(\Omega, \mathbb{C}^L)$ as $n \to \infty$. Consider a sequence of global minimizers of~\eqref{eq:p0}, i.e.
    \begin{equation}
    (u_n,D_n,C_n) \in \argmin_{x=(u, D, C) \in U_{ad} \times \mathcal{D} \times \mathcal{C}} J(x;f_n,\alpha,\lambda,\beta) \quad \forall \, n \in \mathbb{N},
\end{equation}
where the regularization parameters $(\lambda, \beta, \alpha) \in \mathbb{R}_+^3$ are now fixed. Then, up to a subsequence,
\begin{equation}
u_n \rightharpoonup u^* \text{ in } H^1(\Omega), \quad D_n \to D^* \in \mathcal{D}, \quad C_n \to C^* \in \mathcal{C}, \notag
\end{equation}
where $(u^*, D^*, C^*) \in U_{ad} \times \mathcal{D} \times \mathcal{C}$ is a global minimizer for data $f$:
\begin{equation}
    (u^*,D^*,C^*) \in \argmin_{x=(u, D, C) \in U_{ad} \times \mathcal{D} \times \mathcal{C}} J(x;f,\alpha,\lambda,\beta),
\label{eq:stability_thm_1}
\end{equation}
\end{theorem}
\begin{proof}
    The proof follows the lines of the corresponding section in  \cite{scherzer2009variational} and uses very similar steps  as the proof of the previous \autoref{thm:regularization_thm}. We therefore only  sketch the argument: We use coercivity to know that $(u_n)_n$ is uniformly bounded in $H_0^1(\Omega)$. $(D_n)_n$ is bounded because it is contained in the compact set $\mathcal{D}$ and $(C_n)_n$ is bounded by the coercivity of $\|\cdot\|_1$. Hence we may assume after passing to a subsequence that ${(u_n, D_n, C_n)_n \subset U_{ad} \times \mathcal{D} \times \mathcal{C}}$ has an accumulation point $({u^*, D^*, C^*) \in U_{ad} \times \mathcal{D} \times \mathcal{C}})$, where the convergence in $u$ is with respect to the strong topology of $L^p(\Omega)$ and the weak topology in $H^1(\Omega)$, due to compact embedding. By  continuity of $F:L^p(\Omega) \to L^2(\Omega)$ and classical lower semi-continuity arguments we eventually conclude \eqref{eq:stability_thm_1}.
\end{proof}
\begin{remark}[Stability and Recovery]
The preceding two theorems offer only preliminary insight and fall short of addressing the full range of practical challenges. Therefore, we briefly highlight additional issues that arise in practical applications.
\begin{enumerate}
    \item Regarding recovery, a key question is whether the minimum-norm solution of \eqref{eq:minimum_norm_problem} can actually yield a meaningful reconstruction of the true solution $u_{\text{true}}$. In the regularization theory of inverse problems, this question is typically addressed using source conditions (cf. \cite{kaltenbacher2008iterative, scherzer2009variational}), albeit in many practical problems, such conditions are hard to verify. More recent works refer to this issue as the identifiability of the underlying model. In the context of dictionary learning in finite dimensions, this question has been studied recently (see, e.g., \cite{gribonval2010dictionary,cohen2019identifiability,hu2023global}). However, in nonlinear settings, especially in infinite-dimensional spaces, much less is known.
    \item  Regarding stability, we focus here solely on stability with respect to noise in the measurement data. However, it is also important to consider stability with respect to other parameters. In practical applications, particular attention should be given to parameters that affect the underlying model itself, such as flip angles and repetition times. Additionally, factors such as the sampling pattern and the choice of the initial guess can significantly impact the reconstruction quality. Moreover, in \autoref{thm:regularization_thm} and \autoref{thm:stability_thm}, we assume access to global solutions of the underlying nonconvex optimization problem. In practice, however, one typically obtains only stationary points. Relaxing these assumptions and analyzing the model in its full generality including convergence rates is an important topic for future research.
\end{enumerate}

\end{remark}
\section{Numerical experiments}
We present numerical results for the nested alternating algorithm applied to the qMRI problem \eqref{eq:p0}. Spatial discretization follows the finite difference scheme from \cite{dong2019quantitative}. The update steps for $z=(D,C)$ are based on blind dictionary learning methods from \cite{ravishankar2012learning,ravishankar2015efficient}, originally developed for linear MRI. For reference, the algorithm and key properties are summarized in \cref{alg:algorithm_dictionary} (appendix, \cref{alg:algorithm_dictionary}). Ground truth parameters used in our experiments are shown in \autoref{fig:ground_truth}.

\begin{figure}[h!]
    \centering
    \begin{minipage}[b]{0.34\textwidth}
    \centering
    \includegraphics[width=1\linewidth]{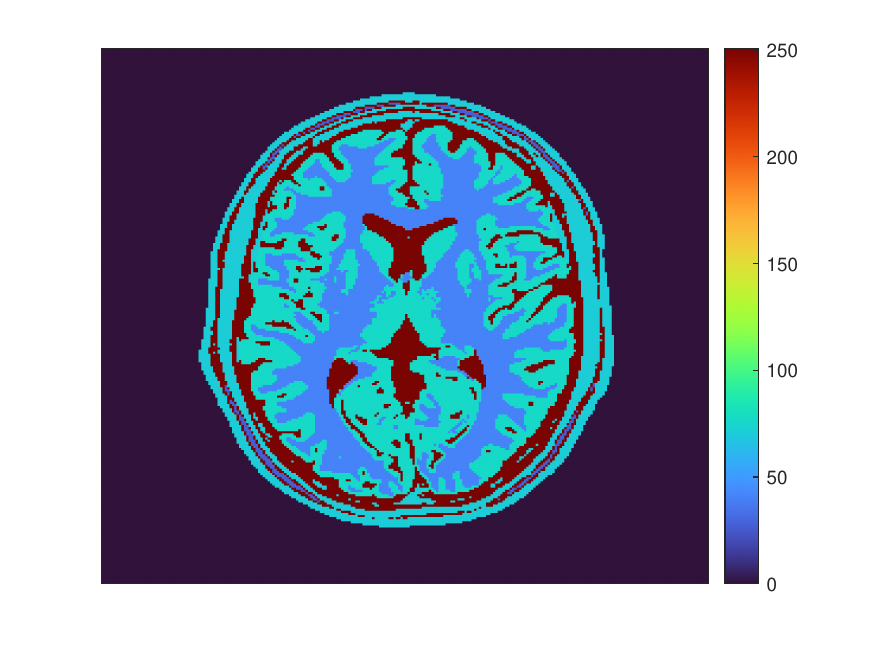} \\[-0.5cm]
    $T_1$-map.
    \end{minipage} 
    \hspace{-0.5cm}
    \begin{minipage}[b]{0.34\textwidth}
    \centering
    \includegraphics[width=1\linewidth]{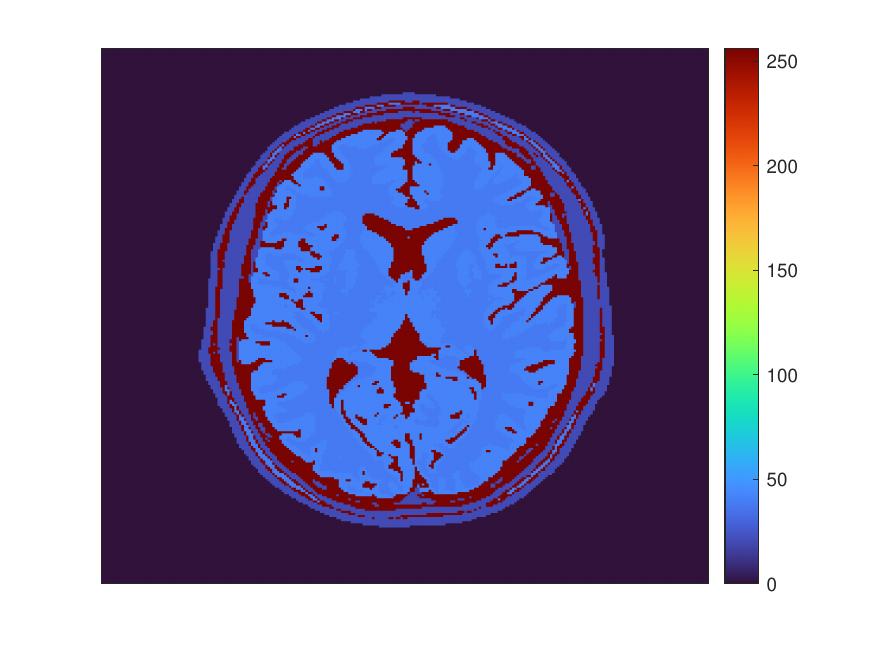}
     \\[-0.5cm]
    $T_2$-map
    \end{minipage}
    \hspace{-0.5cm}
    \begin{minipage}[b]{0.34\textwidth}
    \centering
    \includegraphics[width=1\linewidth]{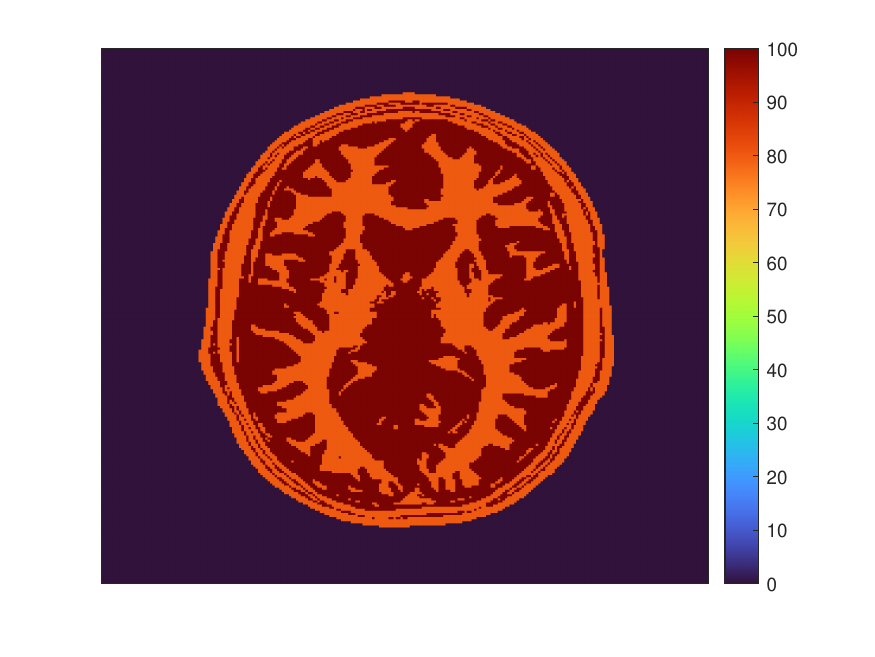}
     \\[-0.5cm]
    $\rho$-map
    \end{minipage}
    \caption{The set of ground truth parameters $T_1,T_2$ (milliseconds) and proton density $\rho$ (dimensionless).}
\end{figure}\label{fig:ground_truth}
\noindent
The ground truth images (physical parameters) follow a standard configuration commonly used for testing qMRI methods, and has been employed in previous studies, such as \cite{davies2014compressed, dong2019quantitative}. To streamline the parameter search and improve the conditioning of the subproblems, we scaled both the \(T_1\) and \(T_2\) variables to the range \([0,250]\) ms. Although these ranges differ from typical relaxation times, which lie in the ranges \(T_1 \in [0,6000]\) ms and \(T_2 \in [0,600]\) ms, this adjustment provides a simplified setting that effectively demonstrates the performance of our algorithm while still being relevant for practical applications.
\begin{figure}[h!]
    \centering
    \begin{minipage}[b]{0.22\textwidth}
    \centering
    \includegraphics[width=1\linewidth]{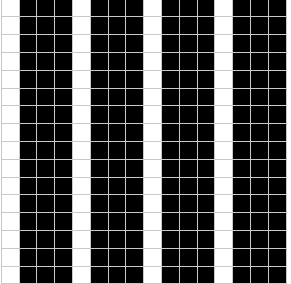} \\
    $\mod(t,4) = 1$
    \end{minipage} 
    \hspace{0.2cm}
    \begin{minipage}[b]{0.22\textwidth}
    \centering
    \includegraphics[width=1\linewidth]{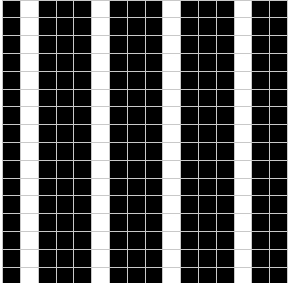}
     \\
    $\mod(t,4) = 2$
    \end{minipage}
    \hspace{0.2cm}
    \begin{minipage}[b]{0.22\textwidth}
    \centering
    \includegraphics[width=1\linewidth]{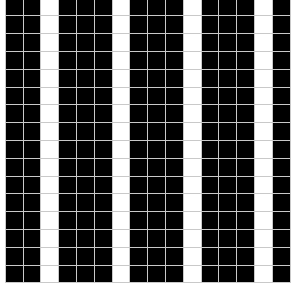}
     \\
    $\mod(t,4) = 3$
    \end{minipage}
    \hspace{0.2cm}
    \begin{minipage}[b]{0.22\textwidth}
    \centering
    \includegraphics[width=1\linewidth]{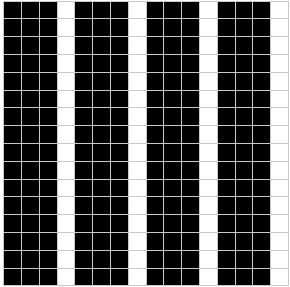}
     \\
    $\mod(t,4) = 0$
    \end{minipage}
    \caption{Schematic illustration of the equidistant Cartesian sampling mask used in the experiment across time steps $t = 1, \ldots, L$, shown on an example $16 \times 16$ image. White pixels indicate sampled frequency components. At each subsequent time step, the sampling pattern shifts one line to the right. A $4\times$ undersampling factor is shown for illustration; in the actual experiments, $16\times$ and $32\times$ undersampling are used. 
}
\end{figure}\label{fig:sampling_masks}
\paragraph{The discrete optimization problem}  For our numerical tests, we employ a uniform grid and a finite differences discretization for the variable \( u= (\rho,T_1,T_2) \). This approach leads us to consider the space \( U := \mathbb{R}^{n_1 \times n_2 \times 3} \), where \( n_1, n_2 \in \mathbb{N} \) denote the number of pixels in each direction. In our case, we set \( n_1 = n_2 = 256 \). We also make use of the classical finite difference approximations for the image gradient, denoted by \( \nabla^\mathfrak{h} : \mathbb{R}^{n_1 \times n_2} \to \mathbb{R}^{n_1 \times n_2 \times 2} \), and the Laplace operator with zero boundary conditions, denoted by \( \Delta^\mathfrak{h} :  \mathbb{R}^{n_1 \times n_2} \to \mathbb{R}^{n_1 \times n_2} \), with a mesh-size $\mathfrak{h}>0$. If $\nabla^\mathfrak{h}$ and $\Delta^\mathfrak{h}$ are applied on elements of $U$, we will use the same notation and apply the operators component wise, i.e.
\begin{align*}
    &\tilde{\nabla}^\mathfrak{h} : U \to U_1 := \mathbb{R}^{3 \times n_1 \times n_2 \times 2 }, 
    &\tilde{\nabla}^\mathfrak{h} u = (\nabla^\mathfrak{h} \rho, \nabla^\mathfrak{h} T_1, \nabla^\mathfrak{h} T_2 ),   \\
    &\tilde{\Delta}^\mathfrak{h} : U \to U,  
     &\tilde{\Delta}^\mathfrak{h} u = (\Delta^\mathfrak{h} \rho, \Delta^\mathfrak{h} T_1, \Delta^\mathfrak{h} T_2 ),
\end{align*}
for $ u = (\rho,T_1,T_2) \in U $. 
We equip both spaces, $U,U_1$, with the following scaled norms, defined by:
\begin{align}
     \|u\|_U^2 &:= \frac{h^2}{M_1^2}\| \rho \|^2_{2} + \frac{h^2}{M_2^2} \|T_1\|^2_2 + \frac{h^2}{M_3^2} \|  T_2\|^2_2, \\
     \|v\|_{U_1}^2 &:= \frac{h^2}{M_1^2}\| v_1 \|^2_{2} + \frac{h^2}{M_2^2} \| v_2 \|^2_2 + \frac{h^2}{M_3^2} \|  v_3 \|^2_2, \label{eq:discrete_norms}
\end{align}
where we used $M_1 = 100$ and $M_2 = M_3 = 250$ as scaling parameters. This results in the following discrete minimization problem:
\begin{align}
    \min_{(u,D,C) \in X}  J_d(u,D,C) :=  \frac{h^2}{2}\|F_d (u) - f^\delta\|_{2}^2 
    & + \frac{\alpha}{2} \| \nabla u \|_{U_1}^2 + \mathcal{I}_{U_{ad}}(u)    \notag \\ 
    & + 
    \sum_{i=1}^3 \lambda^j \left(\frac{1}{2}\|  P [\frac{1}{M_j} u_j ]  - D_jC_j \|_F^2  + \beta_j \|C_i \|_1\right), \tag{$P_2$}\label{eq:P_discrete}
\end{align}
with $X := U_{ad} \times O_K^3 \times \mathbb{R}^{M\times K \times 3}$ and
 \begin{equation}
     U_{ad} = \left\lbrace u= (\rho,T_1,T_2) \in U \middle\vert 
     \begin{array}{l}
     \rho_{ij} \in [0,110] , (T_1)_{ij} \in [0,300], (T_2)_{ij} \in [0,300]  \\
     \text{for all $1\leq i \leq n_1$ and $1\leq j \leq n_1$}
     \end{array}
     \right \rbrace .
 \end{equation}
As for the continuous forward operator, the discrete counterpart  $F_d:U \to \mathbb{C}^{n_1 \times n_2 \times L}$ in \eqref{eq:P_discrete} is defined as the composition $F_d = A\circ \Pi_d$ with the discrete Bloch solution operator given by 
\begin{equation}
    \Pi_d: U \to \mathbb{C}^{n_1 \times n_2 \times L} \quad [\Pi_d(u)]_{ijl} = \pi(u_{ij})_l, \quad \text{$1 \leq l \leq L$}, \notag
\end{equation}
where $\pi: \mathbb{R}^3 \to \mathbb{C}^L$ is the function defined in \eqref{eq:def_pi}. Moreover, the linear operator $A$, modelling the observation process, is defined by 
\begin{equation}
    A: \mathbb{C}^{n_1 \times n_2 \times L} \to \mathbb{C}^{n_1 \times n_2 \times L}, \quad [A y]_{ijl} = S_l \mathcal{F}[y_l], \quad \text{$1 \leq l \leq L$}.\label{eq:def_A_discrete}
\end{equation}
In the definition above, $\mathcal{F}: \mathbb{C}^{n_1 \times n_2}  \to \mathbb{C}^{n_1 \times n_2} $ denotes the normalized, discrete 2D-Fourier transform and $S_l: \mathbb{C}^{n_1 \times n_2} \to \mathbb{C}^{n_1 \times n_2} $ denotes a predefined sampling pattern which acts on the $l$-th magnetization slice as 
\begin{equation}
 S_l(y)_{i,j} = 
 \begin{cases}
     y_{ij} &\text{if frequency $y_{ij} \in \mathbb{C}$ is sampled,} \\ 
     0 &\text{if $y_{ij}$ is not sampled.}
 \end{cases} \notag
\end{equation}
 Regarding the sampling-pattern, we follow exactly the setting in \cite{dong2019quantitative}, it is also depicted in \autoref{fig:sampling_masks}. Moreover, we make use of the linear patch extraction operator $P$, which cuts out small image patches and puts them into a large matrix. More precisely, we define:
\begin{equation}
    P : \mathbb{R}^{n_1 \times n_2} \to \mathbb{R}^{M \times K} \quad P u =  
    \begin{bmatrix}
        P_{11}u , P_{21}u , \ldots P_{n_1 1}u,
        P_{1 2}u , P_{2 2}u \ldots P_{n_1 2}u , \ldots, 
        P_{1 n_2}u , \ldots P_{n_1 n_2}u  
    \end{bmatrix}. \notag
\end{equation}
for $j = 1,2,3$. Here \( P_{kl}: \mathbb{R}^{n_1 \times n_2} \to \mathbb{R}^K \) is an operator that extracts a patch of size \( p \times p \) from an image \( u \in \mathbb{R}^{n_1 \times n_2} \), where the top-left corner of the patch is located at pixel \((k, l)\). The extracted patch is then vectorized into a column vector of size \( K = p^2 \). We use overlapping patches, as described in \cite{ravishankar2010mr,ravishankar2015efficient} such that exactly $M = n_1 \cdot n_2$ can be extracted.  The normalization factor \( {1}/{M_j} \) is introduced because, empirically, we observed improved reconstruction quality for the dictionary learning problem when the data is normalized. Additionally, this normalization significantly simplifies the process of hyperparameter tuning.
The overall discrete version of \cref{alg:algorithm_general} is given by \cref{alg:algorithm_overall_discrete}.
\begin{algorithm}[htbp]
\caption{Computation of a stationary point of Problem \eqref{eq:P_discrete}.}
\begin{algorithmic}[1]
\State Choose initial values $(u_0, D_0,C_0) \in U_{ad} \times O_K^3 \times \mathbb{R}^{K \times M \times 3}$, the parameter $\gamma>0$ to define the stopping tolerance for the nested subroutine, the mesh size $\mathfrak{h}>0$, the step-size parameters for the $u$-step $\lambda_0 >0, \tau >1$ and $\sigma_{BT} \in (0,1)$, $\varepsilon_1, \varepsilon_2 > 0$. 
\State Set $k$ = 0.
\While{{$\| u_{k} - u_{k-1} \|^2_{U} + \| \nabla^\mathfrak{h} u_{k} - \nabla^\mathfrak{h} u_{k-1} \|_{U_1}^2 \geq \varepsilon_1^2 $ or $\| D_{k} - D_{k-1} \|^2_{F} + \| C_{k} - C_{k-1} \|^2_{F} \geq \varepsilon_2^2$}} 
	\State \textbf{Dictionary-learning-step}: Given $(u_k,D_k,C_k) \in U \times O_K^3 \times \mathbb{R}^{M \times K \times 3}$ 
    \State \textbf{for} $j=1,2,3$ \textbf{do}: 
    \State \hspace{1cm} Use \cref{alg:algorithm_dictionary} with initialization $D_k^0 := (D_k)_j$, $C_k^0 := (C_k)_j$ and 
    \State \hspace{1cm} stopping tolerance $\eta_k := k^\gamma \sqrt{\|C_0\|_F^2 + \|D_0\|_F^2}$ for the problem :
    \begin{equation}
    \min_{D \in O_K,C \in \mathbb{R}^{M \times K}} \frac{1}{2} \| DC - P[ \frac{1}{M_j} u_j] \|_{F}^2 + \frac{\beta_j}{\lambda^j} \|C\|_1, \notag 
\end{equation}
\State \hspace{1cm} to obtain $(D_{k+1})_j \in O_K$ and $(C_{k+1})_j \in \mathbb{R}^{K \times M}$. 
\State \textbf{end for}
    \State \textbf{$\mathbf{u}$-step}: Given $(u_k, D_{k+1},C_{k+1}) \in U \times O_K^3 \times \mathbb{R}^{M \times K \times 3}$.
    \State \textbf{for} $j= 1,\ldots$ \textbf{do}: 
    \State \hspace{1cm} Take $\lambda_0>0$, set $\lambda_k := \lambda_0 \tau^j$ 
     and compute a global solution $\widehat{u}(\lambda_k) \in U_{ad}$ of 
    \begin{equation}
        \min_{u \in U_{ad}} g^d_{\lambda_k}(u,u_k) + \frac{\alpha}{2} \| \nabla^\mathfrak{h} u \|_{U_1}^2 + \sum_{j=1}^3 \frac{\lambda^j}{2}\|P[\frac{1}{M_j} u ]- (D_{k+1})_j (C_{k+1})_j \|_F^2  , \label{eq:discrete_subpropblem_u}
    \end{equation} 
    \State \hspace{1cm}  where $g^d_{\lambda_k}(\cdot,u_k):U \to \mathbb{R}$ is defined as the discrete analogue of the model in \eqref{eq:def_model_function_u} as
    \begin{align*}
        g^d_{\lambda_k}(u,u_k) := \frac{\mathfrak{h}^2}{2}\|F_d'(u_k)[u-u_k] 
        &+ F(u_k) - f^\delta \|_{2}^2  \\
        &+ \frac{\lambda_k}{2}\left( \| u - u_k \|^2_{U} + \|\nabla^\mathfrak{h}( u - u_k) \|^2_{U_1} \right),
    \end{align*}
    \State \hspace{1cm} until the descent condition
    \begin{align}
        J_d(\widehat{u}(\lambda_k),D_{k+1},C_{k+1}) & \leq J_d(u_k,D_{k+1},C_{k+1}) \notag \\
        & - \frac{\sigma_3 \lambda_k}{2} \left( \|\widehat{u}(\lambda_k) - u_k\|_{U}^2 + 
        \| \nabla^\mathfrak{h} \widehat{u}(\lambda_k) -  \nabla^\mathfrak{h} u_k\|_{U_1}^2 \right),
    \end{align}
    \State \hspace{1cm} is satisfied.
    \State \textbf{end for}
    \State Set $u_{k+1} = \widehat{u}(\lambda_k)$. 
    \State Set $k = k+1$.
\EndWhile 
\State Return $u_{k} \in U$ as the desired physical parameter.
\end{algorithmic}
\end{algorithm}\label{alg:algorithm_overall_discrete}
\paragraph{Details on the implementation.} The main difficulty in the implementation of \cref{alg:algorithm_overall_discrete} is the solution of the subproblem \eqref{eq:discrete_subpropblem_u}. We use the built-in quadratic programming solver in MATLAB, which relies on the trust-region subspace method,  combining techniques from \cite{coleman1996reflective} and \cite{branch1999subspace}. However, the subproblem is still very delicate to solve, as the Hessian involves the term:
\begin{equation}
     (\Pi_d)'(u_k)^* A^* A (\Pi_d)'(u_k),
    \label{eq:remark_computability_discrete}
\end{equation}
with $A$ as defined in \eqref{eq:def_A_discrete}. Note that if $S_l = Id$ for every $1 \leq l \leq L$, i.e. if no subsampling is applied, then $A^* A = I$. If instead subsampling is used, i.e. if $S_l \neq Id$, $l=1,\ldots,L$, the matrix in \eqref{eq:remark_computability_discrete} is generally dense and of large scale, which complicates the solution procedure of the subproblem significantly. To mitigate this issue, we follow the approximation approach, proposed in \cite{wuebbeler2017large} and replace the term $F'(u_k)^* F'(u_k)$ by the approximation
\begin{equation}
    (\Pi_d)'(u_k)^* A^* A (\Pi_d)'(u_k) \approx \frac{1}{r} (\Pi_d)'(u_k)^* (\Pi_d)'(u_k).\label{eq:hessian_approximation_LM}
\end{equation}
where $r \in \mathbb{N}$ is the undersampling rate. In our experiments in the next paragraph we use $r \in \lbrace 16 , 32\rbrace$. 
Let us comment that, while the quality of this approximation seems well documented in practical applications, a rigorous proof that quantifies potential deviations is still missing. 
\paragraph{Two experiments with different noise levels and undersampling rates.} We test our algorithm on two synthetic test cases. For both we take the ground truth image
$u_{true} \in U$ depicted in \autoref{fig:ground_truth} and simulate noisy data by using the forward operator according to
\begin{equation}
    f^\delta =F_d(u_{true}) + \sigma^2 \mathcal{N}(0,I).
\end{equation}
We utilize complex noise with a standard deviation of $\sigma^2 > 0$ and apply Cartesian subsampling patterns, following the approach in \cite{dong2019quantitative}, cf. also \autoref{fig:sampling_masks}. In the first experiment, we use a $16\times$ undersampling rate and set the noise standard deviation to $\sigma = 2$. In the second experiment, the undersampling rate is increased to $32\times$ and the standard deviation to $\sigma = 5$. For both experiments, we fix $L = 100$, a small yet almost realistic value for practical applications, which ranges typically in $L \in [200,1000]$. Our method is compared against the BLIP reconstruction technique proposed in \cite{davies2014compressed} and the Levenberg-Marquardt method from \cite{dong2019quantitative}. It is important to note that the vanilla Levenberg-Marquardt method guarantees convergence only for zero-residual problems. For other configurations, the outcome is strongly influenced by the number of algorithm steps executed, and convergence to stationary points is typically not expected.. Although numerous well-established stopping criteria exist in the inverse problems literature, we do not apply them in this study. Instead, we manually adjust the number of Levenberg-Marquardt steps for our experiments, to ensure convergence of the method. For the linear system that must be solved at each update step of the Levenberg-Marquardt method, we employ the approximation \eqref{eq:hessian_approximation_LM}. We also compare our algorithm with the dictionary learning algorithm, where, instead of using the nested update procedure, only a single update step for both $D$ and $C$ is performed. We refer to this algorithm as \emph{one-step-dictionary-learning} or \emph{dictionary learning (one step)}. The parameter settings, along with a description of the parameters for both configurations, are presented in \autoref{tab:parameter_choice}.
\paragraph{Results and observations} 
The results of all algorithms are visually presented in \autoref{fig:qMRI_comparisons_dictionary_small_noise} for the small noise case with $16\times$ undersampling and in \autoref{fig:qMRI_comparisons_dictionary_large_noise} for the $32\times$ undersampling with a higher noise level. The relative errors, defined as
\begin{equation}
    \frac{\|X_{\mathrm{reconstruction}} - X_{\mathrm{groundtruth}} \|_2 }{\|X_{\mathrm{groundtruth}}\|_2} \notag
\end{equation}
for $X \in \lbrace \rho, T_1, T_2 \rbrace$ are shown in \autoref{tab:quantitative_results_low_noise} for the low noise regime and in \autoref{tab:quantitative_results_high_noise} for the high noise regime. We observe that, in all experiments, the nested algorithm consistently produces the smallest relative error. While the difference compared to the one-step approach is visually almost unnoticeable, the quantitative values show improvements of up to 10 percent. Additionally, the function values are lower in comparison, as seen in \autoref{fig:function_value_comparison}. 
The results for the $32 \times$ undersampling are particularly promising, as the Levenberg-Marquardt method was unable to produce meaningful results in this scenario. However, one downside of the algorithm is, that the subproblems in \eqref{eq:discrete_subpropblem_u} become very ill-conditioned for small mesh sizes $h > 0$ and step-size parameters $\lambda_k$. This leads to convergence issues, forcing us to set $h = 1$ and to stop the algorithm after 100 iterations in most cases for time reasons, rather than waiting for the stopping criterion in line 19 of \cref{alg:algorithm_overall_discrete} to be met. More tailored algorithms could improve performance in this regard.
\begin{center}
\begin{figure}[htbp]
\centering
\begin{tabular}{|c |p{8.5cm} |c |  }
 \hline
 \multicolumn{3}{|c|}{List of parameters for \cref{alg:algorithm_overall_discrete}} \\
 \hline
 Parameter & \centering Description & Value  \\
 \hline
 $\alpha = (\alpha_1,\alpha_2,\alpha_3) \in \mathbb{R}_{>0}$   &  Regularization parameters for the sparsity of $C$.   & $0.0045*(1,1,1)$ \\
  {}  & {}   & $0.0095*(1,1,1)$ \\
 $\lambda = (\lambda^1,\lambda^2,\lambda^3) \in \mathbb{R}_{>0}$   &  Step-size parameters in \cref{eq:discrete_subpropblem_u}   & $45*(1,1,1)$ \\
 {}   & {}   & $50*(1,1,1)$ \\
 $ (\lambda_D^k,\lambda_C^k) \in \mathbb{R}_{>0} $   &  Step size parameter for the dictionary learning subproblem in \cref{alg:algorithm_dictionary}.    & $(1,1)$ \\ 
  $ \gamma>0$   &  Accuracy parameter for the nested  \cref{alg:algorithm_dictionary}. See also the input of \cref{alg:algorithm_overall_discrete}.    & $0.75$ \\ 
   $ (\lambda_0,\tau,\sigma_3) \in \mathbb{R}_{>0}$   &  Parameters for the backtracking search in \cref{alg:algorithm_overall_discrete}.    & $(1,8,0.5)$ \\ 
   $ (M_1,M_2,M_3) \in \mathbb{R}_{>0}$   & Scaling parameters for the norm in \eqref{eq:discrete_norms}.    & $(100,260,260)$ \\ 
   $ p \in \mathbb{N} $   & Patch size of the squared $p \times p$ patches    & $8$ \\
    $ K \in \mathbb{N} $   & Size of the orthogonal  dictionary $K = p^2$    & $64$ \\
    $ \mathfrak{h} \in \mathbb{R}_{>0} $   & Mesh-size for the differential operators    & $1$ \\
 \hline
\end{tabular} 
\caption{The list and the description of the different parameters in \cref{alg:algorithm_overall_discrete}. The upper value corresponds to the experiment with $16 \times$ undersampling factor and the lower value to the experiment with $32\times$ undersampling. If only one value is provided, we used this value for both experiments.}
\end{figure}\label{tab:parameter_choice}
\end{center}

\begin{center}
\begin{figure}[b!]
\begin{tabular}{|c |c |c |c| c |  }
 \hline
 \multicolumn{5}{|c|}{Numerical results of the algorithm.} \\
 \hline
 {} & BLIP & Levenberg Marquardt &  Dictionary learning (one step) & Dictionary learning (nested) \\
 \hline
 $T_1$ & $0.231$   & $0.155$  & $0.091$ & $\mathbf{0.086}$  \\
 $T_2$  & $0.26$   & $0.177$  & $0.09$ & $\mathbf{0.077}$  \\ 
 $ \rho $  & $0.25$   & $0.222$  & $\mathbf{0.12}$ & $\mathbf{0.12}$  \\
 \hline
\end{tabular}
\caption{Results of the reconstruction algorithm for moderate $16 \times$ undersampling and low noise.}
\end{figure}\label{tab:quantitative_results_low_noise}
\end{center}

\begin{center}
\begin{figure}[b!]
\begin{tabular}{|c |c |c |c| c |  }
 \hline
 \multicolumn{5}{|c|}{Numerical results of the algorithm.} \\
 \hline
 {} & BLIP & Levenberg Marquardt &  Dictionary learning (one step) & Dictionary learning (nested) \\
 \hline
 $T_1$ & $1.195$   & $0.838$  & $0.192$ & $\mathbf{0.184}$  \\
 $T_2$  & $0.733$   & $0.4$  & $0.204$ & $\mathbf{0.185}$  \\ 
 $ \rho $  & $0.236$   & $0.305$  & $0.136$ & $\mathbf{0.134}$  \\
 \hline
\end{tabular}
\caption{Results of the reconstruction algorithm for moderate $32 \times$ undersampling and higher noise.}
\end{figure}\label{tab:quantitative_results_high_noise}
\end{center}

\begin{figure}[h!]
    \centering
    \begin{minipage}[b]{0.45\textwidth}
    \centering
    \includegraphics[width=1\linewidth]{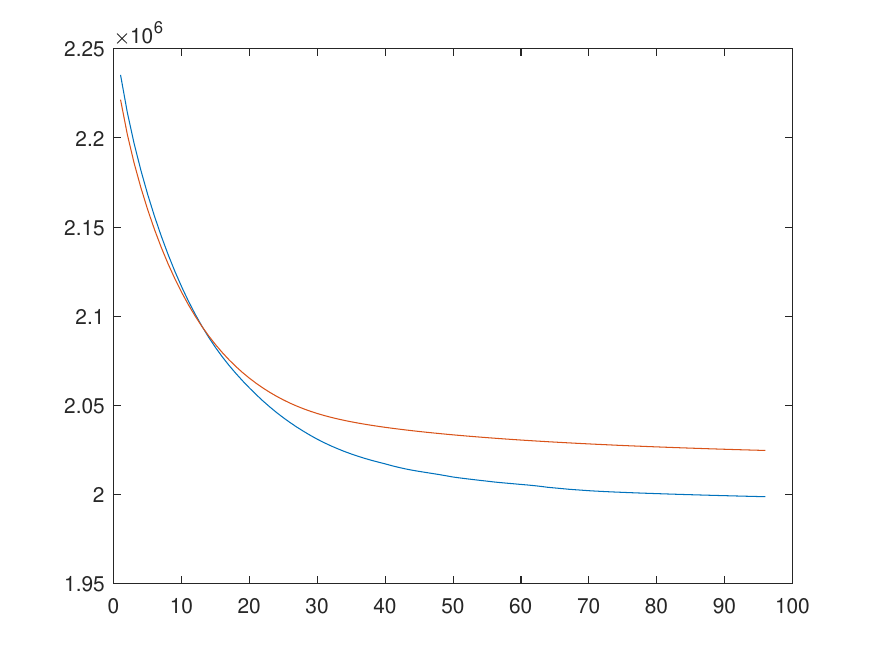}\\[-0.2cm]
    (a)
    \end{minipage}
    \begin{minipage}[b]{0.45\textwidth}
    \centering
    \includegraphics[width=1\linewidth]{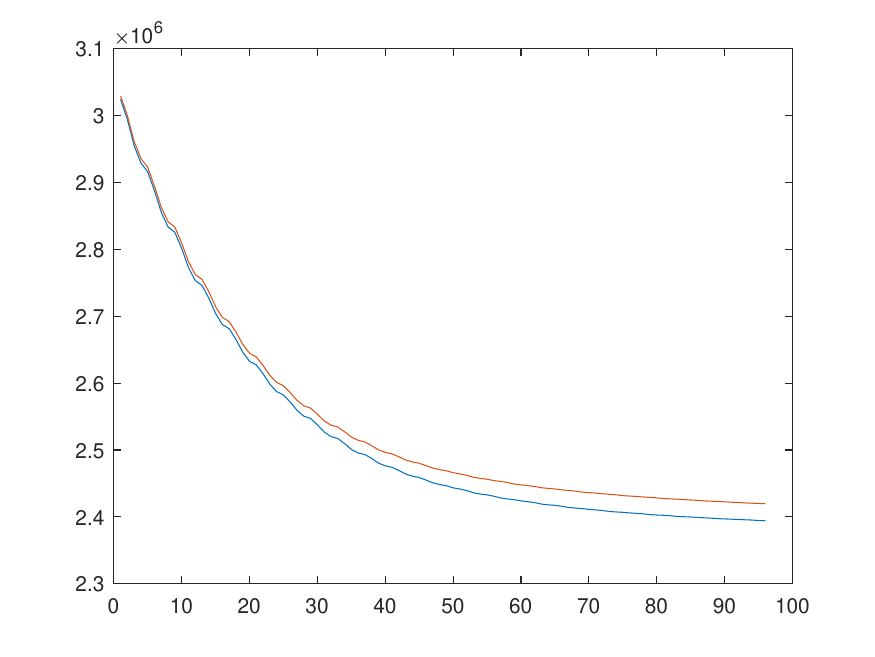}\\[-0.2cm]
    (b)
    \end{minipage}    
    \caption{Comparison of function values over iteration number for the one-step approach (orange) and the nested optimization algorithm (blue). In (a), the graph corresponds to the $16 \times$ undersampling regime, and in (b), to the $32 \times$ undersampling. The nested method demonstrates a faster reduction in function values in both cases.}
\end{figure}\label{fig:function_value_comparison}
\begin{figure}[p!]
 \hspace{0.8cm}
\begin{minipage}[b]{0.15\textwidth}
\centering
Ground truth 
\end{minipage}
\hspace{0.1cm}
\begin{minipage}[b]{0.15\textwidth}
\centering
 BLIP
 \end{minipage}
 \hspace{0.1cm}
 \begin{minipage}[b]{0.15\textwidth}
\centering 
LM
\end{minipage}
\hspace{0.1cm}
\begin{minipage}[b]{0.15\textwidth}
\centering
DL (one step)
\end{minipage}
\hspace{0.1cm}
\begin{minipage}[b]{0.15\textwidth}
\centering
DL (nested) 
\end{minipage}
\vspace{0.2cm}

\begin{minipage}[b]{0.03\textwidth}
\rotatebox{90}{{\hspace{7mm} $T_1$ maps }}
\end{minipage}
\begin{minipage}[b]{0.9\textwidth}
\centering
\includegraphics[width=1\textwidth, trim={1.4cm 0 0 0},clip]{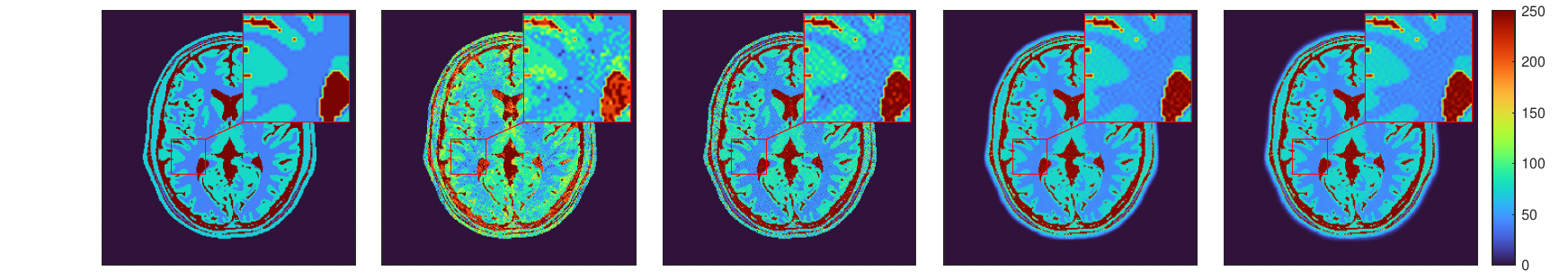} \end{minipage}

\begin{minipage}[b]{0.03\textwidth}
\rotatebox{90}{{\hspace{5mm} $T_1$ Error maps}}
\end{minipage}
\begin{minipage}[b]{0.90\textwidth}
\centering
\includegraphics[width=1\textwidth,  trim={1.4cm 0 0 0},clip]{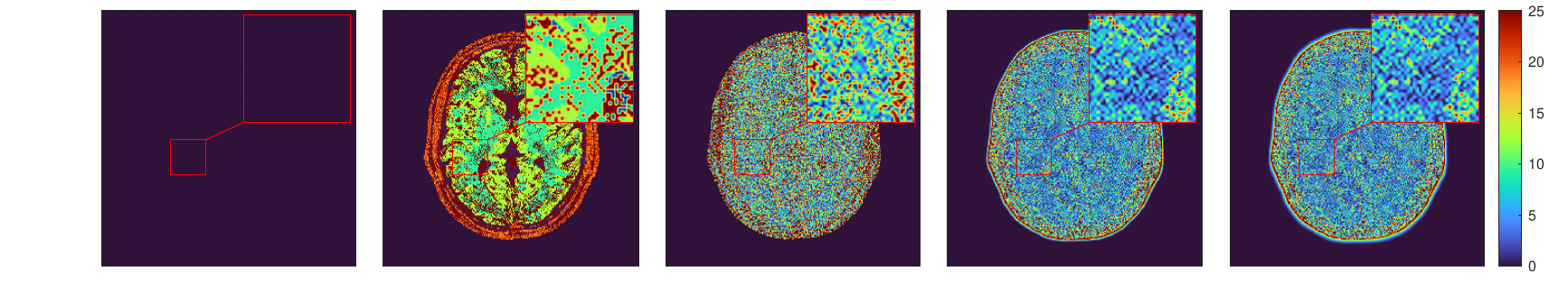}\vspace{0.5em}
\end{minipage}

\begin{minipage}[b]{0.03\textwidth}
\rotatebox{90}{{\hspace{5mm} $T_2$ maps}}
\end{minipage}
\begin{minipage}[b]{0.90\textwidth}
\centering
\includegraphics[width=1\textwidth,  trim={1.4cm 0 0 0},clip]{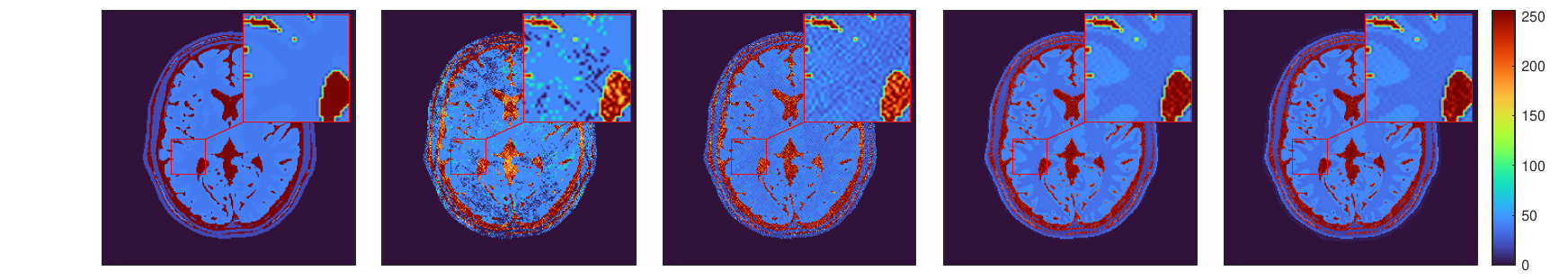} 
\end{minipage}

\begin{minipage}[b]{0.03\textwidth}
\rotatebox{90}{{\hspace{2mm} $T_2$ Error maps}}
\end{minipage}
\begin{minipage}[b]{0.90\textwidth}
\centering
\includegraphics[width=1\textwidth,  trim={1.4cm 0 0 0},clip]{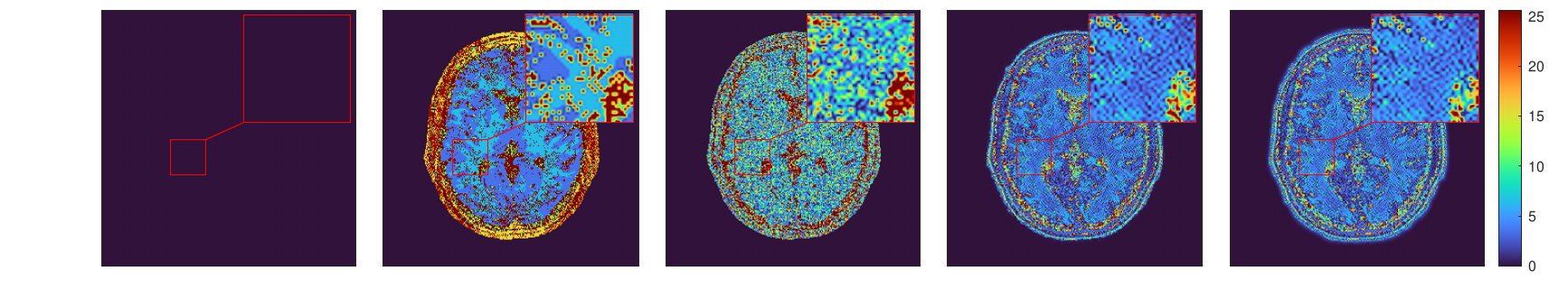} 
\end{minipage}\vspace{0.5em}

\begin{minipage}[b]{0.03\textwidth}
\rotatebox{90}{{\hspace{7mm} $\rho$ maps}}
\end{minipage}
\begin{minipage}[b]{0.90\textwidth}
\centering
\includegraphics[width=1\textwidth,  trim={1.4cm 0 0 0},clip]{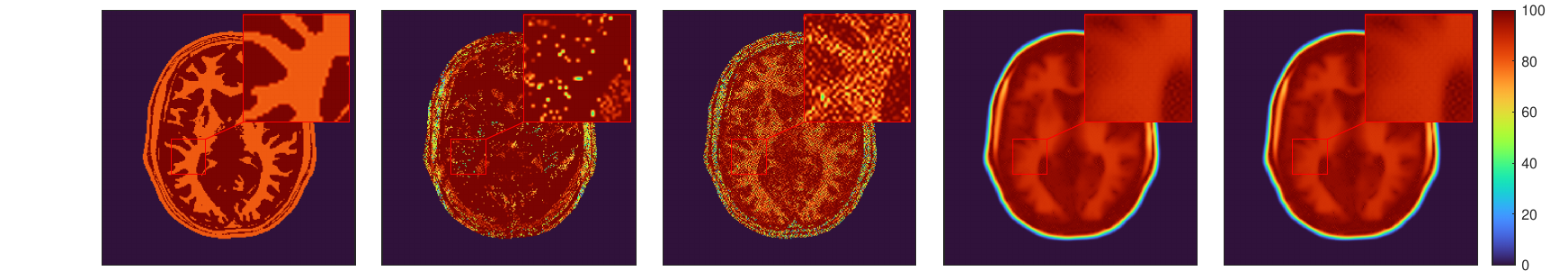} 
\end{minipage}

\begin{minipage}[b]{0.03\textwidth}
\rotatebox{90}{{\hspace{1mm} $\rho$ Error maps}}
\end{minipage}
\begin{minipage}[b]{0.90\textwidth}
\centering
\includegraphics[width=1\textwidth,  trim={1.4cm 0 0 0},clip]{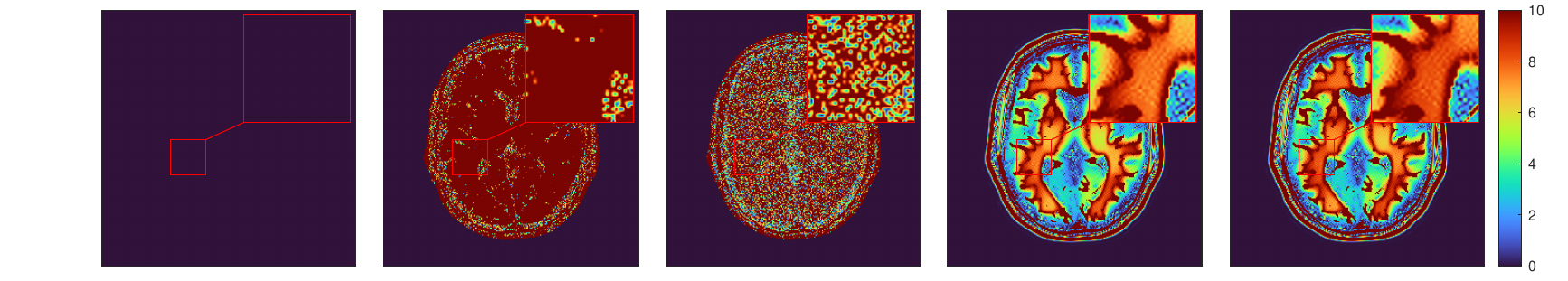} 
\end{minipage}

\caption{Comparison of the estimated parameters $T_1$ and $T_2$ (milliseconds) and proton density $\rho$ (dimensionless; relative ratio). The figure compares the reconstruction quality of the BLIP method, the Levenberg-Marquardt (LM) approach from \cite{dong2019quantitative}, and the dictionary learning (DL) approach proposed in this work for the $16 \times$ undersampling case.
} 
\end{figure}\label{fig:qMRI_comparisons_dictionary_small_noise}

\begin{figure}[p!]
 \hspace{0.8cm}
\begin{minipage}[b]{0.15\textwidth}
\centering
Ground truth 
\end{minipage}
\hspace{0.1cm}
\begin{minipage}[b]{0.15\textwidth}
\centering
 BLIP
 \end{minipage}
 \hspace{0.1cm}
 \begin{minipage}[b]{0.15\textwidth}
\centering 
LM
\end{minipage}
\hspace{0.1cm}
\begin{minipage}[b]{0.15\textwidth}
\centering
DL (one step)
\end{minipage}
\hspace{0.1cm}
\begin{minipage}[b]{0.15\textwidth}
\centering
DL (nested) 
\end{minipage}
\vspace{0.2cm}

\begin{minipage}[b]{0.03\textwidth}
\rotatebox{90}{{\hspace{7mm} $T_1$ maps }}
\end{minipage}
\begin{minipage}[b]{0.9\textwidth}
\centering
\includegraphics[width=1\textwidth, trim={1.4cm 0 0 0},clip]{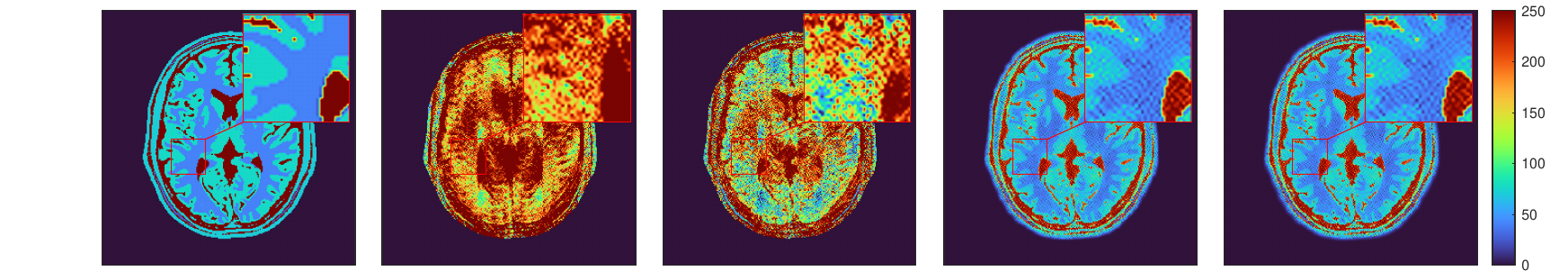} \end{minipage}

\begin{minipage}[b]{0.03\textwidth}
\rotatebox{90}{{\hspace{5mm} $T_1$ Error maps}}
\end{minipage}
\begin{minipage}[b]{0.90\textwidth}
\centering
\includegraphics[width=1\textwidth,  trim={1.4cm 0 0 0},clip]{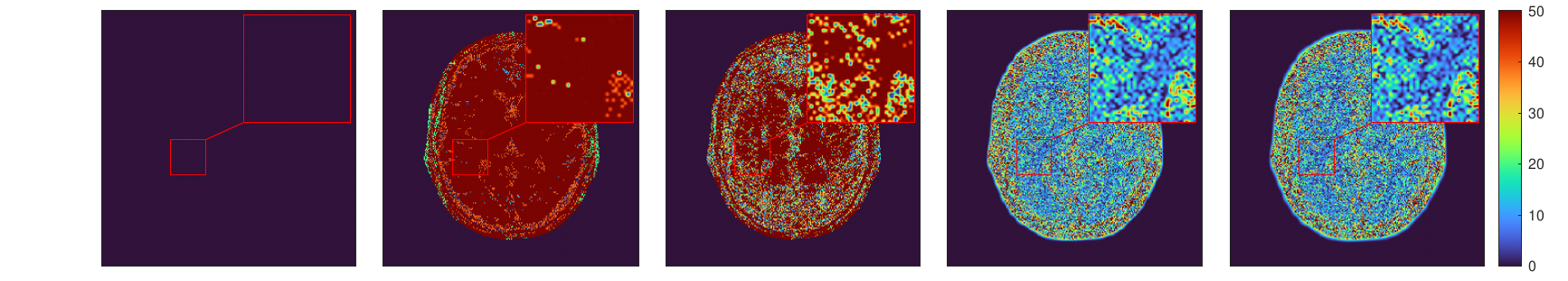}\vspace{0.5em}
\end{minipage}

\begin{minipage}[b]{0.03\textwidth}
\rotatebox{90}{{\hspace{5mm} $T_2$ maps}}
\end{minipage}
\begin{minipage}[b]{0.90\textwidth}
\centering
\includegraphics[width=1\textwidth,  trim={1.4cm 0 0 0},clip]{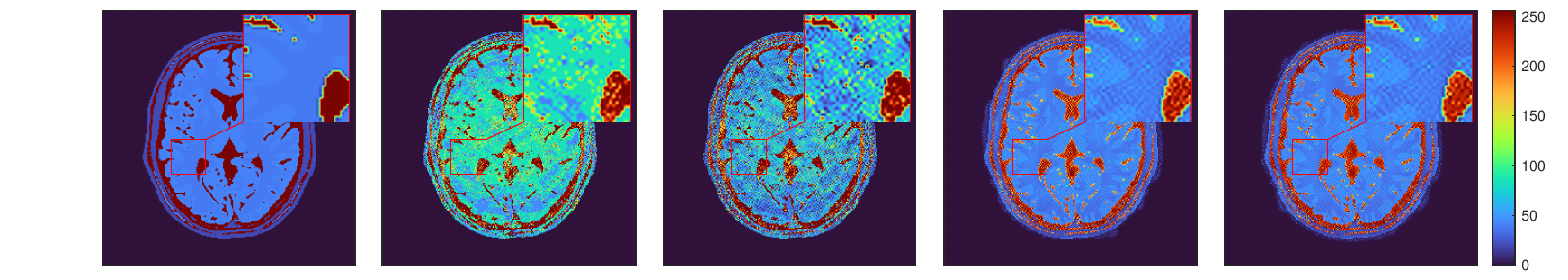} 
\end{minipage}

\begin{minipage}[b]{0.03\textwidth}
\rotatebox{90}{{\hspace{2mm} $T_2$ Error maps}}
\end{minipage}
\begin{minipage}[b]{0.90\textwidth}
\centering
\includegraphics[width=1\textwidth,  trim={1.4cm 0 0 0},clip]{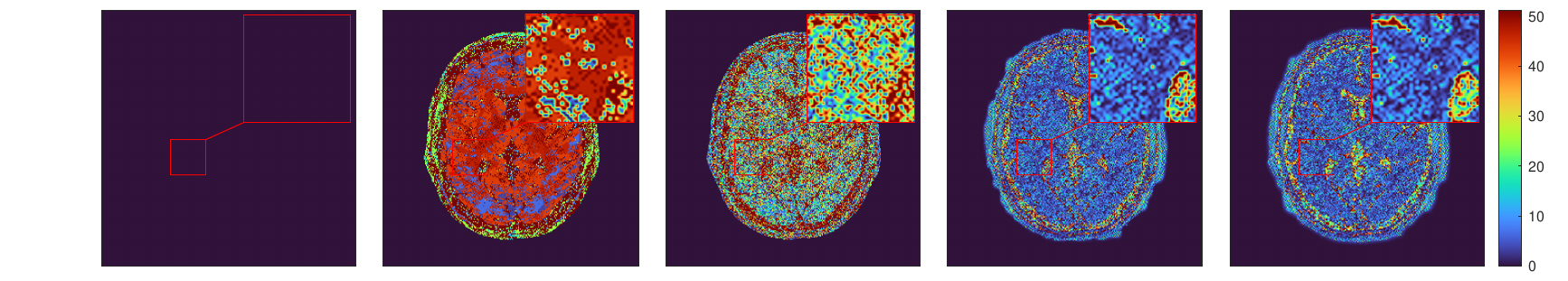} 
\end{minipage}\vspace{0.5em}

\begin{minipage}[b]{0.03\textwidth}
\rotatebox{90}{{\hspace{7mm} $\rho$ maps}}
\end{minipage}
\begin{minipage}[b]{0.90\textwidth}
\centering
\includegraphics[width=1\textwidth,  trim={1.4cm 0 0 0},clip]{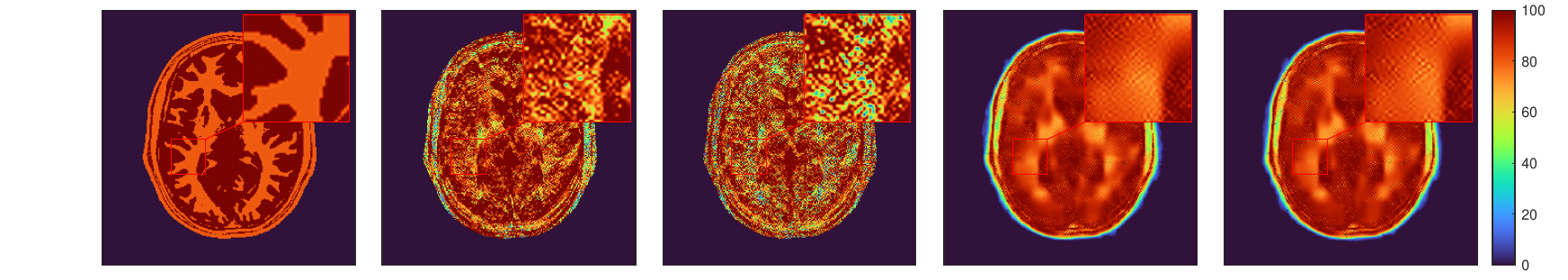} 
\end{minipage}

\begin{minipage}[b]{0.03\textwidth}
\rotatebox{90}{{\hspace{1mm} $\rho$ Error maps}}
\end{minipage}
\begin{minipage}[b]{0.90\textwidth}
\centering
\includegraphics[width=1\textwidth,  trim={1.4cm 0 0 0},clip]{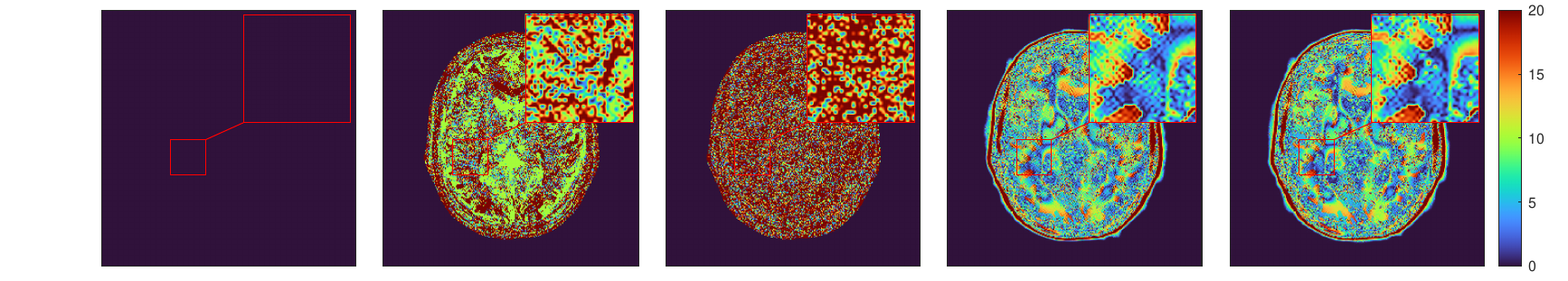} 
\end{minipage}

\caption{Comparison of the estimated parameters $T_1$ and $T_2$ (milliseconds) and proton density $\rho$ (dimensionless; relative ratio). The figure compares the reconstruction quality of the BLIP method, the LM approach from \cite{dong2019quantitative}, and the DL approach proposed in this work for the $32 \times$ undersampling case.
}
\end{figure} \label{fig:qMRI_comparisons_dictionary_large_noise}

\section{Conclusion and Outlook}

In this paper, we propose a dictionary learning-based regularization method for parameter identification in a time-discrete dynamical system framework. This nonlinear inverse problem uses dictionary learning to adapt the regularizer to the image data, with a primary focus on applications in quantitative MRI. The resulting optimization problem is large-scale, nonconvex, and nonsmooth. To tackle it, we employ a nested alternating scheme and develop a theoretical framework based on subdifferential regularity. We establish strong  local convergence to global minima and show that artificially high KL-exponents are needed to ensure fast linear convergence—highlighting a departure from classical block optimization theory, which typically assumes a global KL inequality in both directions. Block optimization in nonconvex settings remains poorly understood, especially regarding local convergence rates, which are often tied to hard-to-compute global KL-exponents. To address this, we analyze individual blocks and identify conditions under which fast strong local convergence can still be achieved. Finally, the impact of the step size $\lambda_k > 0$ remains an open question. Incorporating acceleration techniques, such as Nesterov's method or other related strategies, may further enhance convergence speed and overall algorithmic efficiency.
\\
From inverse problems point of view, many questions related to regularization theory of the proposed formulation are open. For instance, the stability of the regularized solutions, and their convergence to the model solutions under more realistic assumptions than the ones used in \cref{sec:regularization}. This naturally raises the question of how to select the parameters $\lambda, \alpha$ and $\beta$ in the numerical implementation, which is a critical issue but not addressed in the paper. This problem remains open even in the finite-dimensional case for dictionary regularized linear inverse problems, and no result is available in the infinite dimensional, nonlinear context.\\
As discussed briefly in \cref{sec:regularization}
 it remains unclear to what extent \eqref{eq:minimum_norm_problem} is capable of identifying a good dictionary and producing a high-quality reconstruction, even in the noise-free case. These  identifiability issues are challenging to analyze, even in the context of pure matrix factorization problems, and  appear to be unexplored in the context of nonlinear inverse problems. It is worth noting that blind-dictionary regularization shares several similarities with recently studied deep image priors, where instead of learning a dictionary, a neural network is trained jointly with the reconstruction process. Some recent discussions for this approach can be found in \cite{buskulic2024convergence} and the references therein.
\newpage
\ack
{The work of GD is supported by the NSFC grant No. 12471402, the NSF Innovation Research Team Project of Guangxi No. 2025GXNSFGA069001, and the NSF grant of Hunan No. 2024JJ5413.
This work of MH and CS is supported by the Deutsche Forschungsgemeinschaft (DFG, German Research Foundation) under Germany's Excellence Strategy -- The Berlin Mathematics Research Center MATH+ (EXC-2046/1, project ID: 390685689) and by the priority program Non-smooth and Complementarity-based Distributed
Parameter Systems: Simulation and Hierarchical Optimization (SPP 1962).}

\newpage

\appendix
\section*{Appendix A: Dictionary learning routine}\label{sec:appendix_dict_learning}
Let us briefly describe the dictionary learning algorithm from \cite{ravishankar2015efficient}. Note that we need to apply this algorithm for any $u_i \in \lbrace \rho, T_1,T_2 \rbrace$ separately. 
The algorithm described below aims at computing a stationary point of the objective 
\begin{equation}
    \min_{D \in O_K,C \in \mathbb{R}^{M \times K}} \frac{1}{2} \| DC - Pu \|_{F}^2 + \beta \|C\|_1 \label{eq:dictionary_learning_problem_appendix}
\end{equation}
for a given single image $u \in \mathbb{R}^{n_1 \times n_2}$. It can be seen as a denoising algorithm for the patches collected in the matrix $Pu$.
\begin{algorithm}[h]
\caption{Orthogonal dictionary learning by alternating optimization, \cite{ravishankar2015efficient}}
\begin{algorithmic}[1]
\State Get $u \in \mathbb{R}^{n_1 \times n_2}$, $(D_0,C_0) \in O_K \times \mathbb{R}^{M \times K}$, accuracy $\eta>0$, sparsity regularization parameter $\beta>0$ and step-size parameters $\lambda_D^n,\lambda_C^n > 0$, for $n \in \mathbb{N}$.
\State Set $n:=0$;
\State  Initialize with $(D_n,C_n) := (D_0,C_0)$ .\While{ {$\|D_{n} - D_{n-1} \|_F^2 + \|C_{n} - C_{n-1}\|_F^2 > \eta_k^2 $}}
	\State  Compute $D_{n+1} \in O_K$ by solving 
    \begin{equation}
        D_{n+1} \in \argmin_{D \in O_K} \frac{1}{2} \| D C_n - Pu \|_F^2 +  \frac{\lambda_{D}^n}{2} \|D - D_n\|_F^2.   \label{eq:dict_learning_PD}
    \end{equation}
    \State Compute $C_{n+1} \in \mathbb{R}^{M \times K}$ by solving
    \begin{equation}
        C_{n+1} \in \argmin_{C \in \mathbb{R}^{M \times K}} \frac{1}{2} \| D_{n+1} C - Pu \|_F^2 +  \frac{\lambda_{C}^n}{2} \|C - C_n\|_F^2 + \beta \| C \|_1.  \label{eq:dict_learning_PC}
    \end{equation}
    \State Set $n = n+1$.
\EndWhile   
\State Return $(D_{n},C_n)$.
\end{algorithmic}
\end{algorithm}\label{alg:algorithm_dictionary}
\noindent
We collect the properties of \cref{alg:algorithm_dictionary} in the following lemma.
\begin{lemma}[Properties of the dictionary learning algorithm]\label{lem:properties_dict_learning_alg} Consider \cref{alg:algorithm_dictionary}. Then the following assertions are true:
\begin{itemize}
    \item[(i)]Problem \eqref{eq:dict_learning_PD} admits the closed form solution 
       $ D_{n+1} = U V^\top$, 
    where $U \Sigma V^\top = (Pu)  C_n^\top + \lambda_D^n D_n$ is the singular value decomposition. 
    \item[(ii)] Problem \eqref{eq:dict_learning_PC} admits a closed form solution, which is given by
    \begin{align}
        C_{n+1} = \mathrm{prox}_{\beta_n \| \cdot \|_1 } \left( \frac{D_{n+1}^\top Pu + \lambda_C^n C_n}{1+\lambda_C^n} \right),\notag \quad \beta_n = \frac{\beta}{1+\lambda_C^n},
    \end{align}
    where the proximal operator $\mathrm{prox}_{\beta_n \| \cdot \|_1} : \mathbb{R}^{M \times K} \to \mathbb{R}^{M \times K}$ is defined via \emph{soft-thresholding}, i.e.,
    \begin{equation}
        [\mathrm{prox}_{\beta_n \| \cdot \|_1} (C) ]_{i,j} = 
        \begin{cases}
            C_{ij} - \beta_n &\text {if $C_{ij} \geq \beta$} \\
            0 & \text{if $-\beta_n \leq C_{ij} \leq  \beta_n$}\\
            C_{ij} + \beta_n &\text {if $C_{ij} \leq \beta$}
        \end{cases} \notag
    \end{equation}
    \item[(ii)] Let the sequences $\left(\lambda_{D}^n,\lambda_{C}^n \right)_{n \in \mathbb{N}}$ be  bounded, i.e.  $a_D \leq {\lambda}_D^n \leq b_D$ and $ a_C \leq \lambda_C^n \leq b_C$ for $0 < a_D,a_C,b_D,b_C$ and all $n \in \mathbb{N}$. Then the sequence $(D_n,C_n)_{n \in \mathbb{N}}$ that is produced by \cref{alg:algorithm_dictionary} is a descent sequence for problem \eqref{eq:dictionary_learning_problem_appendix} in the sense of  \autoref{def:descent_sequence} with parameters $\sigma_1 = \min(a_D,b_C)$ and $\sigma_2 = \max(L_C, b_C,b_D)$, where $L_C := \sup_{n \in \mathbb{N}} \| C_n \|_F < + \infty$. 
\end{itemize}
\end{lemma}
\begin{proof}
    The proof of this statement is standard. Hence, we do not display it here for the sake of brevity.
\end{proof}

\section*{Appendix B: Proofs of \autoref{sec:nonlinear_inverse_pro}.}\label{Appendix_B_proofs_sec_2}
In this section we prove the missing results in \autoref{sec:nonlinear_inverse_pro}. We start by 
\begin{proof}(of \autoref{theorem_properties_pi_map})
    For the sake of readability, we will again write $\mathbb{T} = (T_1,T_2) \in \mathbb{R} ^2$. 
    Note that (i) is a direct consequence of the corresponding properties of $E_k$ and $b_k$. Let us prove (ii). For this purpose we first investigate the properties of the time discrete magnetization $m: \mathbb{R}^2  \to \mathbb{R}^{3L}$. The Fréchet derivatives of $m$ are computed iteratively using the chain rule and the sum rule.  For directions $h,h_1,h_2 \in \mathbb{R}^3$, we obtain 
    \begin{align}
        m_{k+1}(\mathbb{T}) &= E_k(\mathbb{T}) R_k m_k(\mathbb{T}) + b_k(\mathbb{T}), \label{eq:thm1_eq1}\\
        m'_{k+1}(\mathbb{T})[h] &=  E_k'(\mathbb{T})[h] R_k \cdot m_k(\mathbb{T}) +  E_k(\mathbb{T}) R_k \cdot  m'_k(\mathbb{T})[h] +  b'_k(\mathbb{T})[h],  \label{eq:thm1_eq2}\\
        m''_{k+1}(\mathbb{T})[h_1,h_2] &=  E''_k(\mathbb{T})[h_1,h_2]  R_k \cdot m_k(\mathbb{T})  
        +  E'_k(\mathbb{T})[h_1]  R_k \cdot  m'_k(\mathbb{T})[h_2] \notag\\\ 
        &\quad +   E'_k(\mathbb{T})[h_2] R_k \cdot  m'_k(\mathbb{T})[h_1] + E_k(\mathbb{T})R_k \cdot  m''_k(\mathbb{T})[h_1,h_2] \notag\\  
        &\quad +  b''_k(\mathbb{T})[h_1,h_2].\label{eq:thm1_eq3}
    \end{align}
    Since $R_k:\mathbb{R}^3 \to \mathbb{R}^3$ are rotation matrices, we have $\|R_k\|_{\mathcal{L}(\mathbb{R}^3,\mathbb{R}^3)} = 1$. From \eqref{eq:thm1_eq1} we directly see that 
    \begin{equation}
        \|m_{k+1}(\mathbb{T}) \|_{2} \leq \| E_k(\mathbb{T}) \|_{\mathcal{L}(\mathbb{R}^{3 \times 3})} \| m_k( \mathbb{T})\|_{2}  + \| b_k(\mathbb{T})\|_{2} 
        \leq 
        C \| m_k(\mathbb{T})\|_{2}  + b. \notag
    \end{equation}
    For constants $b,C > 0$ that are uniformly bounded in $k$. Here we used the boundedness of the function $t \mapsto \exp(-1/t)$ on $\mathbb{R}$.  Hence it is easy to infer by iterating over all $k$ that $\|m(\mathbb{T})\|_{\mathbb{R}^{3 \times L}} \leq C$ for some uniform constant $C>0$ by induction. Similarly, using \eqref{eq:thm1_eq2}, we show for $\|h \|_2= 1$ that
    \begin{align*}
        \| m'_{k+1}(\mathbb{T})[h] \|_2 
        &\leq
        \| E'_k(\mathbb{T})[h] \|_2 
        C + \|E_k(\mathbb{T}) \|_{\mathcal{L}(\mathbb{R}^{3 \times 3})}  \| m'_k(\mathbb{T})\|_{\mathcal{L}(\mathbb{R}^3,\mathbb{R}^{3})} + \|b'_k(\mathbb{T})\|_{\mathcal{L}(\mathbb{R}^3,\mathbb{R}^{3})} \\
        &\leq
        C \|h\|_2 
        C + \|E_k(\mathbb{T}) \|_{\mathcal{L}(\mathbb{R}^{3 \times 3})}  \| m'_k(\mathbb{T})\|_{\mathcal{L}(\mathbb{R}^3,\mathbb{R}^{3})} + \| b'_k(\mathbb{T})\|_{\mathcal{L}(\mathbb{R}^3,\mathbb{R}^{3})}.
    \end{align*}
    From the definition of $E_k,b_k$, one easily infers  that $\|E'_k(\mathbb{T})\|_{\mathcal{L}(\mathbb{R}^3,\mathbb{R}^{3\times 3})} \leq C$, $\|E_k(\mathbb{T}) \|_{\mathcal{L}(\mathbb{R}^{3 \times 3})} \leq C$ and $\|b'_k(\mathbb{T}) \|_{\mathcal{L}(\mathbb{R}^3,\mathbb{R}^3)} \leq C$ for a generic constant $C>0$. Thus, by induction 
    $\|m'\|_{\mathcal{L}(\mathbb{R}^3,\mathbb{R}^{3 \times L})} \leq C$ by a possibly larger constant $C>0$. Exactly the same argumentation applies to show that $$\| m''(\mathbb{T})\|_{\mathcal{L}^2(\mathbb{R}^3,\mathbb{R}^{3 \times L})} \leq C$$ for all $\mathbb{T} \in \mathbb{R}^2$ by making again $C>0$ larger. This proves (ii). In order to see $(iii)$ we note that by the mean value theorem, which guarantees  the existence of  a constant $L_m>0$ such that
    \begin{align*}
        \| m(\mathbb{T}_1) - m(\mathbb{T}_2) \|_2 = \|       
        \leq L_m \| \mathbb{T}_1 - \mathbb{T}_2 \|_2, \\
        \|  m'(\mathbb{T}_1) -  m'(\mathbb{T}_2) \|_{\mathcal{L}(\mathbb{R}^3,\mathbb{R}^{3 \times L})}\leq L_m \| \mathbb{T}_1 - \mathbb{T}_2 \|_2.
    \end{align*}
    Hence, we have shown the Lipschitz continuity of \( \pi \) and \( \pi' \). Consequently we obtain:
    \begin{align*}
        \|\pi(u_1) - \pi(u_2) \|_2 
        &=
        \| \rho_1 m(\mathbb{T}_1) - \rho_2 m(\mathbb{T}_2) \|_2 \\
        &\leq | \rho_1 - \rho_2| \| m(\mathbb{T}_1)\|_2 + |\rho_2|\|  m(\mathbb{T}_1)  -  m(\mathbb{T}_2) \|_2  \\
        &\leq 
        \|u_1 - u_2 \|_2 C + b L_m \|u_1 - u_2\|_2. 
    \end{align*}
    The latter estimate proves the Lipschitz continuity of $\pi$ as desired. Similarly, we obtain for a direction $h = (h_\rho, h_{\mathbb{T}}) \in \mathbb{R}^3$ that 
    \begin{align*}
        \| ( \pi'(u_1) -  \pi'(u_2)) [h] \|_2 
        &\leq  \|h_\rho m(\mathbb{T}_1)  - h_\rho m(\mathbb{T}_2) \|_2 + 
        \|\rho_1 m'(\mathbb{T}_1)[h_{\mathbb{T}}] - \rho_2 m'(\mathbb{T}_2)[h_{\mathbb{T}}] \|_2 \\
        &\leq  |h_\rho| \|  m(\mathbb{T}_1)  - m(\mathbb{T}_2) \|_2 + 
        |\rho_1 - \rho_2| \| m'(\mathbb{T}_1)\|_2 \|h_{\mathbb{T}}\| \\
        &\quad + |\rho_2| \| (m'(\mathbb{T}_1) - m'(\mathbb{T}_2))[h_{\mathbb{T}}] \|_2 \\
        &\leq  \|h \|_2 L_m  \|u_1  - u_2 \|_2 + 
        \|u_1 - u_2\|_2 C \|h\|_2 \\
        &\quad + b L_m \| u_1 - u_2\|_2 \| h \|_2 \\
        &\leq   (L_m + C + bL_m)  \|u_1  - u_2 \|_2 \|h \|_2,
    \end{align*}
    which proves eventually the Lipschitz continuity in (iii) with Lipschitz constant ${L = (L_m + C + bL_m) }$.
\end{proof}
\noindent
Using the theorem above, we may now prove  \autoref{theorem_properties_solution_operator}.
\begin{proof}(Of \autoref{theorem_properties_solution_operator}).
We employ the theory on abstract superposition operators developed in \cite{goldberg1992nemytskij} and make use of the notation $u=(\rho,\mathbb{T}) \in \mathbb{R}^3$ as above. For (i) we have to verify the growth condition $$\|\pi(u)\|_2 \leq c + c\|u\|_2^\frac{p}{q}$$ 
for all $u \in \mathbb{R}^3$ some $c>0$. Since $\pi(u) = \rho \cdot m_{12}({\mathbb{T}})$ we directly obtain $\| \pi(u) \|_2 = | \rho | \| m_{12}({\mathbb{T}})\|_2 \leq \| u \|_2 \|m({\mathbb{T}})\|_2$ and consequently the assertion (i). For (ii) we recall \cite[Theorem 7]{goldberg1992nemytskij}. For the continuous F-differentiability of $\Pi$ we have to verify that the mapping $u \mapsto \pi'(u(\cdot))$ is continuous as a mapping from $L^p(\Omega,\mathbb{R}^3)$ to $L^r(\Omega,\mathcal{L}(\mathbb{R}^3,\mathbb{C}^L))$ for $r=pq/(p-q)$. This can be shown by verifying again the growth condition
$$\|\pi'(u) \|_{\mathcal{L}(\mathbb{R}^3,\mathbb{C}^L)} \leq c + c\| u \|_2^{\frac{p}{r}}$$
for all $u \in \mathbb{R}^3$ and some $c>0$. Since $\|  \pi'(u)[h]\|_2 = \|h_\rho m({\mathbb{T}}) + \rho m'({\mathbb{T}})[h_{\mathbb{T}}] \|_2 \leq |h_\rho| \|m({\mathbb{T}})\| + |\rho| \| m'({\mathbb{T}}) \|_2 \|h_{\mathbb{T}}\|_2$ again for a direction $h = (h_\rho,h_{\mathbb{T}}) \in \mathbb{R}^3$. By the boundedness of $m({\mathbb{T}})$ and $ m'({\mathbb{T}})$ we obtain $\| \pi'(u)[h]\|_2 \leq (c + c\|u\|_2) \|h\|_2 $ for some $c>0$ as desired. Since $p/r = (p-q)/q = 1$ for $p=4$ and $q=2$ the statement follows. In order to show (iii) we invoke theorem 9 of \cite{goldberg1992nemytskij}. Completely analogous to (ii) we have to show that $u \mapsto \pi''(u(\cdot))$ is continuous as a mapping from $L^p(\Omega,\mathbb{R}^3)$ to $L^s(\Omega,\mathcal{L}^2(\mathbb{R}^3,\mathbb{C}^L))$ for $s=pq/(p-2q)$, which again can be done by verifying the growth condition
$$\| \pi''(u) \|_{\mathcal{L}^2(\mathbb{R}^3,\mathbb{C}^L)} \leq c + c\| u \|_2^{\frac{p}{s}}$$ for all $u \in \mathbb{R}^3$ and some other $c>0$ as above. Calculating 
\begin{equation}
     \pi'' (u)[h^1,h^2] =  h_\rho^2 m'({\mathbb{T}})[h_{\mathbb{T}}^1] + h_\rho^1 m'({\mathbb{T}})[h_{\mathbb{T}}^2] + \rho m''({\mathbb{T}})[h^1_{\mathbb{T}},h^2_{\mathbb{T}}] \label{eq:growth_cond_2}
\end{equation}
yields indeed the growth condition $\|  \pi'' (u)[h^1,h^2] \|_2 \leq (c + c\|u\|_2) \|h^1\|_2 \|h^2\|_2$ for all $h^1,h^2 \in \mathbb{R}^3$ with decomposition $h^{i} = (h^{i}_\rho,h^{i}_{\mathbb{T}})$. Hence the growth-condition is satisfied for $p/s = 1$ from which we get $1 = p/s = (p-2q)/q = (p - 4)/2$ using $q=2$. Hence for $p\geq 6$ the \eqref{eq:growth_cond_2} is satisfied, and second order Frechet-differentiability holds.   
\end{proof}
\section*{Appendix C: Proofs of \autoref{sec:nested_optimization}.}\label{Appendix_B_proofs_sec_3}
\begin{proof}(Of \autoref{lemma_general_descent})
Set again $R_1 = \mathcal{I}_{U_{ad}}$. By the definition of $J$ and \autoref{fundamental_inequality} we obtain directly the following inequality
\begin{align*}
J(u, z_{k+1})   
&= f(u) + \frac{\alpha}{2}\| \nabla u \|_{L^2(\Omega)}^2 +  h(u,z_{k+1}) + R_1(u) + R(z_{k+1}) \\
 &\geq
 g(u,u_k) - \frac{L_2}{2} \|u - u_k \|^2_{H_0^1(\Omega)} + \frac{\alpha}{2}\| \nabla u \|_{L^2(\Omega)}^2 + h(u,z_{k+1}) + R_1(u) + R(z_{k+1}). 
\end{align*}
Now we invoke the convexity of $g_{\lambda_k}(\cdot, u_k)$ and the minimization property of $u_{k+1}$ to obtain: 
\begin{align*}
J(u, z_{k+1})   
&\geq 
 g_{\lambda_k}(u,u_k) - \left(\frac{\lambda_k + L_2}{2} \right) \|u - u_k \|^2_{H_0^1(\Omega)}  \\
 &\quad + \frac{\alpha}{2}\| \nabla u \|_{L^2(\Omega)}^2 + h(u,z_{k+1}) + R_1(u) + R(z_{k+1}) \\
 &\geq
 g_{\lambda_k}(u_{k+1},u_k)  - \left( \frac{\lambda_k + L}{2}\right) \|u - u_k \|^2_{H_0^1(\Omega)}  \\   
 & \quad + \frac{\alpha}{2}\| \nabla u_{k+1} \|_{L^2(\Omega)}^2 + h(u_{k+1},z_{k+1}) + R_1(u_{k+1}) + R(z_{k+1})  \qquad \qquad \text{(by minimization)}\\
 &\geq
 g(u_{k+1},u_k) + \frac{ L_2 }{2} \|u_{k+1} - u_k \|^2_{H_0^1(\Omega)}  + \left(\frac{ \lambda_k - L_2}{2} \right)\|u_{k+1} - u_k \|^2_{H_0^1(\Omega)}   \\   
 & \quad - \left(\frac{\lambda_k + L_2}{2}\right) \|u - u_k \|^2_{H_0^1(\Omega)} + \frac{\alpha}{2}\| \nabla u_{k+1} \|_{L^2(\Omega)}^2 + h(u_{k+1},z_{k+1}) + R_1(u_{k+1}) + R(z_{k+1}) \\
 &\geq
 f(u_{k+1}) + \frac{\alpha}{2}\| \nabla u_{k+1} \|_{L^2(\Omega)}^2 + \left( \frac{  \lambda_k - L_2}{2} \right) \|u_{k+1} - u_k \|^2_{H_0^1(\Omega)}   \\   
 & \quad - \left( \frac{\lambda_k + L_2}{2} \right) \|u - u_k \|^2_{H_0^1(\Omega)} + h(u_{k+1},z_{k+1}) + R_1(u_{k+1}) + R(z_{k+1}),
\end{align*}
which yields the desired estimate \eqref{eq:gen_descent_local_1}. We continue by setting $u=u_{k}$ in \eqref{eq:gen_descent_local_1} and obtain
\begin{equation}
    J(u_k,z_{k+1}) \geq J(u_{k+1},z_{k+1}) 
    + \left( \frac{\lambda_k - L_2}{2} \right) \| u_{k+1} - u_{k} \|^2_{H_0^1(\Omega)}. \notag
\end{equation}
Using the descent property of nested inner loop algorithm, \eqref{eq:descent_z} for some $\sigma_1 > 0$, we further conclude
\begin{equation}
    J(u_k,z_{k+1}) \leq J(u_k,z_{k}) -  \sigma_1 \sum_{i=1}^{n_k - 1} \|z_k^{i+1} -z_k^{i}\|^2_Z. \notag
\end{equation}
Combining the previous two inequalities yields \eqref{eq:gen_descent_local_2}.
\end{proof}

\printbibliography

\end{document}